\documentclass[11pt]{article}
\usepackage{fullpage}

\usepackage[T1]{fontenc}
\usepackage[english]{babel}

\usepackage{amsmath,amssymb,amsthm,mathtools,esint}
\usepackage{tikz-cd}
\usetikzlibrary{decorations.markings}
\usepackage{graphicx,subcaption}
\usepackage{enumerate}
\usepackage{dsfont}
\usepackage{verbatim}
\usepackage{hyperref}

\theoremstyle{plain}
\newtheorem{theorem}{Theorem}[section] 
\newtheorem{lemma}[theorem]{Lemma}
\newtheorem{proposition}[theorem]{Proposition}
\newtheorem{corollary}[theorem]{Corollary}

\newtheorem{example}[theorem]{Example}
\theoremstyle{definition}
\newtheorem{definition}[theorem]{Definition}
\newtheorem{remark}[theorem]{Remark}
\numberwithin{equation}{section}

\mathtoolsset{showonlyrefs}

\newcommand{\ZZ}{\mathbb{Z}}
\newcommand{\RR}{\mathbb{R}}
\newcommand{\CC}{\mathbb{C}}

\newcommand{\Ordo}{\mathcal{O}}
\newcommand{\EE}{\mathbb{E}}
\newcommand{\PP}{\mathbb{P}}
\renewcommand{\d}{\,\mathrm{d}}
\renewcommand{\i}{\mathrm{i}}
\newcommand{\e}{\mathrm{e}}

\DeclareMathOperator{\im}{Im}
\DeclareMathOperator{\re}{Re}

\DeclareMathOperator{\Tr}{Tr}

\DeclareMathOperator{\adj}{adj}
\DeclareMathOperator{\one}{\mathds{1}}

\DeclareMathOperator{\GUE}{GUE}

\graphicspath{ {Images/} }

\title{Marked GUE-corners process in doubly periodic dimer models}

\author{Tomas Berggren\footnote{Department of Mathematics \& statistics, University of South Florida, USA. E-mail: tberggren@usf.edu} 
\and Nedialko Bradinoff\footnote{Department of Mathematics, Royal Institute of Technology, Sweden. E-mail: nedialko@kth.se}
}

\begin{document}
\date{}

\maketitle

\begin{abstract}
We study a family of periodically weighted Aztec diamond dimer models near their turning points. We establish that, asymptotically, as $N\rightarrow\infty$, their fluctuations there, scaled by $\sqrt{N}$, are described by a marked GUE-corners process. This limiting point process is constructed by assigning a Bernoulli mark independently to each particle in a realization of the GUE-corners process. The Bernoulli parameters associated with the random marks reflect the periodicity of the model in the limit. To prove this result we use a double-contour integral representation of the inverse Kasteleyn matrix on a higher-genus Riemann surface, which is well-suited for asymptotic analysis. 
\end{abstract}

\tableofcontents

\section{Introduction}

\subsection{Preface}

\textit{Planar dimer models} take a central position in statistical mechanics and combinatorics, serving as a rare class of two-dimensional lattice models that remain amenable to exact analysis. Beginning with the seminal works of Kasteleyn \cite{Kas61} and Temperley–Fisher \cite{TF61}, which express the \emph{partition function} of a dimer model (equivalently, the enumeration of \textit{dimer covers} or \textit{perfect matchings}) in terms of determinants and Pfaffians, the subject has developed into a vibrant area of research, with sustained activity in the mathematical community over the last three decades.
Among the most accessible and influential instances are domino tilings of the \textit{Aztec diamond}, introduced in \cite{EKLP92a, EKLP92b}, where the ``arctic circle phenomenon''--the coexistence of frozen and disordered regions separated by a sharp interface known as the \textit{arctic curve}--was proved for the first time \cite{JPS98}. This model subsequently became a testing ground for general ideas in statistical mechanics and particularly dimer models.

The uniformly weighted Aztec diamond is by now a classical model, and its local and global statistics are well understood \cite{BG18, CJY15, CEP96, Joh05a}. 
 Over the last decade, new analytic and algebraic tools--combining spectral curves, algebraic geometry, and refined asymptotics--have led to substantial progress for \textit{doubly periodic} edge weights. In this richer setting, the phase diagram may include not only \textit{frozen} and \textit{rough} (\textit{liquid}) regions, but also \textit{smooth} (\textit{gaseous}) phases, reflecting the fact that the underlying spectral curve typically has higher genus \cite{KOS06}. Following the initial analysis of the \emph{two-periodic Aztec diamond}--the simplest instance in which a smooth phase appears--\cite{CJ16,CY14,DK21}, the theory has developed rapidly: for general periodic weights one now has explicit descriptions of limit shapes and arctic curves \cite{ADPZ20,BB23,BB24}, and detailed results on local \cite{BB23,BdT24} and global \cite{BN25} fluctuations, revealing new qualitative phenomena absent in the uniform case.
 
A natural question is then whether periodicity influences critical behavior, more precisely, the local edge fluctuations near the arctic curve. 
We study this question near the \textit{turning points}--the points where the arctic curve touches the boundary.

In the uniformly weighted Aztec diamond these local fluctuations 
are governed by the \textit{GUE-corners} process \cite{JN06}. The GUE-corners process is a multi-level determinantal point process defined by the eigenvalues of the principal leading sub-matrices of an infinite GUE-matrix and has by now been proved to be a universal limit at turning points in uniformly weighted lozenge tilings \cite{AG21, GP15, MP17, OR06}.
In addition, the GUE-corners process has been proved to govern the limit around turning points also in plane partitions containing more than one turning point on one side of the domain \cite{Mkr21}, as well as in the six-vertex model, \cite{Dim20, DR20, GL25}.
Very recently, the GUE-corners process was also observed in the two-periodic Aztec diamond \cite{RR25}, see the last paragraph in Section~\ref{sec:intro:main_results} for more details.

In this paper, we study the corresponding 
limit for a class of doubly periodic Aztec diamond dimer models and show that the microscopic periodic structure survives in a nontrivial way: the limiting object is a \textit{marked GUE-corners process}. Informally, the periodicity introduces intrinsic \textit{marks} (or colors) that persist under the critical scaling, and the resulting correlation functions converge to those of a marked extension of the classical GUE-corners ensemble. 
Moreover, we define, on the discrete level, a point process that converges to this limit. 
This identifies a new mechanism by which periodic microscopic data enrich critical limits.

\subsection{The Aztec diamond and an interlacing particle system}\label{sec:Aztec_interlacing_intro}

The results in this work are in the context of the \textit{Aztec diamond dimer model} (or the \textit{Aztec diamond}), a dimer model defined on the {Aztec diamond graph} $G_{Az}$. The \textit{Aztec diamond graph} of size $N$ is a bipartite graph defined on a subset of $\{0,1,2,\ldots, 2N\}^2\subset \mathbb{Z}^2$ so that there are black vertices on points with even $x$-coordinate and odd $y$-coordinate and white vertices at points with odd $x$-coordinate and even $y$-coordinate. Two vertices are adjacent if both their coordinates differ by exactly $1$, see Figure \ref{fig:Aztec_diamond_graph_4}. Depending on the vertices they are adjacent to, the edges  (and their corresponding dimers in a random outcome) are classified as North, South, East, or West. A \emph{perfect matching} or \emph{dimer cover} of $G_{Az}$ is a subgraph of $G_{Az}$ in which each vertex is incident to exactly one edge (\emph{dimer}).  

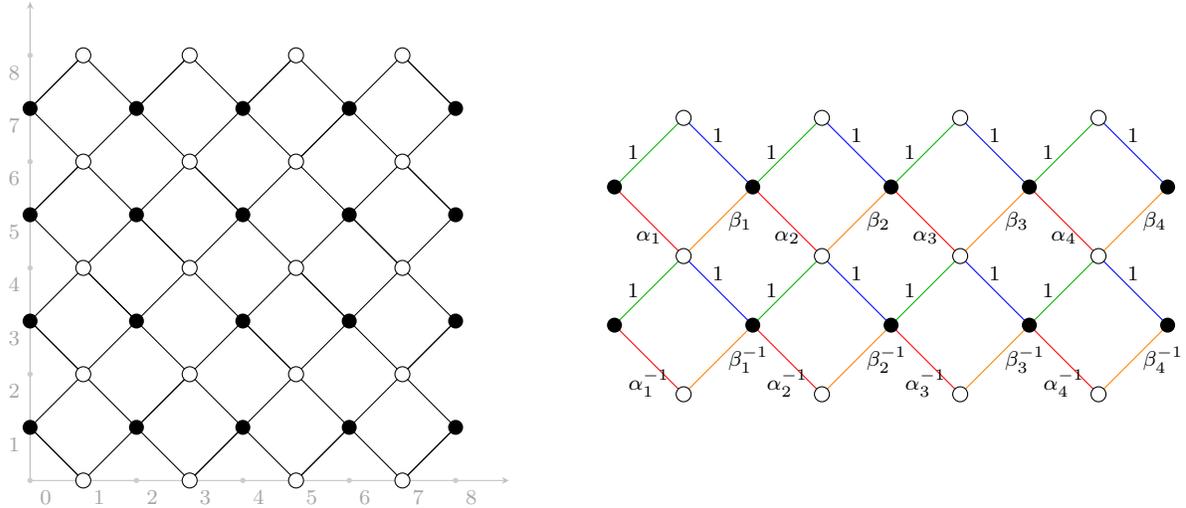
\begin{figure}
\centering
\begin{minipage}{0.48\textwidth}
\centering
\begin{tikzpicture}[scale=1, rotate=-45] 
\begin{scope}[rotate=0]
  \draw[gray!50, very thin, -{stealth[scale=0.8]}] (0,-3) -- (4.5,1.5);
  \draw[gray!50, very thin] (0,-3) -- (-4.5,1.5);
  \draw[gray!50, very thin, {stealth[scale=0.8]}-] (-4.5,1.5) -- (0,-3); 
  
  \foreach \x/\label in {0.5/1, 1/2, 1.5/3, 2/4, 2.5/5, 3/6, 3.5/7, 4/8} {
    \coordinate (pt) at (\x, \x-3); 
    \fill[gray!40] (pt) circle (1pt);
    \node[gray!70, font=\scriptsize, anchor=north west] at (pt) {$\label$};
  }
  
  \foreach \x/\label in {0.5/1, 1/2, 1.5/3, 2/4, 2.5/5, 3/6, 3.5/7, 4/8} {
    \coordinate (pt) at (-\x, \x-3); 
    \fill[gray!40] (pt) circle (1pt);
    \node[gray!70, font=\scriptsize, anchor=north east] at (pt) {$\label$};
  }
  
  \fill[gray!40] (0,-3) circle (1pt);
  \node[gray!70, font=\scriptsize, anchor=north west] at (0,-3) {$0$};
\end{scope}

\foreach \x in {0,1,2,3}
\foreach \y in {0,...,3}
{
  \draw[thin] (-3.5+\x+\y,.5-\x+\y) rectangle (-2.5+\x+\y,1.5-\x+\y);
}

\foreach \x/\y in {1/-2,2/-1,3/0,2/1,-1/2}
{
  \draw[thin] (\x+.5,\y+.5)--(\x+.5,\y+1.5);
  \draw (\x+.5,\y+.5) node[circle,draw=black,fill=white,inner sep=2pt]{};
  \draw (\x+.5,\y+1.5) node[circle,fill,inner sep=2pt]{};
}

\foreach \x/\y in {-2/1,0/1,0/3,-1/4,-2/-1}
{
  \draw[thin] (\x+.5,\y+.5)--(\x+1.5,\y+0.5);
  \draw (\x+.5,\y+.5) node[circle,draw=black,fill=white,inner sep=2pt]{};
  \draw (\x+1.5,\y+.5) node[circle,fill,inner sep=2pt]{};
}

\foreach \x/\y in {-4/0,-3/1,-2/2}
{ 
  \draw[thin] (\x+.5,\y+.5)--(\x+.5,\y+1.5);
  \draw (\x+.5,\y+.5) node[circle,fill,inner sep=2pt]{};
  \draw (\x+.5,\y+1.5) node[circle,draw=black,fill=white,inner sep=2pt]{};
}

\foreach \x/\y in {-3/-1,0/-2}
{ 
  \draw[thin] (\x+.5,\y+.5)--(\x+.5,\y+1.5);
  \draw (\x+.5,\y+.5) node[circle,fill,inner sep=2pt]{};
  \draw (\x+.5,\y+1.5) node[circle,draw=black,fill=white,inner sep=2pt]{};
}

\foreach \x/\y in {-2/-2,0/0}
{
  \draw[thin] (\x+.5,\y+.5)--(\x+1.5,\y+0.5);
  \draw (\x+.5,\y+.5) node[circle,fill,inner sep=2pt]{};
  \draw (\x+1.5,\y+.5) node[circle,draw=black,fill=white,inner sep=2pt]{};
}

\foreach \x/\y in {-1/-3,0/2,-2/0}
{
  \draw[thin] (\x+.5,\y+.5)--(\x+1.5,\y+0.5);
  \draw (\x+.5,\y+.5) node[circle,fill,inner sep=2pt]{};
  \draw (\x+1.5,\y+.5) node[circle,draw=black,fill=white,inner sep=2pt]{};
}

\end{tikzpicture}
\end{minipage}%
\hfill
\begin{minipage}{0.48\textwidth}
\centering
\begin{tikzpicture}[scale=1.3, rotate=-45]

\foreach \y in {0,...,3}
{
  \draw[thin, red] (-3.5+\y,.5+\y) -- (-2.5+\y,.5+\y); 
  \draw[thin, blue] (-3.5+\y,1.5+\y) -- (-2.5+\y,1.5+\y); 
  \draw[thin, green!70!black] (-3.5+\y,.5+\y) -- (-3.5+\y,1.5+\y); 
  \draw[thin, orange] (-2.5+\y,.5+\y) -- (-2.5+\y,1.5+\y); 
  
  \pgfmathtruncatemacro{\subscript}{\y+1}
  \node[font=\scriptsize, black] at (-3+\y, .5+\y) [below] {$\alpha_{\subscript}$}; 
  \node[font=\scriptsize, black] at (-3+\y, 1.5+\y) [above] {$1$}; 
  \node[font=\scriptsize, black] at (-3.5+\y, 1+\y) [left] {$1$}; 
  \node[font=\scriptsize, black] at (-2.5+\y, 1+\y) [right] {$\beta_{\subscript}$}; 
}

\foreach \y in {0,...,3}
{
  \draw[thin, red] (-2.5+\y,-.5+\y) -- (-1.5+\y,-.5+\y); 
  \draw[thin, blue] (-2.5+\y,.5+\y) -- (-1.5+\y,.5+\y); 
  \draw[thin, green!70!black] (-2.5+\y,-.5+\y) -- (-2.5+\y,.5+\y); 
  \draw[thin, orange] (-1.5+\y,-.5+\y) -- (-1.5+\y,.5+\y); 
  
  \pgfmathtruncatemacro{\subscript}{\y+1}
  \node[font=\scriptsize, black] at (-2+\y, -.5+\y) [below] {$\alpha_{\subscript}^{-1}$}; 
  \node[font=\scriptsize, black] at (-2+\y, .5+\y) [above] {$1$}; 
  \node[font=\scriptsize, black] at (-2.5+\y, 0+\y) [left] {$1$}; 
  \node[font=\scriptsize, black] at (-1.5+\y, 0+\y) [right] {$\beta_{\subscript}^{-1}$}; 
}

\foreach \x/\y in {3/0,2/1,-1/2}
{
  \pgfmathparse{(\y+0.5) - (\x+0.5) > 0}
  \ifdim\pgfmathresult pt>0pt
    \draw (\x+.5,\y+.5) node[circle,draw=black,fill=white,inner sep=2pt]{};
  \fi
  \pgfmathparse{(\y+1.5) - (\x+0.5) > 0}
  \ifdim\pgfmathresult pt>0pt
    \draw (\x+.5,\y+1.5) node[circle,fill,inner sep=2pt]{};
  \fi
}

\foreach \x/\y in {-2/1,0/1,0/3,-1/4,-2/-1}
{
  \pgfmathparse{(\y+0.5) - (\x+0.5) > 0}
  \ifdim\pgfmathresult pt>0pt
    \draw (\x+.5,\y+.5) node[circle,draw=black,fill=white,inner sep=2pt]{};
  \fi
  \pgfmathparse{(\y+0.5) - (\x+1.5) > 0}
  \ifdim\pgfmathresult pt>0pt
    \draw (\x+1.5,\y+.5) node[circle,fill,inner sep=2pt]{};
  \fi
}

\foreach \x/\y in {-4/0,-3/1,-2/2}
{ 
  \pgfmathparse{(\y+0.5) - (\x+0.5) > 0}
  \ifdim\pgfmathresult pt>0pt
    \draw (\x+.5,\y+.5) node[circle,fill,inner sep=2pt]{};
  \fi
  \pgfmathparse{(\y+1.5) - (\x+0.5) > 0}
  \ifdim\pgfmathresult pt>0pt
    \draw (\x+.5,\y+1.5) node[circle,draw=black,fill=white,inner sep=2pt]{};
  \fi
}

\foreach \x/\y in {-3/-1}
{ 
  \pgfmathparse{(\y+0.5) - (\x+0.5) > 0}
  \ifdim\pgfmathresult pt>0pt
    \draw (\x+.5,\y+.5) node[circle,fill,inner sep=2pt]{};
  \fi
  \pgfmathparse{(\y+1.5) - (\x+0.5) > 0}
  \ifdim\pgfmathresult pt>0pt
    \draw (\x+.5,\y+1.5) node[circle,draw=black,fill=white,inner sep=2pt]{};
  \fi
}

\foreach \x/\y in {0/2,-2/0}
{
  \pgfmathparse{(\y+0.5) - (\x+0.5) > 0}
  \ifdim\pgfmathresult pt>0pt
    \draw (\x+.5,\y+.5) node[circle,fill,inner sep=2pt]{};
  \fi
  \pgfmathparse{(\y+0.5) - (\x+1.5) > 0}
  \ifdim\pgfmathresult pt>0pt
    \draw (\x+1.5,\y+.5) node[circle,draw=black,fill=white,inner sep=2pt]{};
  \fi
}

\end{tikzpicture}
\end{minipage}
\caption{The Aztec diamond graph of size 4 on the left and the fundamental domain of the $2\times 4$ periodically weighted Aztec diamond dimer model on the right. North, East, South, and West edges are colored blue, orange, red, and green respectively.}\label{fig:Aztec_diamond_graph_4} 
\end{figure}
One introduces a weight to each edge and the associated dimer model on the graph is produced by selecting a random perfect matching on the graph 
 with probability proportional to the product of the weights of the edges in that matching. 
  In this work the weights put on the edges are selected periodically with period $\ell$ in the $x$ coordinate and period $2$ in the $y$ coordinate. We introduce $2\ell$ parameters, $\alpha_1,\ldots, \alpha_{\ell}, \beta_1,\ldots, \beta_{\ell}>0$ and distribute them on the \textit{fundamental domain} of size $2\times \ell$ as in Figure \ref{fig:Aztec_diamond_graph_4}, see \eqref{eqn:edge_weights} below for a precise definition.   
We impose the additional constraint that~$\prod_{i=1}^\ell \alpha_i=\prod_{i=1}^\ell \beta_i$ and thus are in the setting studied in \cite{Ber21}.
   For clarity we restrict our attention to the Aztec diamond of size $2\ell N$.

One way to describe a given dimer cover is to describe which vertices are incident to a South or West dimer, and, following \cite{Joh05a}, we place black/white particles at such black/white vertices.
If we restrict our attention to either only the white particles or only the black particles we get an interlacing particle system. Here we consider particles with the same $x$ coordinate to lie on the same level and we enumerate the levels from the right $(x=2\ell N)$ to the left $(x=0)$. This interlacing particle system is set to capture the local fluctuations around the turning point with coordinate $x=2\ell N$. 
The two interlacing particle systems corresponding to the white and black vertices are closely related, 
 see Figure \ref{fig:corners_aztec}. 
Proceeding with the convention in \cite{Joh05a} we restrict our attention to black vertices adjacent to a South or West edge and study the point process described by them, $(u_{s}^t)_{1\leq s\leq t}$.  %

\begin{figure}
\begin{center}
\includegraphics[scale=.35]{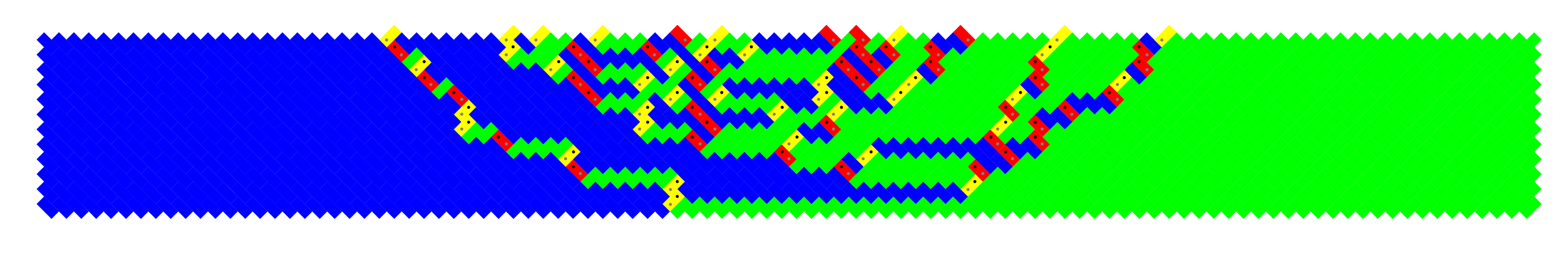}
\end{center}
\caption{The two interlacing particle systems near the turning point, in grey and black, defined from South (yellow) and West (red) edges from a random outcome of the Aztec diamond; The picture is rotated so that the first particle is at the bottom instead of to the right. Dimers in the picture are illustrated as dominoes.} \label{fig:corners_aztec}
\end{figure}

It turns out significant for the studied process in this work that the two periodicity of the weights along the $y$ axis gives rise to a natural further coloring of these particles;  we color them \textbf{red} if their $y$ coordinate is even and \textbf{cyan} if it is odd, see prescribed colors in Figure \ref{fig:interlacing}.
Enumerated from the right to the left the colored particles describe a (colored) interlacing particle system, $(u_{s,j}^t)_{1\leq s\leq t}$, where $j\in\{0,1\}$ indicates the color associated to each particle, $t$ is its level, and $s$ is its relative position with respect to the particles on the same level.

\subsection{The GUE-corners process}
The \textit{ GUE-corners process}  (also known as the \textit{GUE minor process}), $\PP_{\operatorname{GUE}}$, is a \textit{determinantal point process} on $\Lambda=\ZZ_{>0}\times \mathbb{R}$.
One way in which it can be constructed is by taking a random $\ZZ_{>0}\times \ZZ_{>0}$ hermitian matrix, $X=(X_{ij})_{i,j\in\ZZ_{>0}}$ with $X_{ij}=\overline{X_{ji}}$ such that the entries above and on the main diagonal $(i=j)$ are independent random variables with
$$X_{ij}\sim \frac{1}{\sqrt{2}}( N(0,1)+i N(0,1))\quad \text{for } i>j,\quad \text{and}\quad X_{ii}\sim N(0,1).$$
Then for $t\in\ZZ_{>0}$ the top left $t\times t$ corner of  $X$, $X^{(t)}=(X_{ij})_{1\leq i,j \leq t}$ is distributed as a $\GUE(t)$ matrix with ordered eigenvalues 
$\xi_1^t\leq \xi_2^t\leq\ldots \leq \xi_t^t,$
and the matrix $X$ couples the eigenvalues on different levels so that they interlace, namely 
 $\xi_s^{t+1}\leq \xi_s^t\leq \xi_{s+1}^{t+1}$ and the process given by 
$\left(\xi_s^{t} \right)_{1\leq s \leq t,  t\in{\ZZ_{>0}}}$ defines the GUE-corners process. That this process is well-defined is explained, for instance, in \cite{JN06}. 

\subsection{The marked GUE-corners process}
A defining feature in this work is the emergence of a \emph{marked point process} in the limit of the interlacing particle system as $N\rightarrow\infty$; the limiting point process ``remembers'' the colors of the points coming from the $2$ periodicity in the vertical direction.

For a measurable function $\theta:\Lambda\rightarrow [0,1]$ we call the \textit{marked GUE-corners} process, $\PP_{\operatorname{GUE}}^\theta$, the point process on ~$\Lambda_{\{0,1\}}=\Lambda \times \{0,1\}$ constructed out of the GUE-corners process on $\Lambda$ by independently assigning to each point a binary value; a mark $m\sim \operatorname{Bernoulli}(\theta(t,\mu))$, where~$(t,\mu)\in \Lambda$ is the location of the point. We denote these marked particles by~$(\xi_{s,j}^t)_{1\leq s\leq t}$, where $j\in \{0,1\}$ represents the marking.
It was proved in~\cite{CG23} -- for a general determinantal point process under very mild assumptions -- that the point process~$\PP_{\operatorname{GUE}}^\theta$ is a determinantal point process with correlation kernel
\begin{equation}\label{eq:intro:marked_kernel}
K_{\operatorname{GUE}}^\theta(t_1,\mu_1,j_1;t_2,\mu_2,j_2)=\left(\theta(t_1,\mu_1)\delta_{1j_1}+(1-\theta(t_1,\mu_1))\delta_{0j_1}\right)K_{\operatorname{GUE}}(t_1,\mu_1;t_2,\mu_2),
\end{equation}
where~$\delta_{jj'}$ is the Kronecker delta function. 

Marked point processes were introduced and studied in \cite{CG23}, with the interpretation that a particle is observed or not in a measurement depending on the value of the mark. The purpose of that paper was to study certain conditional probabilities of these marked processes. This was later used to study certain deformations of biorthogonal ensembles \cite{CS25}. 

\subsection{Main results}\label{sec:intro:main_results}
For the choice of weights considered in this paper, there are four \emph{turning points}--points where the \emph{arctic curve} touches the boundary of the Aztec diamond--one on each side of the Aztec diamond. We focus in this work on the right-most turning point with coordinates $(2\ell N, 2 \tau N)$ for some $\tau>0$ (which we determine). The fluctuations around this turning point are captured by the particle system~$(u_{s}^t)_{1\leq s\leq t\leq 2\ell N}$. 
The unique particle on the first level is distributed around the right-most turning point with fluctuations of order $\sqrt{N}$, similarly to the uniform case. Thus we center the vertical coordinates of the studied process at $2\tau N$ and rescale by $\sqrt{N}$, while we do not rescale the horizontal coordinate. The new coordinates are $(t,\mu)\in \Lambda=\ZZ_{>0}\times \RR$. The process constructed in this fashion in the case of the Aztec diamond with uniform weights is classically known to converge to the celebrated GUE-corners process \cite{JN06}. 
In this work we prove that the corresponding limit in the $2\times \ell$ periodic setting is affected by the 2-periodicity in the vertical direction. 

The particle system~$(u_{s}^t)_{1\leq s\leq t\leq 2\ell N}$ is a determinantal point process with correlation kernel~$K_{\operatorname{Int}}$,
see Theorem~\ref{thm:correlation_kernel} below. We study its fluctuations around the turning point via this correlation kernel. The coordinates of the correlation kernel are naturally expressed by $(\ell x+i,2y+j)$ for $i=0,\dots,\ell-1$ and $j=0,1$, and in those coordinates, the coordinate of the black vertex associated with~$u_s^t$ is~$(\ell x+i,2y+j)=(2\ell N-t,u_s^t)$. The centering at the turning point and scaling discussed above, is then given by ~$y=\lfloor N\tau +\sqrt{N}\mu\rfloor$.

We have the following result for the limiting correlation kernel of the process under the described scaling.

\begin{theorem}[Theorem \ref{thm:limit_correlation_function}]\label{thm:0}
Let~$K_{\operatorname{Int}}$ be the correlation kernel of $(u_{s}^t)_{1\leq s\leq t\leq 2\ell N}$ and suppose~$\nu:\ZZ_{>0}\times \{0,1\} \rightarrow \mathbb{R}$ is given by
\begin{equation}
 \nu(t,j)=
\begin{cases}
 \frac{1}{1+\alpha_{\ell+1-t}\beta_{\ell-t}^{-1}} \quad \text{if } j=0, \\
 \frac{\alpha_{\ell+1-t}\beta_{\ell-t}^{-1}}{1+\alpha_{\ell+1-t}\beta_{\ell-t}^{-1}} \quad \text{if } j=1,
\end{cases} 
\end{equation}
where $\alpha_{\ell-t}$ and $\beta_{\ell-t}$ are the edge weights of the model (see Figure~\ref{fig:Aztec_diamond_graph_4} and later \eqref{eqn:edge_weights}).
Let the gauge function~$g:\{0,1,\ldots, 2\ell N\}^2\rightarrow \mathbb{R}$ be defined by~\eqref{eq:gauge_function} and let~$\tau, \sigma>0$ be (explicitly) given by~\eqref{eq:turning_point} and~\eqref{eq:sigma_square} below.
%
%
Set~$\ell x_k+i_k=2\ell N-t_k$ and~$y_k=\lfloor N\tau +\sqrt{N}\mu_k\rfloor$, for $k=1,2$. Then, for $\mu_1\neq \mu_2$,
\begin{multline}
\lim_{N\to \infty}\frac{g(\ell x_1+i_1,2y_1+j_1)}{g(\ell x_2+i_2,2y_2+j_2)}N^{\frac{1}{2}}K_{\operatorname{Int}}\left(\ell x_1+i_1,2y_1+j_1;\ell x_2+i_2,2y_2+j_2\right) \\
=\nu(t_2,j_2)\sigma^{-1}K_{\operatorname{GUE}}(t_1,\sigma^{-1}\mu_1;t_2,\sigma^{-1}\mu_2), 
\end{multline}
%
%
where
\begin{equation}
K_{\operatorname{GUE}}(t_1,\mu_1;t_2,\mu_2)=\frac{1}{(2\pi\i)^2}\int_{\gamma_s'}\int_{\gamma_l'}\e^{\frac{1}{2}(z_2^2-z_1^2)}\e^{\mu_1 z_1-\mu_2z_2}\frac{z_2^{t_2}}{z_1^{t_1}}\frac{\d z_1\d z_2}{z_2-z_1},
\end{equation}
 is the correlation kernel of the GUE-corners process. The sequence on the left-hand side is uniformly bounded in $N$
 on compact subsets of 
$\Lambda_{\{0,1\}}^2$ with respect to the embedding given by $(t_p, \mu_p, j_p)\in \Lambda_{\{0,1\}}$, for $p=1,2$. See Theorem~\ref{thm:limit_correlation_function} for a definition of the curves~$\gamma_s'$ and~$\gamma_{\ell}'$.
 \end{theorem}

The correlation kernel $K_{\operatorname{GUE}}$ as presented here was obtained in \cite{JN06, OR06}. 

The factor $\nu(t_2, j_2)$ in the limit changes the limiting kernel non-trivially and cannot be removed via a gauge. In fact the limiting function is the correlation kernel of a 
marked GUE-corners process instead of the GUE-corners process. 
Indeed, if 
\begin{equation}\label{eq:intro:theta}
\theta(t,\mu)=\theta(t)=
\frac{\alpha_{\ell+1-t}\beta_{\ell-t}^{-1}}{1+\alpha_{\ell+1-t}\beta_{\ell-t}^{-1}},
\end{equation}
then 
\begin{equation}
\nu(t,j)=\theta(t)\delta_{1j}+(1-\theta(t))\delta_{0j}
\end{equation}
is the pre-factor in \eqref{eq:intro:marked_kernel}.
By studying the colored interlacing system $(u_{s,j}^t)_{1\leq s\leq t}$, we can interpret the convergence of the correlation functions in Theorem \ref{thm:0} as a convergence of processes.

\begin{theorem}[Corollary \ref{cor:week_converence}]\label{thm:1}
Let~$(u_{s,j}^t)_{1\leq s\leq t\leq 2\ell N}$ be the colored interlacing particle system described above.
Then there are (explicit) $\tau, \sigma>0$ (see later ~\eqref{eq:turning_point} and~\eqref{eq:sigma_square}) 
 such that the turning point has coordinates $(2\ell N, 2\tau N)$ and 
\begin{equation}
\left(  \frac{u_{s,j}^{t}-2 N\tau }{2\sigma\sqrt{N}}\right)_{1\leq s\leq t}\to \left(\xi_{s,j}^t\right)_{1\leq s\leq t},
\end{equation}
in the sense of weak convergence as~$N\to \infty$, and~$(\xi_{s,j}^t)_{1\leq s\leq t}$ is the marked GUE-corners process with an underlying marking function~$\theta$ given in \eqref{eq:intro:theta}.
\end{theorem}

The above theorems show that the discreteness of the fundamental domain persists in the scaling limit, although the horizontal and vertical periodicities are encoded in different ways. The horizontal $\ell$-periodicity is retained through the parameter $\theta$, which depends on $t$--more precisely, on the edge weights in the ``$\ell+1-t$ column'' of the fundamental domain. By contrast, the vertical $2$-periodicity is carried by the marks $j\in\{0,1\}$. We conjecture that if the model instead is $k$-periodic in the vertical direction, then the limit will be a \textit{$k$-marked} GUE-corners process; to each outcome of the point process one assigns independently a random integer between $0$ and $k-1$.

It is interesting--and somewhat surprising--that, even though the $y$-coordinate is rescaled and converges to a continuous coordinate $\mu$, a trace of the microscopic structure survives via these marks. Moreover, once the discreteness persists in this form, it is not obvious a priori that the contributions associated with the two marks should decouple, i.e., that the marks should behave independently.

Marked processes are closely related to \emph{thinned processes}. A thinned process is defined by independently deleting each particle of a given process with some probability. These processes were introduced in the random matrix theory literature by \cite{BP06, BP04} with the motivation behind their study being to model situations in which the possibility of failed detection/measurement of particles may occur. They have since been studied in many contexts within the random matrix theory literature, for instance, in \cite{BD17, Bot17, CC17}. 
In the dimer literature a thinned process was observed in \cite{CJY15}, as the authors studied the process defined by viewing only the south dimers as particles in the \emph{biased Aztec diamond}.

Theorem \ref{thm:1} implies that if we study the process described by only the red or only the cyan particles we get a \textit{thinned} GUE-corners process as a limiting process; that is, a GUE-corners process in which each particle is deleted independently with probability $1-\theta$ or $\theta$, respectively.

\begin{corollary}[Corollary \ref{cor:thinned}]\label{maincor:1} 
The restriction of the point process~$(u_{s,j}^t)_{1\leq s\leq t\leq 2\ell N}$ to points with~$j=1$ converges, under the same scaling and in the same sense as in Theorem~\ref{thm:1}, to a thinned GUE–corners process with deletion probability~$\theta$.
\end{corollary}

We stress that the limit in Theorem~\ref{thm:1} is subtle. If one identifies the two colors (equivalently, forgets the marks), one recovers the classical GUE-corners process (see Remark~\ref{rmk:forget_marks} below). In our framework, however, the marks arise naturally, as seen from Theorem~\ref{thm:0}. 
While this manuscript was being prepared, the paper \cite{RR25} appeared. In that work, the authors analyse the same limiting regime in the special case of the two-periodic Aztec diamond (in our notation $\ell=2$, and $\alpha_1^{-1}=\beta_1^{-1}=\alpha_2=\beta_2=a$ for some $a>0$) using a completely different method. However, it appears that their method does not capture this refined limit; instead, the classical (unmarked) GUE-corners process is obtained.

\subsection{Outlook on further developments} 
Let us briefly specialize to the two-periodic Aztec diamond. This model depends on a single parameter $a$; in our notation $\ell=2$ and
$\alpha_1^{-1}=\beta_1^{-1}=\alpha_2=\beta_2=a$.
In this case the marking function $\theta$ simplifies to
\begin{equation}
\theta(t,\mu)=
\begin{cases}
\frac{a^2}{1+a^2}, & t \quad \text{odd,} \\
\frac{a^{-2}}{1+a^{-2}}, & t \quad \text{even.} \\
\end{cases}
\end{equation} 
Thus, as $a\to 0$ or $a\to\infty$, we have $\theta(t,\mu)\to 0$ or $1$ (depending on the parity of $t$), so the marks become asymptotically deterministic. In particular, after forgetting the now-trivial marks, one recovers the classical GUE-corners process in this limit.

At the same time, the global geometry changes as $a\to 0$ or $a\to\infty$: the rough region shrinks away, while the smooth phase turns into a tilted square that meets the frozen regions. In particular, the smooth phase reaches the turning point, and the local configuration there appears to freeze.

These observations suggest a natural two-parameter asymptotic question: let the weights vary with the size $N$ and consider a scaling limit in which $a=a(N)\to\infty$. If $a(N)$ diverges sufficiently slowly compared with $N$, one should see the GUE-corners process. At the opposite extreme, if $a(N)$ diverges very rapidly, the particles should freeze and the configuration becomes essentially deterministic. 

Understanding the intermediate regime between these extremes would be particularly interesting. It is tempting to speculate that an appropriate interpolation limit might be related to the $\infty$-corners process studied in \cite{GK24,GM20}. While an intermediate scaling regime has been analyzed along the frozen-smooth interface in \cite{JM23}, the corresponding behavior at the turning point appears to remain open.

One can attempt to extend the above discussion to more general periodic weights. In \cite{BB24}, a temperature parameter is introduced and the corresponding zero-temperature limit is analyzed. In that limit, the limit shape converges to a piecewise linear profile, and--much as in the two-periodic Aztec diamond--the rough region vanishes and the smooth phase expands so as to reach the turning points. This raises a natural question: in such a joint scaling regime (with $N\to\infty$ and temperature tending to zero), to what extent is the turning-point behavior universal, and to what extent does it depend on the specific choice of weights?

\paragraph{Outline of the paper.}
The necessary background for this work, including a double-contour integral expression for the inverse Kasteleyn matrix obtained in \cite{Ber21} (see Theorem \ref{thm:correlation_kernel}), is introduced in Section \ref{section:prelim}. We state and prove our main results in Section \ref{sec:main}, deferring the technical analysis of the double-contour integral to Section \ref{sec:steepest}. The explicit computation of the parameters $\tau$ and $\sigma$ appearing in Theorem \ref{thm:1} is then given in Section \ref{sec:ty_sigma}.

\paragraph{Acknowledgements}
We are grateful to Maurice Duits for all the valuable discussions throughout this project.
Nedialko Bradinoff was supported by the European Research Council (ERC), Grant Agreement No. 101002013.

\section{Preliminaries}\label{section:prelim}
We discuss in more detail the necessary background and context of this work. 
\subsection{The Aztec diamond dimer model}
In this section we define the Aztec diamond dimer model.

The Aztec diamond graph~$G_\text{Az}=(\mathcal B_\text{Az}, \mathcal W_\text{Az}, \mathcal E_\text{Az})$ of size~$N$ is a bibartite graph defined from a subset of the tilted square lattice. More precisely, set
\begin{equation}
\mathcal B_\text{Az}=\{(2i,2j+1): i\in \{0,\dots,N\}, \, j\in\{0,\dots,N-1\}\}, 
\end{equation}
and 
\begin{equation}
\mathcal W_\text{Az}=\{(2i+1,2j): i\in \{0,\dots,N-1\}, \, j\in\{0,\dots,N\}\}.
\end{equation}
The set of vertices of~$G_\text{Az}$ is the union~$\mathcal B_\text{Az}\cup \mathcal W_\text{Az}$, where we call~$\mathcal B_\text{Az}$ \emph{black vertices} and~$\mathcal W_\text{Az}$ white vertices. The edges~$\mathcal E_\text{Az}$ consists of the union of four type of edges,  \emph{north}, \emph{east}, \emph{south}, and \emph{west} edges:
\begin{align*}
\text{north}& =\{((2i-1,2j),(2i,2j-1)):(2i-1,2j)\in \mathcal W_{\text{Az}},(2i,2j-1))\in \mathcal B_{\text{Az}}\}, \\
\text{east}& =\{((2i-1,2j),(2i,2j+1)):(2i-1,2j)\in \mathcal W_{\text{Az}},(2i,2j+1))\in \mathcal B_{\text{Az}}\}, \\
\text{south}& =\{((2i+1,2j),(2i,2j+1)):(2i+1,2j)\in \mathcal W_{\text{Az}},(2i,2j+1))\in \mathcal B_{\text{Az}}\}, \\
\text{west}& =\{((2i+1,2j),(2i,2j-1)):(2i+1,2j)\in \mathcal W_{\text{Az}},(2i,2j-1))\in \mathcal B_{\text{Az}}\}.
\end{align*} 

To introduce the dimer model, we introduce edge weights~$\alpha_i,\beta_i>0$ for~$i\in \ZZ$, and define the weight function~$w:\mathcal E_\text{Az}\to \RR_{>0}$ as follows: Let~$e_\text{south}=((2i+1,2j),(2i,2j+1))\in \mathcal E_\text{Az}$ and~$e_\text{east}=((2i-1,2j),(2i,2j+1))\in \mathcal E_\text{Az}$ be a south and east edge. For~$e\in \mathcal E_\text{Az}$, we set
\begin{equation}\label{eqn:edge_weights}
w(e)=
\begin{cases}
\alpha_{i+1}^{-1}, & e=e_\text{south}, \, j \text{ even}, \\
\alpha_{i+1}, & e=e_\text{south}, \, j \text{ odd}, \\
\beta_i^{-1}, & e=e_\text{east}, \, j \text{ even}, \\
\beta_i, & e=e_\text{east}, \, j \text{ odd}, \\
1, & \text{otherwise}.
\end{cases}
\end{equation}
By definition, the weight function~$w$ is~$2$-periodic in the vertical direction, that is, $2$-periodic in the $j$ coordinate. We will also assume it is periodic in the horizontal direction, that is, periodic in the $i$ coordinate, let say with period~$\ell$. We therefore assume that~$\alpha_{i+\ell}=\alpha_i$ and~$\beta_{i+\ell}=\beta_i$ for all~$i$. In addition, following \cite{Ber21}, we will assume that~$\prod_{i=1}^\ell \alpha_i=\prod_{i=1}^\ell \beta_i$. For simplicity, we also impose the generic assumption that the associated spectral curve has maximum genus, that is,~$g=\ell-1$, see Section~\ref{sec:kernel}.

Given the graph~$G_\text{Az}$ and the weight function~$w$, we define a probability measure on all \emph{dimer covers} of the Aztec diamond. Recall that a dimer cover~$\mathcal D$ of~$G_\text{Az}$ is a subset of~$\mathcal E_\text{Az}$ such that each vertex in~$\mathcal B_\text{az}$ and~$\mathcal W_\text{Az}$ is adjacent to exactly one edge in~$\mathcal D$. An element in~$\mathcal D$ is called a \emph{dimer}. We consider here the Boltzmann measure,
\begin{equation}\label{eq:dimer_measure}
\PP_{\operatorname{dimer}}(\mathcal D)=\frac{1}{Z}\prod_{d\in \mathcal D} w(d),
\end{equation}
where~$Z=\sum_{\mathcal D}\prod_{d\in \mathcal D} w(d)$ is called the \emph{partition function} and the sum is over all possible dimer covers~$\mathcal D$ of~$G_\text{Az}$.

The edge inclusion probability, that is, the probability to see a specific set of dimers in a dimer cover can be expressed in terms of the \emph{Kasteleyn matrix}~$K_{\operatorname{Az}}$ and its inverse. To define the Kasteleyn matrix, we introduce the \emph{Kasteleyn sign}~$\sigma:\mathcal E_\text{Az}\to \{-1,1\}$ by letting~$\sigma(e)=-1$ if~$e$ is a north edge and~$\sigma(e)=1$ otherwise. In general, the Kasteleyn sign can be chosen to take values on the unit circle with the condition that the alternating product around each face of the graph is~$(-1)^{k+1}$, where~$2k$ is the number of edges around the face. The Kasteleyn matrix~$K_{\operatorname{Az}}$ is essentially a weighted adjacency matrix:~$K_{\operatorname{Az}}:\CC^{\mathcal B_\text{Az}}\to \CC^{\mathcal W_\text{Az}}$, defined as
\begin{equation}\label{eq:kasteleyn}
K_{\operatorname{Az}}(\mathrm w,\mathrm b)=\one_{\mathrm w\mathrm b\in \mathcal E_{\text{Az}}}\sigma(\mathrm w\mathrm b)w(\mathrm w\mathrm b),
\end{equation}
where~$\mathrm w\in \mathcal W_\text{Az}$ and~$\mathrm b\in \mathcal B_\text{Az}$. The expression on the right-hand side should be interpreted as zero if~$\mathrm w\mathrm b$ is not an edge. The Kasteleyn signs are defined so that~$Z=|\det K_{\text{Az}}|$~\cite{Kas61, TF61}, and it follows~\cite{Ken97} that for~$k$ edges~$\mathrm b_p\mathrm w_p\in \mathcal E_\text{Az}$,~$p=1,\dots,k$,
\begin{equation}\label{eq:kenyons_formula}
\PP_{\operatorname{dimer}}\left(\mathrm b_p \mathrm w_p\in \mathcal D\text{ for all }p=1,\dots,k\right)=\det\left(K_\text{Az}(\mathrm w_{p'},\mathrm b_{p'})K_\text{Az}^{-1}(\mathrm b_p,\mathrm w_{p'})\right)_{p,p'=1}^k.
\end{equation}
In general it is a hard problem to obtain a suitable expression for the inverse Kasteleyn matrix. In the setting considered here, the inverse Kasteleyn matrix, or rather, the correlation kernel for the associated non-intersecting path model, was obtained in~\cite{Ber21} using the technique developed in~\cite{BD19}.

\subsection{An interlacing particle system}\label{sec:interlacing_particles}

The dimer model discussed in the previous section is a \emph{determinantal point process}, where the dimers are viewed as points and the \emph{correlation kernel} can be expressed in terms of the \emph{Kasteleyn matrix} and its inverse. In this paper, we are interested in a closely related point process. Namely, given a dimer cover~$\mathcal D$ of~$G_\text{Az}$, we define a point configuration by putting a point, or particle, at a black vertex~$\mathrm b\in \mathcal B_\text{Az}$ if~$\mathrm b$ is adjacent to a south or west edge in~$\mathcal D$. See Figure~\ref{fig:interlacing}. The probability measure~$\PP_{\operatorname{dimer}}$ in~\eqref{eq:dimer_measure} induces a probability measure~$\PP_{\operatorname{Int}}$ on the set of point configurations defined above.

We say that a black vertex is on the~$2i$th \emph{level} if its first coordinate is~$2i$. Similarly, we say that a white vertex with first coordinate~$2i+1$ is on level~$2i+1$.
\begin{lemma}\label{lem:particles_level}
Fix a dimer cover~$\mathcal D$ of~$G_\text{Az}$ and~$i\in \{0,\dots,N\}$. The black vertices on level~$2i$ are adjacent to~$N-i$ south or west dimers.   
\end{lemma}
\begin{proof}
On each even level, there are~$N$ black vertices and on each odd level, there are~$N+1$ white vertices. If there are~$k$ black vertices on level~$2i$ that are adjacent to a south or west dimer, it means that there have to be~$N+1-k$ black vertices on level~$2(i+1)$ adjacent to a north or east dimer, since all white vertices on level~$2i+1$ have to be adjacent to some dimer. Consequently, there are~$k-1$ black vertices adjacent to a south or west dimer on level~$2(i+1)$. Since there are~$N$ vertices on level~$0$ that are adjacent to a south or west dimer, we have proved the statement.
\end{proof}
The previous lemma shows that the point process introduced in the beginning of this section can naturally be denoted by~$\{u_s^t\}_{1\leq s\leq t\leq N}$ where~$u_s^t\leq u_{s+1}^t$ and lies on level~$N-t$. In particular, the coordinate of the black vertex associated with~$u_s^t$ is~$(2(N-t),2u_s^t+1)$.
\begin{lemma}[\cite{JN06}]
The point process~$\{u_s^t\}_{1\leq s\leq t\leq N}$ is an interlacing particle system, that is,~$u_s^{t+1}\leq u_s^t\leq u_{s+1}^{t+1}$.
\end{lemma}
\begin{proof}
Pick a black vertex~$b=(2i,2j+1)$ and let~$b'=(2(i+1),2j+1)$ be the black vertex immediately to its right. Suppose there are~$k$ particles whose vertical coordinate is less than or equal to that of~$b$. By an argument analogous to that in the proof of Lemma~\ref{lem:particles_level}, there are either~$k$ or~$k-1$ particles whose vertical coordinate is less than or equal to that of~$b'$. Since this holds for all black vertices, it follows that if there are~$k$ particles at or below~$b'$, then there must be a particle located at~$b'$.

Let~$b$ correspond to the black vertex associated with~$u_s^{t+1}$, and~$b'$ to the black vertex immediately to its right. If there are~$k$ particles at or below~$b$ and also~$k$ particles at or below~$b'$, then a particle occupies~$b'$, implying~$u_s^{t+1}=u_s^t<u_{s+1}^{t+1}$. If instead there are~$k-1$ particles at or below~$b'$, then~$u_s^{t+1}<u_s^t$, and moreover, we must have~$u_s^t\leq u_{s+1}^{t+1}$ to preserve the property established at the beginning of the proof.
\end{proof}

To capture the periodicity in the underlying dimer model, we color the particles of the point process~$\{u_s^t\}_{1\leq s\leq t}$ in two different colors depending on the location, see Definition~\ref{def:marked_interlacing} below.

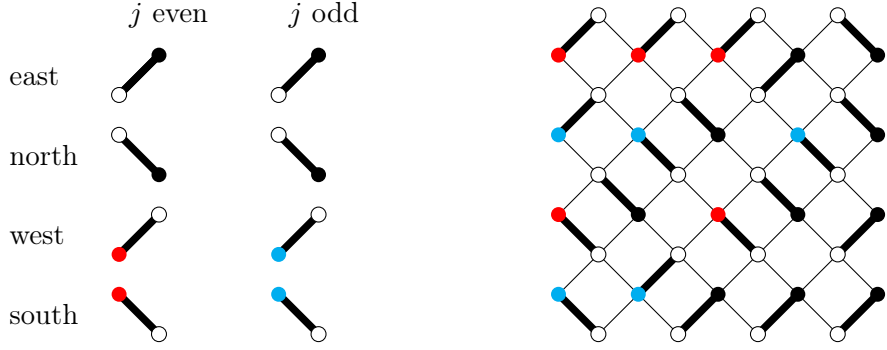
\begin{figure}
\begin{center}
\begin{tikzpicture}[scale=.75, rotate=-45]

\draw[right] (-4,1) node{$j$ even};   
\draw[right] (-2,3) node{$j$ odd};   

\draw[above right] (-4.5,-1.5) node{east};   
\foreach \x/\y in {-3.5/-.5,-1.5/1.5}
{
\draw[line width = 1mm] (\x+.5,\y+.5)--(\x+.5,\y+1.5);
\draw[line width = 1mm] (\x+.5,\y+.5)--(\x+.5,\y+1.5);
\draw (\x+.5,\y+.5) node[circle,draw=black,fill=white,inner sep=2pt]{};
\draw (\x+.5,\y+1.5) node[circle,fill,inner sep=2pt]{};
}

\draw[below right] (-4,-2) node{north};   
\foreach \x/\y in {-3/-1,-1/1}
{
\draw[line width = 1mm] (\x+.5,\y+.5)--(\x+1.5,\y+0.5);
\draw (\x+.5,\y+.5) node[circle,draw=black,fill=white,inner sep=2pt]{};
\draw (\x+1.5,\y+.5) node[circle,fill,inner sep=2pt]{};
}

\draw[above right] (-2.5,-3.5) node{west};   
\foreach \x/\y in {-1.5/-2.5}
{ 
\draw[line width = 1mm] (\x+.5,\y+.5)--(\x+.5,\y+1.5);
\draw (\x+.5,\y+.5) node[circle,fill=red,inner sep=2pt]{};
\draw (\x+.5,\y+1.5) node[circle,draw=black,fill=white,inner sep=2pt]{};
}

\foreach \x/\y in {.5/-.5}
{ 
\draw[line width = 1mm] (\x+.5,\y+.5)--(\x+.5,\y+1.5);
\draw (\x+.5,\y+.5) node[circle,fill=cyan,inner sep=2pt]{};  
\draw (\x+.5,\y+1.5) node[circle,draw=black,fill=white,inner sep=2pt]{};
}

\draw[below right] (-2,-4) node{south};   
\foreach \x/\y in {-1/-3}
{
\draw[line width = 1mm] (\x+.5,\y+.5)--(\x+1.5,\y+0.5);
\draw (\x+.5,\y+.5) node[circle,fill=red,inner sep=2pt]{};
\draw (\x+1.5,\y+.5) node[circle,draw=black,fill=white,inner sep=2pt]{};
}

\foreach \x/\y in {1/-1}
{
\draw[line width = 1mm] (\x+.5,\y+.5)--(\x+1.5,\y+0.5);
\draw (\x+.5,\y+.5) node[circle,fill=cyan,inner sep=2pt]{};  
\draw (\x+1.5,\y+.5) node[circle,draw=black,fill=white,inner sep=2pt]{};
}

\end{tikzpicture}
\qquad \qquad \qquad 
\begin{tikzpicture}[scale=.75, rotate=-45]

\foreach \x in {0,1,2,3}
\foreach \y in {0,...,3}
{\draw (-3.5+\x+\y,.5-\x+\y) rectangle (-2.5+\x+\y,1.5-\x+\y);
}

\foreach \x/\y in {1/-2,2/-1,3/0,2/1,-1/2}
{
\draw[line width = 1mm] (\x+.5,\y+.5)--(\x+.5,\y+1.5);
\draw (\x+.5,\y+.5) node[circle,draw=black,fill=white,inner sep=2pt]{};
\draw (\x+.5,\y+1.5) node[circle,fill,inner sep=2pt]{};
}

\foreach \x/\y in {-2/1,0/1,0/3,-1/4,-2/-1}
{
\draw[line width = 1mm] (\x+.5,\y+.5)--(\x+1.5,\y+0.5);
\draw (\x+.5,\y+.5) node[circle,draw=black,fill=white,inner sep=2pt]{};
\draw (\x+1.5,\y+.5) node[circle,fill,inner sep=2pt]{};
}

\foreach \x/\y in {-4/0,-3/1,-2/2}
{ 
\draw[line width = 1mm] (\x+.5,\y+.5)--(\x+.5,\y+1.5);
\draw (\x+.5,\y+.5) node[circle,fill=red,inner sep=2pt]{};
\draw (\x+.5,\y+1.5) node[circle,draw=black,fill=white,inner sep=2pt]{};
}

\foreach \x/\y in {-3/-1,0/-2}
{ 
\draw[line width = 1mm] (\x+.5,\y+.5)--(\x+.5,\y+1.5);
\draw (\x+.5,\y+.5) node[circle,fill=cyan,inner sep=2pt]{};  
\draw (\x+.5,\y+1.5) node[circle,draw=black,fill=white,inner sep=2pt]{};
}

\foreach \x/\y in {-2/-2,0/0}
{
\draw[line width = 1mm] (\x+.5,\y+.5)--(\x+1.5,\y+0.5);
\draw (\x+.5,\y+.5) node[circle,fill=red,inner sep=2pt]{};
\draw (\x+1.5,\y+.5) node[circle,draw=black,fill=white,inner sep=2pt]{};
}

\foreach \x/\y in {-1/-3,0/2,-2/0}
{
\draw[line width = 1mm] (\x+.5,\y+.5)--(\x+1.5,\y+0.5);
\draw (\x+.5,\y+.5) node[circle,fill=cyan,inner sep=2pt]{};  
\draw (\x+1.5,\y+.5) node[circle,draw=black,fill=white,inner sep=2pt]{};
}

 
\end{tikzpicture}
\end{center}
\caption{Left: The particles defined from the dimers visualized in red and cyan. The color depends on the parity of the coordinates of the particle. Right: A dimer cover of the Aztec diamond of size~$4$. The particles form an interlacing particle system. \label{fig:interlacing}}
\end{figure}

\subsection{The correlation kernel and the spectral curve}\label{sec:kernel}

In the uniform case,~$\alpha_i=\beta_i=1$, for all~$i$, it was proved in~\cite{Joh05a} that the point process~$\{u_s^t\}_{1\leq s\leq t}$ is a determinantal point process with correlation kernel determined from non-intersecting paths known as the DR-paths. More precisely, the point process is the restriction of the point process defined from the non-intersecting paths to the even levels. In our setting, the same argument holds and the correlation kernel for the point process~$\{u_s^t\}_{1\leq s\leq t}$ is the restriction of the correlation kernel given in~\cite{Ber21} restricted to the even levels. 

To express the correlation kernel we define the following~$2\times 2$ matrices:
\begin{equation}\label{eq:transition_matrices}
\phi_{2i-1}(z)=
\begin{pmatrix}
1 & \alpha_i^{-1} z^{-1} \\
\alpha_i & 1
\end{pmatrix},
\quad \text{and} \quad
\phi_{2i}(z)=\frac{1}{1-z^{-1}}
\begin{pmatrix}
1 & \beta_i^{-1} z^{-1} \\
\beta_i & 1
\end{pmatrix}, 
\end{equation}
for~$i=1,\dots,\ell$. We denote the product of these matrices by~$\Phi=\prod_{m=1}^{2\ell}\phi_m$. Note that~$\det \Phi(z)=1$. 

The correlation kernel is naturally described as a double contour integral on a higher genus Riemann surface. We discuss this Riemann surface before we recall the integral formulation of the correlation kernel. 

The characteristic polynomial~$P$ of the dimer model is given by
\begin{equation}
P(z,w)=(1-z^{-1})^\ell \det(\Phi(z)-wI),
\end{equation}
and the spectral curve is defined by
\begin{equation}
\mathcal R^\circ=\{(z,w)\in (\CC^*)^2: P(z,w)=0\},
\end{equation}
where~$\CC^*=\CC\backslash\{0\}$. It was proved in~\cite{Ber21} that the characteristic polynomial defined here coincides with the characteristic polynomial defined in~\cite{KOS06}. The spectral curve introduced in~\cite{KOS06} was proved to be a so-called \emph{Harnack curve} -- a specifically nice type of curve. Since the two spectral curves coincide,~$\mathcal R^\circ$ is a Harnack curve. The spectral curve is naturally compactified in an appropriate toric surface, and we denote this compactification by~$\mathcal R$. Concretely,~$\mathcal R=\mathcal R^{\circ}\cup\{p_0,p_\infty,q_0,q_\infty\}$, where
\begin{equation}
p_0=(0,1), \quad p_\infty=(\infty,1), \quad q_0=(1,0), \quad \quad q_\infty=(1,\infty).
\end{equation}
This Riemann surface was described in~\cite{Ber21} by gluing together two copies of~$\CC$ along intervals of the negative part of the real line. We describe that construction below.
\begin{lemma}[\cite{Ber21}]\label{lem:zeros_disc}
Let~$\operatorname{Disc}_w(P)$ be the discriminant of the polynomial~$P$ in the variable~$w$, and let~$p(z)=z^{2\ell}\operatorname{Disc}_w(P)$. Then
\begin{equation}
p(z)=\left((z-1)^\ell \Tr \Phi(z)\right)^2-4(z-1)^{2\ell}.
\end{equation}
The function~$p$ is a degree~$2\ell-1$ polynomial with zeros
\begin{equation}
0=z_0>z_1\geq z_2>z_3\geq z_4>\dots >z_{2\ell-3}\geq z_{2\ell-2}.
\end{equation}
\end{lemma}
To see that~$p$ is the discriminant of~$P$, we use that~$\det \Phi(z)=1$. It follows from the definition of~$\mathcal R$, that the zeros of the polynomial~$p$ from the previous lemma are the branch points of the Riemann surface. The cuts along which we glue together the two copies of~$\CC$ are taken between~$z_{2k}$ and~$z_{2k+1}$ for~$k=0,\dots,\ell-1$, where we set~$z_{2\ell-1}=-\infty$. The real part of~$\mathcal R$ consists of~$\ell$ connected components~$A_k$,~$k=0,\dots,\ell-1$, that are known as \emph{ovals}. For~$k=1,\dots,\ell-1$, we set~$A_k=\{(z,w)\in \mathcal R:z_{2k-1}\geq z\geq z_{2k}\}$ and we refer to them as the \emph{compact ovals}. For~$k=0$, the oval is referred to as the \emph{non-compact oval} and is given by~$A_0=\{(z,w)\in \mathcal R:0\leq z\leq +\infty\}$. By Lemma~\ref{lem:zeros_disc}, the cuts cannot shrink to a point for any choice of positive edge weights, while the compact ovals can be points. This means that the genus $g$ of $\mathcal R$, is bounded above by $\ell-1$. For simplicity, we assume throughout the paper, that $g=\ell-1$. See Figure~\ref{fig:rs_sheets} for an illustration of the Riemann surface~$\mathcal R$.

 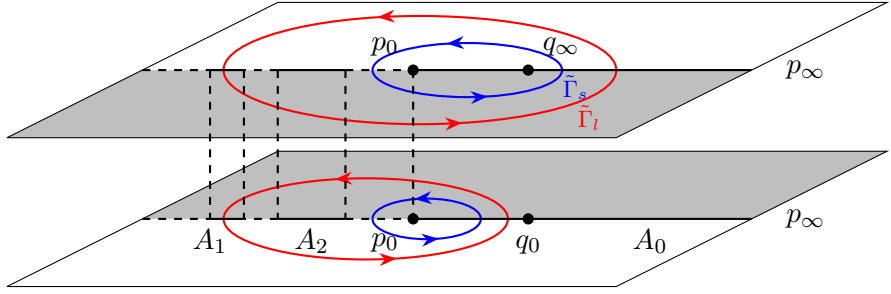
\begin{figure}[t]
 \begin{center}
 \begin{tikzpicture}[scale=.9]
   \tikzset{->--/.style={decoration={
  markings, mark=at position .3 with {\arrow{Stealth[length=2.3mm]}}},postaction={decorate}}}
   \tikzset{->-/.style={decoration={
  markings, mark=at position .6 with {\arrow{Stealth[length=2.3mm]}}},postaction={decorate}}}
   \fill[color=lightgray](-4,0)--(5,0)--(3,-1)--(-6,-1);
   \draw (-6,-1)--(3,-1)--(7,1)--(-2,1)--(-6,-1);
   \draw[blue,thick,->-] (-.6,0) arc(-180:0:1.4cm and .4cm);
   \draw[red,thick,->-] (-2.8,0) arc(-180:0:2.9cm and .8cm);   
   \draw[blue,thick,->-] (2.2,0) arc(0:180:1.4cm and .4cm);
   \draw[red,thick,->-] (3.,0) arc(0:180:2.9cm and .8cm);   
   \draw (0,0) node[circle,fill,inner sep=1.5pt,label=above left:$p_0$]{};
   \draw (5.8,0) node{$p_{\infty}$};
   \draw (1.7,0) node[circle,fill,inner sep=1.5pt,label=above right:$q_{\infty}$]{};
   \draw[blue] (2.4,-.25) node{\footnotesize{$\tilde \Gamma_s$}};
   \draw[red] (2.6,-.7) node{\footnotesize{$\tilde \Gamma_l$}};   
   \draw[thick] (5,0)--(0,0);
   \draw[thick] (-2,0)--(-1,0);
   \draw[thick] (-3,0)--(-2.5,0);
   \draw [thick,dashed] (-1,0)--(0,0);
   \draw [thick,dashed] (-2,0)--(-2.5,0);
   \draw [thick,dashed] (-4,0)--(-3,0);
   \fill[color=lightgray]
   (-2,-1.2)--(7,-1.2)--(5,-2.2)--(-4,-2.2);
   \draw (-6,-3.2)--(3,-3.2)--(7,-1.2)--(-2,-1.2)--(-6,-3.2);
   \draw[blue,thick,->-] (-.6,-2.2) arc(-180:0:.8cm and .3cm);
   \draw[red,thick,->-] (-2.8,-2.2) arc(-180:0:2.1cm and .6cm);   
   \draw[blue,thick,->-] (1.,-2.2) arc(0:180:.8cm and .3cm);
   \draw[red,thick,->-] (1.4,-2.2) arc(0:180:2.1cm and .6cm);   
   \draw (0,-2.2) node[circle,fill,inner sep=1.5pt,label=below left:$p_0$]{};
   \draw (5.8,-2.2) node{$p_{\infty}$};
   \draw (1.7,-2.2) node[circle,fill,inner sep=1.5pt,label=below:$q_0$]{};
   \draw (3.5,-2.2) node[below]{$A_0$};
   \draw (-3.,-2.2) node[below]{$A_1$};
   \draw (-1.5,-2.2) node[below]{$A_2$};
   \draw[thick] (0,-2.2)--(5,-2.2);
   \draw[thick] (-2.5,-2.2)--(-3,-2.2);
   \draw[thick] (-1,-2.2)--(-2,-2.2);
   \draw [thick,dashed] (-1,-2.2)--(0,-2.2);
   \draw [thick,dashed] (-2,-2.2)--(-2.5,-2.2);
   \draw [thick,dashed] (-4,-2.2)--(-3,-2.2);
   \draw[thick,dashed] (0,0)--(0,-2.2);
   \draw[thick,dashed] (-1,0)--(-1,-2.2);
   \draw[thick,dashed] (-2,0)--(-2,-2.2);
   \draw[thick,dashed] (-2.5,0)--(-2.5,-2.2);
   \draw[thick,dashed] (-3,0)--(-3,-2.2);
  \end{tikzpicture}
 \end{center}
  \caption{The Riemann surface $\mathcal R$ represented as two copies of the complex plane. The cuts (dashed) are located along the negative part of the real lines. The compact ovals $A_k$, $k=1,\dots,\ell-1$ (solid) are located along the negative part of the real lines, and the non-compact oval $A_0$ (solid) are located along the positive part of the real lines. The gray areas are connected via the cuts, and so are the white areas. In this illustration, the curve $\tilde \Gamma_s$ (blue) is a simple loop, while $\tilde \Gamma_l$ (red) is the union of two simple loops.\label{fig:rs_sheets}}
\end{figure}

For a simple closed curve~$\tilde \Gamma$ in~$\mathcal R$, we define the \emph{exterior} of~$\tilde \Gamma$ as the connected component of~$\mathcal R\backslash \tilde \Gamma$ that contains~$p_\infty=(\infty,1)$. The complement of the exterior is called the \emph{interior} of~$\tilde \Gamma$. The curve~$\tilde \Gamma$ is positively oriented if the interior is to the left of the curve. 

We are ready to express the correlation kernel for the point process~$\{u_s^t\}_{1\leq s\leq t\leq N}$. To simplify notation, we write
\begin{equation}\label{eq:adjuagate_matrix}
Q(z,w)=\frac{\adj(wI-\Phi(z))}{\partial_w\det(wI-\Phi(z))}.
\end{equation}
For simplicity, we will assume that the size of the Aztec diamond is~$2\ell N$ instead of~$N$. We also write~$(2(\ell x+i),2(2y+j)+1)$, where~$x=0\dots,2N-1$,~$i=0,\dots,\ell-1$,~$y=0,\dots,\ell N-1$, and~$j=0,1$, instead of~$(2i,2j+1)$, and identify these points with~$(\ell x+i,2y+j)$. Recall also that the coordinate of the black vertex associated with~$u_s^t$ is~$(\ell x+i,2y+j)=(2\ell N-t,u_s^t)$.
\begin{theorem}\label{thm:correlation_kernel}
Let~$\rho_k^{\operatorname{Int}}=\rho_{k,N}^{\operatorname{Int}}$ be the~$k$th correlation function associated to the point process~$\{u_s^t\}_{1\leq s\leq t\leq 2\ell N}$. For~$k$ black points~$(\ell x_p+i_p,2y_p+j_p)$,~$p=1,\dots,k$, where~$x_p=0\dots,2N-1$,~$i_p=0,\dots,\ell-1$,~$y_p=0,\dots,\ell N-1$, and~$j_p=0,1$, we have
\begin{multline}
\rho_k^{\operatorname{Int}}\left(\ell x_1+i_1,2y_1+j_1;\dots;\ell x_k+i_k,2y_k+j_k\right) \\
=\det\left(K_{\operatorname{Int}}\left(\ell x_p+i_p,2y_p+j_p;\ell x_{p'}+i_{p'},2y_{p'}+j_{p'}\right)\right)_{p,p'=1}^k,
\end{multline}
where
\begin{multline}\label{eq:correlation_kernel}
\left[K_{\operatorname{Int}}(\ell x+i,2y+j;\ell x'+i',2y'+j')\right]_{j',j=0}^1 \\
= -\frac{\one_{\ell x+i>\ell x'+i'}}{2\pi\i}\int_{\tilde \Gamma_l}\left(\prod_{m=1}^{2i'}\phi_m(z)\right)^{-1}Q(z,w)\prod_{m=1}^{2i}\phi_m(z)w^{x-x'}z^{y'-y}\frac{\d z}{z} \\
+\frac{1}{(2\pi\i)^2}\int_{\tilde \Gamma_s}\int_{\tilde \Gamma_l}\left(\prod_{m=1}^{2i'}\phi_m(z_1)\right)^{-1}Q(z_1,w_1)Q(z_2,w_2)\prod_{m=1}^{2i}\phi_m(z_2) \\
\times \frac{(z_2-1)^{\ell N}}{(z_1-1)^{\ell N}}\frac{w_2^{x-N}}{w_1^{x'-N}}\frac{z_1^{y'}}{z_2^y}\frac{\d z_2\d z_1}{z_2(z_2-z_1)}.
\end{multline}
The contours~$\tilde \Gamma_l$ and~$\tilde \Gamma_s$ are simple closed positively oriented curves in~$\mathcal R$ containing~$p_0$ and~$q_\infty$ in their interior, and~$p_\infty$ and~$q_0$ in their exterior. Moreover,~$\tilde \Gamma_s$ is contained in the interior of~$\tilde \Gamma_l$. 
\end{theorem}
\begin{proof}
Since we consider a point process on a discrete space, 
\begin{equation}\label{eq:rho_probability}
\rho_k^{\operatorname{Int}}\left(\ell x_1+i_1,2y_1+j_1;\dots;\ell x_k+i_k,2y_k+j_k\right)
=\PP_{\operatorname{Int}}\left(\text{particles at } \mathrm b_1, \dots, \mathrm b_k\right),
\end{equation}
where~$\mathrm b_p=(2(\ell x_p+i_p),2(2y_p+j_p)+1)$ for~$p=1,\dots,k$. By definition of the interlacing particle system, the probability on the right-hand side of~\eqref{eq:rho_probability} can be expressed in terms of~$\PP_{\operatorname{dimer}}$: For a black vertex~$\mathrm b_p=(2(\ell x_p+i_p),2(2y_p+j_p)+1)$ and~$q=0,1$, let~$\mathrm w_{pq}=(2(\ell x_p+i_p)+1,2(2y_p+j_p+q))$ be the adjacent white vertex so that~$\mathrm b_p\mathrm w_{pq}$ is a south edge if~$q=0$ and a west edge if~$q=1$. Then
\begin{equation}\label{eq:particle_dimer_probabilities}
\PP_{\operatorname{Int}}\left(\text{particles at } \mathrm b_1, \dots, \mathrm b_k\right)
=
\PP_{\operatorname{dimer}}\left(\mathrm b_p \mathrm w_{p0}\in \mathcal D \text{ or } \mathrm b_p \mathrm w_{p1} \in \mathcal D \text{ for all }p=1,\dots,k\right).
\end{equation}
Since~$\mathrm b_p \mathrm w_{p0}\in \mathcal D$ and~$\mathrm b_p \mathrm w_{p1}\in \mathcal D$ are disjoint events, the right-hand side of~\eqref{eq:particle_dimer_probabilities} is equal to
\begin{equation}\label{eq:sum_probabilities}
\sum_{(q_1,\dots,q_k)\in \{0,1\}^k} \PP_{\operatorname{dimer}}\left(\mathrm b_p \mathrm w_{pq_p}\in \mathcal D \text{ for all }p=1,\dots,k\right),
\end{equation}
where the sum runs over all~$k$-tuple in~$\{0,1\}^k$.

The terms in~\eqref{eq:sum_probabilities} are of the form~\eqref{eq:kenyons_formula}, and the sum is therefore a sum over determinants. By the multilinearity of determinants,~\eqref{eq:sum_probabilities} is equal to
\begin{multline}\label{eq:sum_det}
\sum_{(q_1,\dots,q_k)\in \{0,1\}^k}
\det\left(K_\text{Az}(\mathrm w_{p'q_{p'}},\mathrm b_{p'})K_\text{Az}^{-1}(\mathrm b_p,\mathrm w_{p'q_{p'}})\right)_{p,p'=1}^k \\
=\det\left(\sum_{q_{p'}\in \{0,1\}}K_\text{Az}(\mathrm w_{p'q_{p'}},\mathrm b_{p'})K_\text{Az}^{-1}(\mathrm b_p,\mathrm w_{p'q_{p'}})\right)_{p,p'=1}^k.
\end{multline}
What is left to prove, is that the sum inside the determinant is the double contour integral given in the statement. 

The Kasteleyn matrix~$K_\text{Az}$ is given in~\eqref{eq:kasteleyn} and its inverse~$K_\text{Az}^{-1}$ was derived in~\cite{Ber21}, see also~\cite{BB23, BNR25}
for a formulation closer to what we use here, and is given by
\begin{multline}\label{eq:inverse_kasteleyn}
K_\text{Az}^{-1}(\mathrm b_{\ell x+i,2y+j},\mathrm w_{\ell x'+i',2y'+j'}) \\
= -\left(\frac{\one_{\ell x+i>\ell x'+i'}}{2\pi\i}\int_{\tilde \Gamma_l}\left(\prod_{m=1}^{2i'+1}\phi_m(z)\right)^{-1}Q(z,w)\prod_{m=1}^{2i}\phi_m(z)w^{x-x'}z^{y'-y}\frac{\d z}{z}\right)_{j'+1,j+1} \\
+\left(\frac{1}{(2\pi\i)^2}\int_{\tilde \Gamma_s}\int_{\tilde \Gamma_l}\left(\prod_{m=1}^{2i'+1}\phi_m(z_1)\right)^{-1}Q(z_1,w_1)Q(z_2,w_2)\prod_{m=1}^{2i}\phi_m(z_2)\right. \\
\left.\times \frac{(z_2-1)^{\ell N}}{(z_1-1)^{\ell N}}\frac{w_2^{x-N}}{w_1^{x'-N}}\frac{z_1^{y'}}{z_2^y}\frac{\d z_2\d z_1}{z_2(z_2-z_1)}\right)_{j'+1,j+1},
\end{multline}
where~$\mathrm b_{\ell x+i,2y+j}=(2(\ell x+i),2(2y+j)+1)$,~$\mathrm w_{\ell x'+i',2y'+j'}=(2(\ell x'+i')+1,2(2y'+j')+2)$, and~$\tilde \Gamma_s$ and~$\tilde \Gamma_l$ are as in the statement. 

With the notation used above, we have~$\mathrm b_{p'}=\mathrm b_{\ell x_{p'}+i_{p'},2y_{p'}+j_{p'}}$, and~$\mathrm w_{p'q_{p'}}=\mathrm w_{\ell x_{p'}+i_{p'},2y_{p'}+j_{p'}-1+q_{p'}}$. It follows from~\eqref{eq:kasteleyn} and~\eqref{eq:inverse_kasteleyn}, that the sum within the determinant on the right-hand side of~\eqref{eq:sum_det} results in multiplying the integrand of~\eqref{eq:inverse_kasteleyn} with~$\phi_{2i_{p'}+1}(z_1)$, given in~\eqref{eq:transition_matrices}, from the left. This proves the theorem.
\end{proof}

\begin{remark}\label{rem:relation_kernal}
The correlation kernel in Theorem~\ref{thm:correlation_kernel} is simply the restriction to even levels of the correlation kernel for the associated non-intersecting paths point process. That correlation kernel was derived in~\cite{Ber21} using a technique developed in~\cite{DK21, BD19}, and Theorem~\ref{thm:correlation_kernel} therefore follows from~\cite{JN06}. Here, we instead used the inverse Kasteleyn matrix, obtained in~\cite{BB23}. However, the expression~\eqref{eq:inverse_kasteleyn} was derived using the already known double contour integral formulation of the correlation kernel for the non-intersecting paths point process. So, the proof provided here is a bit of a detour. The reason we still include it, is that a similar argument would allow us to study the corresponding point process at the other turning points using the same expression for the inverse Kasteleyn matrix~\eqref{eq:inverse_kasteleyn}. This is not as easy to do if we use the non-intersecting paths formulation and rely on the formulas from~\cite{BD19, BB23}. Indeed, if we did, the matrices~\eqref{eq:transition_matrices} would be~$\ell\times \ell$ matrices, and we would have to rely on the more sophisticated method developed in~\cite{BB23}.
\end{remark}

\begin{remark}
As seen in Theorem~\ref{thm:correlation_kernel}, the expression of the correlation kernel~\eqref{eq:correlation_kernel} is naturally expressed as a matrix:
\begin{equation}
\left[K_{\operatorname{Int}}(\ell x+i,2y+j;\ell x'+i',2y'+j')\right]_{j',j=0}^1=
\begin{psmallmatrix}
K_{\operatorname{Int}}(\ell x+i,2y;\ell x'+i',2y') & K_{\operatorname{Int}}(\ell x+i,2y+1;\ell x'+i',2y') \\
K_{\operatorname{Int}}(\ell x+i,2y;\ell x'+i',2y'+1) & K_{\operatorname{Int}}(\ell x+i,2y+1;\ell x'+i',2y'+1)
\end{psmallmatrix}.
\end{equation}
However, as we evaluate the correlation functions, we use the individual entries. 
\end{remark}

\section{Convergence to the marked GUE-corners process}\label{sec:main}
In this section we state our results. In particular, we introduce the marked GUE-corners process and show that it is the limit of the appropriate colored interlacing particle system defined from the dimer model. The steepest descent analysis of the correlation kernel is deferred to Section~\ref{sec:steepest}, and the derivations of $\tau$ and $\sigma^2$ are postponed to Section~\ref{sec:ty_sigma}.

In~\cite{Ber21}, a diffeomorphism between the rough region and ``half'' of the Riemann surface~$\mathcal R$ (the gray shaded part in Figure~\ref{fig:rs_sheets}) was constructed. This diffeomorphism was used to describe the geometry of the arctic curve and, in particular, to show that there are four \emph{turning points} -- points where the rough region touches the boundary of the Aztec diamond -- one on each side of the Aztec diamond. Under this diffeomorphism the turning points correspond to~$p_0$,~$p_\infty$,~$q_0$, and~$q_\infty$. We focus on a neighborhood around the turning point on the right side of the Aztec diamond, which corresponds to~$q_\infty$. In particular, we show that the correlation kernel from Theorem~\ref{thm:correlation_kernel} converges to the correlation kernel associated to the marked GUE-corners process, see Theorem~\ref{thm:limit_correlation_function} below. We postpone the proof to Section~\ref{sec:steepest}. Our main result, that the (appropriately scaled) interlacing particle system from Section~\ref{sec:interlacing_particles} converges weakly to the marked GUE-corners process follows as a corollary.  

The turning point we are interested in is located, up to leading order, at~$(\ell x+i,2y+j)\sim (2\ell N, 2 N\tau )$, for some~$\tau \in (0,\ell)$.
\begin{proposition}\label{prop:turning_point}
For~$k=1,\dots,\ell$, set~$a_k=\alpha_k\beta_{k-1}^{-1}$ and~$b_k=\alpha_k^{-1}\beta_k$. Then
\begin{equation}\label{eq:turning_point}
\tau =
\sum_{k=1}^\ell \frac{1+a_k+a_kb_k+a_kb_ka_{k+1}}
{(1+a_k)(1+b_k)(1+a_{k+1})}
\end{equation}
where~$a_{\ell+1}=a_1$.
\end{proposition}
This statement is part of Proposition~\ref{prop:exact_values}, and is proved in Section~\ref{sec:ty_sigma}. Note that each term in the sum on the right-hand side of \eqref{eq:turning_point} is in $(0,1)$, so $\tau \in (0,\ell)$ if $\tau $ is defined by the right-hand side of \eqref{eq:turning_point}. The value of~$\tau $ was expressed in terms of the polynomial from Lemma~\ref{lem:zeros_disc} in~\cite{Ber21}:
\begin{equation}
\tau =\frac{1}{2}\frac{p'(1)}{p(1)},
\end{equation}
Proposition~\ref{prop:turning_point} refines that result. 

We zoom in around the turning point and introduce the new coordinates~$(t,\mu)\in \Lambda=\ZZ_{>0}\times \RR$ and set
\begin{equation}\label{eqn:zoom_coordinates}
\ell x_p+i_p=2\ell N-t_p, \quad \text{and} \quad y_p=\lfloor N\tau +\sqrt{N}\mu_p\rfloor, \quad p=1,2,
\end{equation}
where~$t_p=0,\dots,2\ell N-1$, and~$\mu_p\in \RR$ is such that~$y_p=0,\dots,\ell N-1$.

Before we formulate the limiting result for the correlation kernel, we need to introduce a few more objects. For~$(t,j)\in \ZZ\times \{0,1\}$, we introduce the function that will capture the periodicity of the model in the scaling limit,
\begin{equation}\label{eq:period_function}
\nu(t,j)=
\begin{cases}
\frac{1}{1+\alpha_{\ell+1-t}\beta_{\ell-t}^{-1}}, & j=0, \\
\frac{\alpha_{\ell+1-t}\beta_{\ell-t}^{-1}}{1+\alpha_{\ell+1-t}\beta_{\ell-t}^{-1}}, & j=1.
\end{cases}
\end{equation}
Recall that the weights are periodic --~$\alpha_{i+\ell}=\alpha_i$ and~$\beta_{i+\ell}=\beta_i$ for all~$i\in \ZZ$. We define a gauge function~$g$ by setting
\begin{equation}\label{eq:gauge_function}
g(\ell x+i,2y+j)=N^{\frac{\ell x+i}{2}} B^{x}\sigma^{\ell x+i}
\prod_{m=1}^i(1+\alpha_m^{-1}\beta_m)(1+\alpha_{m+1}\beta_m^{-1})\alpha_{i+1}^{-j}
\end{equation}
where
~$B=\prod_{m=1}^\ell(1+\alpha_m^{-1}\beta_m)(1+\alpha_{m+1}\beta_m^{-1})$ and the parameter~$\sigma>0$ is defined by
\begin{equation}\label{eq:sigma_square}
\sigma^2=
\sum_{k=1}^\ell \frac{(1+a_k+a_kb_k+a_kb_ka_{k+1})(b_k+a_{k+1}+b_ka_{k+1}+a_ka_{k+1})}
{(1+a_k)^2(1+b_k)^2(1+a_{k+1})^2}
\end{equation}
where, as before,~$a_k=\alpha_k\beta_{k-1}^{-1}$ and~$b_k=\alpha_k^{-1}\beta_k$. The parameter~$\sigma^2$ is naturally defined by the second derivative of the \emph{action function} at~$q_\infty$, which is how it is introduced in Section~\ref{sec:steepest}, see Corollary \ref{lem:action_taylor}. The exact form given in~\eqref{eq:sigma_square} is determined in Section~\ref{sec:ty_sigma}. 
\begin{example}[The two-periodic Aztec diamond]
Let $\ell=2$, and $\alpha_1^{-1}=\beta_1^{-1}=\alpha_2=\beta_2=a$ for some $a>0$. Then
\begin{equation}
\tau=1, \quad \sigma^2=\frac{2}{(a+a^{-1})^2}, \quad 
\nu(t,j)=
\begin{cases}
\frac{a^2}{1+a^2}, & t \text{ even}, \,\, j=0, \quad \text{or} \quad t \text{ odd}, \,\, j=1, \\
\frac{1}{1+a^2}, & t \text{ odd}, \,\, j=0 \quad \text{or} \quad t \text{ even}, \,\, j=1.
\end{cases}
\end{equation}
\end{example}

We denote $\Lambda_{\{0,1\}}= \mathbb{Z}_{>0}\times\mathbb{R}\times\{0,1\}$. Then we have the following limiting result for the correlation kernel.
\begin{theorem}[Theorem \ref{thm:0}]\label{thm:limit_correlation_function}
Let~$K_{\operatorname{Int}}$ be the correlation kernel given in Theorem~\ref{thm:correlation_kernel}, then,  with the notation introduced above,  for $\mu_1\not=\mu_2$, 
\begin{multline}
\lim_{N\to \infty}\frac{g(\ell x_1+i_1,2y_1+j_1)}{g(\ell x_2+i_2,2y_2+j_2)}N^{\frac{1}{2}}K_{\operatorname{Int}}\left(\ell x_1+i_1,2y_1+j_1;\ell x_2+i_2,2y_2+j_2\right) \\
=\nu(t_2,j_2)\sigma^{-1}K_{\operatorname{GUE}}(t_1,\sigma^{-1}\mu_1;t_2,\sigma^{-1}\mu_2),
\end{multline}
where
\begin{equation}\label{eq:kernel_GUE}
K_{\operatorname{GUE}}(t_1,\mu_1;t_2,\mu_2)=\frac{1}{(2\pi\i)^2}\int_{\gamma_s'}\int_{\gamma_l'}\e^{\frac{1}{2}(z_2^2-z_1^2)}\e^{\mu_1 z_1-\mu_2z_2}\frac{z_2^{t_2}}{z_1^{t_1}}\frac{\d z_1\d z_2}{z_2-z_1},
\end{equation}
and the sequence on the left-hand side is uniformly bounded in $N$
 on compact subsets of 
$\Lambda_{\{0,1\}}^2$ with respect to the embedding given by $(t_p, \mu_p, j_p)\in \Lambda_{\{0,1\}}$ for $p=1,2$.
Here~$\gamma_s'$ is a counterclockwise oriented contour around the origin and~$\gamma_{\ell}'$ is a contour connection of $-i\infty$ and $+i\infty$ and 
 is at the right of $\gamma_s$ if $\mu_1<\mu_2$ and at the left if $\mu_1> \mu_2$, see Figure \ref{fig:contours_gue_corners}. 
 \end{theorem}
 
 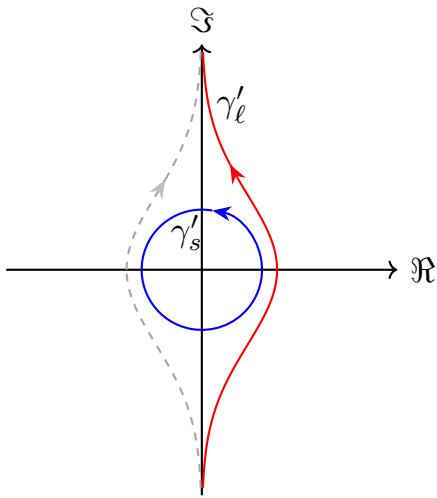
\begin{figure}
 \begin{center}
 \begin{tikzpicture}[scale=2]

    \draw[->, thick] (-1.3,0) -- (1.3,0) node[right, scale=1.2] {$\Re$}; 
    \draw[->, thick] (0,-1.5) -- (0,1.5) node[above, scale=1.2] {$\Im$};

    \draw[thick, blue, -{Stealth[length=2.5mm, width=2mm]}] (0.4,0) arc (0:80:0.4);
    \draw[thick, blue] (80:0.4) arc (80:360:0.4);

    \draw[thick, red, domain=0:0.7, samples=100, smooth, variable=\y, -{Stealth[length=2.5mm, width=2mm]}]
         plot ( {exp(-2*\y*\y)/2}, {\y} );
    \draw[thick, red, domain=0.7:1.45, samples=150, smooth, variable=\y]
         plot ( {exp(-2*\y*\y)/2}, {\y} );
    \draw[thick, red, domain=-1.45:0, samples=200, smooth, variable=\y]
         plot ( {exp(-2*\y*\y)/2}, {\y} );

    \draw[thick, gray!70, dashed, domain=-1.45:0, samples=150, smooth, variable=\y]
         plot ( {-exp(-2*\y*\y)/2}, {\y} );
    \draw[thick, gray!70, dashed, domain=0:0.6, samples=120, smooth, variable=\y, -{Stealth[length=2.5mm, width=2mm, color=gray!50]}]
         plot ( {-exp(-2*\y*\y)/2}, {\y} );
    \draw[thick, gray!70, dashed, domain=0.6:1.45, samples=150, smooth, variable=\y]
         plot ( {-exp(-2*\y*\y)/2}, {\y} );

    \node[scale=1.2] at (-0.1, 0.25) {$\gamma_s'$};

    \node[scale=1.2] at (0.2, 1.1) {$\gamma_{\ell}'$};

\end{tikzpicture}

\end{center}
 \caption{Contours of integration for the integral representation of the GUE-corners process; $\gamma_s'$ is drawn in blue and $\gamma_{\ell'}$ is drawn in red when $\mu_1<\mu_2$ and in a dashed grey line for $\mu_1> \mu_2$.} \label{fig:contours_gue_corners}
 \end{figure}

The proof of the previous theorem is postponed to Section~\ref{sec:steepest}.

The kernel~$K_{\operatorname{GUE}}$ is the correlation kernel associated to the GUE-corners process. Recall that the GUE-corners process is defined as follows (see e.g.,~\cite{Gor21}): Let~$X=(X_{ij})_{i,j \in\mathbb{Z}_{>0}}$ be a random infinite matrix defined by the independent gaussian random variables ~$X_{ij}\sim \frac{1}{\sqrt{2}}(N(0,1)+\i N(0,1))$ for $i>j$ and ~$X_{ii}\sim N(0,1)$. For~$t\in \ZZ_{>0}$, let~$\xi_1^t\leq \dots\leq \xi_t^t$ be the eigenvalues of the principal~$t\times t$ top-left corner of~$X$. The point configuration~$(\xi_s^t)_{1\leq s\leq t}$ is called the \emph{GUE-corners process}. The GUE-corners process is expected to arise as a universal scaling limit at turning points of uniformly distributed dimer models. This limit was first observed in~\cite{OR06, JN06}, see also~\cite{Joh05a}, and has by now been proved for a large family of lozenge tiling models~\cite{GP15, AG21}.

The GUE-corners process is a determinantal point process~\cite{JN06} with correlation kernel~$K_{\operatorname{GUE}}$ given in~\eqref{eq:kernel_GUE}. Classically, this kernel is expressed in terms of Hermite polynomials; see~\cite{Joh05a, JN06, OR06}. However, we will not use that formulation here, instead we use the double integral representation~\eqref{eq:kernel_GUE} derived in \cite{OR06} (see also \cite{JN06}).

In contrast to the uniform setting, we consider inhomogeneous edge weights. This inhomogeneity persists in the scaling limit and is encoded in the function~$\nu(s,j)$. In particular, Theorem~\ref{thm:limit_correlation_function} shows that we do not observe the GUE-corners process. Instead, we see an inhomogeneous version that we call the \emph{marked GUE-corners process}.
\begin{definition}\label{def:marked_GUE}
Let~$\PP_{\operatorname{GUE}}$ be the GUE-corners process defined on~$\Lambda=\ZZ_{>0}\times \RR$. For a measurable function~$\theta:\Lambda\to [0,1]$, the \emph{marked GUE-corners process}~$\PP_{\operatorname{GUE}}^\theta$ is the point process on~$\Lambda_{\{0,1\}}=\Lambda \times \{0,1\}$ obtained by assigning to each point in a given realization from~$\PP_{\operatorname{GUE}}$ an independent mark~$m\sim \operatorname{Bernoulli}(\theta(t,\mu))$, where~$(t,\mu)\in \Lambda$ is the location of the point. We denote its corresponding point configurations ~$(\xi_{s,j}^t)_{1\leq s\leq t}$.
\end{definition} 
That the mark~$m\sim \operatorname{Bernoulli}(\theta(t,\mu))$ means that~$m$ is distributed as a Bernoulli random variable taking the value~$1$ with probability~$\theta(t,\mu)$ and value~$0$ with probability~$1-\theta(t,\mu)$. 

It was proved in~\cite{CG23} -- for a general determinantal point process under very mild assumptions -- that the point process~$\PP_{\operatorname{GUE}}^\theta$ is a determinantal point process with correlation kernel
\begin{equation}\label{eq:marked_kernel}
K_{\operatorname{GUE}}^\theta(t_1,\mu_1,j_1;t_2,\mu_2,j_2)=\left(\theta(t_1,\mu_1)\delta_{1j_1}+(1-\theta(t_1,\mu_1))\delta_{0j_1}\right)K_{\operatorname{GUE}}(t_1,\mu_1;t_2,\mu_2),
\end{equation}
where~$\delta_{jj'}$ is the Kronecker delta function. 
\begin{remark}
In~\cite{CG23}, the pre-factor in~\eqref{eq:marked_kernel} was combined with the reference measure instead of the correlation kernel. The reason we do not follow their convention is that this formulation naturally appears from our calculations, as seen in Theorem~\ref{thm:correlation_kernel}.  
\end{remark}

We define~$\theta:\Lambda\to [0,1]$ by
\begin{equation}\label{eq:theta}
\theta(t,\mu)=\theta(t)=
\frac{\alpha_{\ell+1-t}\beta_{\ell-t}^{-1}}{1+\alpha_{\ell+1-t}\beta_{\ell-t}^{-1}},
\end{equation}
so that
\begin{equation}
\nu(t,j)=\theta(t)\delta_{1j}+(1-\theta(t))\delta_{0j}.
\end{equation}

Before we can state our main result, we need to embed the interlacing particle system~$\{u_s^t\}_{1\leq s\leq t\leq 2\ell N}$ into~$\Lambda_{\{0,1\}}$. Recall that~$u_s^t=2y+j$ for some~$y=0,\dots,\ell N-1$ and~$j=0,1$. 
\begin{definition}\label{def:marked_interlacing}
Each point~$u_s^t=2y+j$ in a particle realization of~$\PP_{\operatorname{Int}}$ is assigned the mark~$j\in\{0,1\}$. This marked point is denoted by~$u_{s,j}^t$. 
\end{definition}
The marking introduced in the previous definition is not the same type of marking as discussed in Definition~\ref{def:marked_GUE}, in particular, here the marking is deterministic given the location of the particle. 

With the above definitions, our main theorem is a corollary of Theorem~\ref{thm:limit_correlation_function}.
\begin{corollary}[Theorem \ref{thm:1}]\label{cor:week_converence}
Let~$(u_{s,j}^t)_{1\leq s\leq t\leq 2\ell N}$ be the interlacing particles system defined in Definition~\ref{def:marked_interlacing}, and let~$\tau $ and~$\sigma^2$ be defined in~\eqref{eq:turning_point} and~\eqref{eq:sigma_square}. Then
\begin{equation}
\left(\frac{u_{s,j}^{t}-2N\tau }{2\sigma\sqrt{N}}\right)_{1\leq s\leq t}\to \left(\xi_{s,j}^t\right)_{1\leq s\leq t},
\end{equation}
in the sense of weak convergence as~$N\to \infty$, and~$(\xi_{s,j}^t)_{1\leq s\leq t}$ is the marked GUE-corners process with~$\theta$ given in~\eqref{eq:theta}. That is, for any compactly supported continuous functions~$\varphi:\Lambda_{\{0,1\}}\to [0,1]$,
\begin{equation}
\lim_{N\to \infty}\EE_{\operatorname{Int}}\left[\prod_v \left(1-\varphi\left(t_v,\frac{u_{s_v,j_v}^{t_v}-2 N\tau }{2\sigma \sqrt{N}},j_v\right)\right)\right]
=\EE_{\operatorname{GUE}}^\theta\left[\prod_v (1-\varphi(t_v,\xi_{s_v,j_v}^{t_v},j_v))\right],
\end{equation}
where~$\EE_{\operatorname{Int}}$ and~$\EE_{\operatorname{GUE}}^\theta$ are expectations of~$\PP_{\operatorname{Int}}$ and~$\PP_{\operatorname{GUE}}^\theta$, respectively.
\end{corollary}
\begin{proof} 
Recall that~$\rho_k^{\operatorname{Int}}$ is the~$k$th correlation function associated with the point process~$\PP_{\operatorname{Int}}$. By definition of correlation functions, we have
\begin{multline}\label{eq:fredholm}
\EE_{\operatorname{Int}}\left[\prod_v \left(1-\varphi\left(t_v,\frac{u_{s_v,j_v}^{t_v}-2 N\tau }{2\sigma\sqrt{N}},j_v\right)\right)\right]
=
\sum_{k=0}^\infty\frac{(-1)^k}{k!}\sum_{\Lambda_N^k}\prod_{v=1}^k \varphi\left(t_v,\frac{2y_v+j_v-2 N\tau }{2\sigma\sqrt{N}},j_v\right) \\
\times \rho_k^{\operatorname{Int}}(\ell x_1+i_1,2y_1+j_1;\dots;\ell x_k+i_k,2y_k+j_k),
\end{multline}
where the second summation runs over all~$(t_1,2y_1+j_1;\dots;t_k,2y_k+j_k)\in (\Lambda_N)^k$, where $\Lambda_N=\{0,\dots,2\ell N-1\}\times \{0,\dots,2\ell N-1\}$. It is natural to divide this summation into three parts: first summing over all~$y_v$, then over all~$j_v$, and last over all~$t_v$,~$v=1,\dots,k$.
By Theorem ~\ref{thm:correlation_kernel} the correlation functions ~$\rho_k^{\operatorname{Int}}$ express as determinants and Theorem  ~\ref{thm:limit_correlation_function} asserts that under the embedding coming from the coordinate choice in \eqref{eqn:zoom_coordinates} they converge to the correlation functions of the marked GUE-corners process, $\rho_k^{\operatorname{GUE},\theta}$, uniformly on compact subsets of $\Lambda_{\{0,1\}}^k$.
Thus it follows from Theorem ~\ref{thm:limit_correlation_function}, that for fixed~$k$, the summation over~$y_v$ converges to an integral:
\begin{multline}\label{eq:riemann_sum}
\lim_{N\to\infty}\sum_{y_1,\dots,y_k}\prod_{v=1}^k \varphi\left(t_v,\frac{2y_v+j_v-2 N\tau }{2\sigma \sqrt{N}},j_v\right)
\rho_k^{\operatorname{Int}}(\ell x_1+i_1,2y_1+j_1;\dots;\ell x_k+i_k,2y_k+j_k) \\
=
\int_{\RR^k}\prod_{v=1}^k \varphi\left(t_v,\sigma^{-1}\mu_v,j_v\right)
\sigma^{-k}\rho_k^{\operatorname{GUE},\theta}(t_1,\sigma^{-1}\mu_1,j_1;\dots;t_k,\sigma^{-1}\mu_k,j_k)\d \mu_1,\dots,\d \mu_k \\
=
\int_{\RR^k}\prod_{v=1}^k \varphi\left(t_v,\mu_v,j_v\right)
\rho_k^{\operatorname{GUE},\theta}(t_1,\mu_1,j_1;\dots;t_k,\mu_k,j_k)\d \mu_1,\dots,\d \mu_k,
\end{multline}
where~$\rho_k^{\operatorname{GUE,\theta}}$ is the~$k$th correlation function associated to the point process~$\PP_{\operatorname{GUE}}^\theta$. Using that~$\varphi$ is compactly supported and that the integrand is uniformly bounded together with the point-wise convergence a.e. in Theorem~\ref{thm:limit_correlation_function}, the above limit is an application of the dominated convergence theorem.

For a compact subset of~$\Lambda_{\{0,1\}}$ containing no point~$(t,\mu,j)$ with~$t>t_0$ for some~$t_0$, contains at most~$\sum_{t=1}^{t_0} t=\frac{1}{2}t_0(t_0+1)$ number of particles, independent of~$N$. This means that in~\eqref{eq:fredholm}, all terms with~$k>\frac{1}{2}t_0(t_0+1)$ are zero. Hence, summing~\eqref{eq:riemann_sum} over all~$j_v$ and all~$t_v$ proves the corollary.
\end{proof}

\begin{remark}\label{rmk:forget_marks}
If we ignore the marks by identifying the two colors, the limiting point process is the classical GUE-corners process. Indeed, if~$\varphi:\Lambda\to [0,1]$ is a compactly supported continuous functions, but now on~$\Lambda$ instead of~$\Lambda_{\{0,1\}}$, the proof of Corollary~\ref{cor:week_converence} still applies. The only difference is that at the very end as we sum over all~$j_v$, we get
\begin{equation}
\sum_{j_1,\dots,j_k}\rho_{\operatorname{GUE}}^\theta(t_1,\mu_1,j_1;\dots;t_k,\mu_k,j_k)=\rho_{\operatorname{GUE}}(t_1,\mu_1;\dots;t_k,\mu_k).
\end{equation}
This leads us to the limit
\begin{equation}
\lim_{N\to \infty}\EE_{\operatorname{Int}}\left[\prod_v \left(1-\varphi\left(t_v,\frac{u_{s_v}^{t_v}-2 N\tau }{2\sqrt{N}}\right)\right)\right]
=\EE_{\operatorname{GUE}}\left[\prod_v (1-\varphi(t_v,\xi_{s_v}^{t_v}))\right].
\end{equation}
\end{remark}

Marked point processes are naturally connected to the notion of \emph{thinning}. A \emph{thinned point process} is obtained by independently deleting each particle of a given process, with a deletion probability that may depend on its location; see~\cite{BP04, BP06}. As discussed (for a general point process) in~\cite{CG23}, the subset of points in the limiting process~$(\xi_{s,j}^t)_{1\leq s\leq t}$ with~$j=1$ forms a thinning of the GUE-corners process, where the probability of deleting a particle is given by the function~$\theta$. Similarly, selecting the points with~$j=0$ corresponds to a thinning with deletion probability~$1-\theta$. This leads to the following corollary.
\begin{corollary}[Corollary \ref{maincor:1}]\label{cor:thinned}
The restriction of the point process~$(u_{s,j}^t)_{1\leq s\leq t\leq 2\ell N}$ to points with~$j=1$ converges, under the same scaling and in the same sense as in Corollary~\ref{cor:week_converence}, to a thinned GUE–corners process with deletion probability~$\theta$.
\end{corollary}

\section{Asymptotic analysis of the correlation Kernel }\label{sec:steepest}
In this section we prove Theorem \ref{thm:limit_correlation_function}. Before going on to do so, we establish some notation and elaborate on several key notions to this discussion.

{In the model studied in the current paper, the frozen phase consists of four connected components (one in each corner of the Aztec diamond). Consequently, the part of the arctic curve separating the rough phase from the frozen phase is naturally divided into four connected segments. The boundary of those segments are the \textit{turning points}--the points where the arctic curve touches the boundary of the Aztec diamond.
This was proved in \cite{Ber21} by performing a steepest descent analysis of the kernel from Theorem \ref{thm:correlation_kernel} in the bulk. 
Moreover, the geometry of the arctic curve was described through the action function $F$ defined by 
\begin{equation}
F(z,w, u,v)=
\ell\log(z-1)+\left(u-1\right)\log w-v\log z
\end{equation}

Indeed, it was established that the differential $\d F$ has precisely $4$ simple poles at~$q_0$,~$q_{\infty}$,~$p_0$,~$p_{\infty}$ and~$2\ell$ zeros, with  $2$ zeros on each of the $(\ell-1)$ compact ovals $A_k, k\in\{1,\ldots, \ell-1\}$. The location of the remaining $2$ zeros at $\zeta_1$, $\zeta_2$ determines whether the point $(u,v)$  is in: the \textit{rough phase} 
 $(\overline{\zeta_1}=\zeta_2\in \mathcal{R} \text{ and } \zeta_1,\zeta_2 \not\in \bigcup_{k=0}^{\ell-1} A_k)$; the \textit{frozen phase} 
 $(\zeta_1, \zeta_2\in A_0$, the non-compact oval); the \textit{smooth phase} 
 $(\zeta_1, \zeta_2\in A_k$ for $k\in{1,\ldots,\ell-1}$); or the \textit{arctic curve} ($\zeta_1=\zeta_2\in \cup_{k=0}^{\ell-1} A_k$),  
 of the underlying Aztec diamond.

 Each of the four segments separating the frozen phase from the rough phase corresponds to having $\zeta_1=\zeta_2$ in one of the components of $A_0\backslash \{q_0,q_\infty,p_0,p_\infty\}$. In particular, the turning points correspond to $\zeta_1=\zeta_2=q_0; q_1; p_0; p_{\infty}$.} 
The turning point at which we zoom in the process $\{u_{s}^t\}_{1\leq s\leq t\leq2\ell N}$ has $\zeta_1=\zeta_2=q_{\infty}$. It touches the boundary of the Aztec diamond along the line $2\ell x+i =2\ell N$ and hence has coordinates
 $$(\ell x+ i,2y +j)= (2\ell N, 2 \tau  N).$$
Thus the \textit{action function at the turning point} is given by
\begin{equation}\label{eqn:action_fn_tp}
G(z,w)= F(z,w, 2, \tau )=\ell \log(z-1)+\log w -\tau  \log z,
\end{equation}
and since the two zeros of $dF$ coincide with the pole at $q_{\infty}$, $dG(z,w)$ has a simple zero at $q_{\infty}$.

Now we can rewrite the correlation kernel from Theorem \ref{thm:correlation_kernel} in terms of $G(z,w)$. Prior to Theorem \ref{thm:limit_correlation_function} we introduced new coordinates \eqref{eqn:zoom_coordinates},
and associated them to an embedding of the point process in $\Lambda_{\{0,1\}}$.
 For the convenience of the proof we set
\begin{equation}\label{eqn:proof_coordinates}
(t_p, \mu_p, j_p)= (\ell r_p-i_p, \mu_p, j_p),\quad r_p=0,1,\ldots 2N-1,\quad i_p=0,1,\ldots \ell-1,\quad p=1,2,
\end{equation}
and we identify the compact subsets in these variables with the compact subsets of $\Lambda_{\{0,1\}}^2$.
In these coordinates
the correlation kernel of the process from Theorem \ref{thm:correlation_kernel} is given by
\begin{multline} \label{eqn:kerneljdjs}
K_{\operatorname{Int}}(\ell x_1+i_1, 2y_1+j_1; \ell x_2+i_2, 2y_2+j_2 )
\\
=-\mathcal{J}_{s}(\ell r_1- i_1, 2\mu_1+j_1, \ell r_2-i_2, 2\mu_2+j_2)
+\mathcal{J}_{d}(\ell r_1- i_1, 2\mu_1+j_1, \ell r_2-i_2, 2\mu_2+j_2),
\end{multline}
\begin{equation} \label{eqn:Js}
\mathcal{J}_{s}=
\frac{\one_{\ell r_1-i_1>\ell r_2-i_2}}{(2\pi\i)}\int_{\tilde \Gamma_s}
\boldsymbol{\mathcal{M}}_{i_1,j_1,i_2,j_2}(z,w;z,w) 
w^{r_1-r_2}
z^{(\mu_1-\mu_2)\sqrt{N}}
\frac{\d z}{z}
\end{equation}
and
\begin{equation} \label{eqn:Jd}
\mathcal{J}_{d}=
\frac{1}{(2\pi\i)^2}\int_{\tilde \Gamma_s}\int_{\tilde \Gamma_l}
\boldsymbol{\mathcal{M}}_{i_1,j_1,i_2,j_2}(z_1,w_1;z_2,w_2) 
\exp\left({N(G(z_2,w_2)-G(z_1,w_1))}\right)
\frac{w_1^{r_1}}{w_2^{r_2}}
 \frac{z_1^{\mu_1\sqrt{N}}}{z_2^{\mu_2\sqrt{N}}}
 \frac{\d z_2\d z_1}{z_2(z_2-z_1)}.
\end{equation}
where
\begin{equation}\label{eqn:matrix_part_integrand}
\boldsymbol{\mathcal{M}}_{i_1,j_1,i_2,j_2}\left(z_1,w_1;z_2,w_2\right) =\left(\left(\prod_{m=1}^{2i_1}\phi_m(z_1)\right)^{-1}Q(z_1,w_1)Q(z_2,w_2)\prod_{m=1}^{2i_2}\phi_m(z_2)\right)_{j_1+1,j_2+1},
\end{equation}
for $i_1, i_2=0,1,\ldots \ell-1$ and $j_1, j_2=0, 1$, is used for compactness of notation. For readability of the above expressions we disregarded the difference between $y_p=\lfloor N\tau +\sqrt{N}\mu_p\rfloor$ and $N\tau +\sqrt{N}\mu_p$ for $p=1,2$ since it does not introduce complications and is only notationally more involved.

A double contour integral of the form \eqref{eqn:Jd} can be analysed by the \textit{method of steepest descent/ascent} for double contour integrals. In loose terms, given an integral of this form this method assures that if the curves of integration can be deformed appropriately through a saddle point of $G$, the main contribution to the integral as $N\rightarrow\infty$ comes from a neighbourhood of this saddle point. For a more detailed discussion, see \cite{BB23, Oko03, OR03}.

The case of the turning point of the action function is not typical, since the integrand in \eqref{eqn:Jd}, in fact has a pole at the zero of $dG$ at $q_{\infty}$, so the saddle point analysis has to be performed with extra care. Our method in this setting is close to the approach in  \cite{JN06, OR06}, with the main difference that the contour integrals in the setting of this work are over a higher genus Riemann surface.

Most of the analysis in this section is concerned with $\mathcal{J}_d$. It is divided into
into the following subsections:
\begin{enumerate}
\item[\ref{sec:properties_G}:]  Establishing key properties of $G(z,w)$;
\item[\ref{sec:local_analysis}:] Local analysis at $q_{\infty}$ and some global estimates;
\item[\ref{sec:Jd}:]  Asymptotic analysis of $\mathcal{J}_d$;
\item[\ref{sec:Js}:] Asymptotic analysis of $\mathcal{J}_s$;
\item[\ref{sec:proof}:] Proof of Theorem \ref{thm:limit_correlation_function}.
\end{enumerate}

\subsection{Establishing key properties of $G(z,w)$}\label{sec:properties_G}
Due to the exponential dependence of the integrand in $\mathcal{J}_d$ on $N(G(z_2,w_2)-G(z_1,w_1))$, the function $G(z,w)$ governs the asymptotic behaviour of this double-contour integral. In this section we determine the local behaviour of $\exp(G)$ near $q_{\infty}$, the location of the poles and zeros of the differential of $G$, $\d G$, and some properties of the level lines given by $\exp(\im G)=\exp(\im G(q_{\infty}))$. We remark that although $G$ is not well-defined as a function on $\mathcal{R}$, $\d G$, $\exp(G)$ and $\re G$ are, and in fact the analysis to follow only requires the latter three.   

The correspondence of the zeros of the action function to the regions in the Aztec diamond allows us to define the investigated turning point in a more technical manner solely through the action function, $G$, given by the expression \eqref{eqn:action_fn_tp} above:
\begin{definition}\label{def:turning_point}
A point $(\ell x+i,2y+j)=(2\ell N,2\tau N)$ with $\tau \in (0,\ell)$ is said to be a \emph{turning point} if $\d G$ has a simple zero at $q_\infty$. 
\end{definition}
That the above definition of the turning point is well-defined follows from \cite{Ber21}, as explained in the beginning of Section \ref{sec:steepest}. Moreover in the same work it is explained that there is only one turning point on the right most side of the Aztec diamond, that is, there is a unique $\tau \in (0,\ell)$ so that $\d G$ has a simple zero at $q_\infty$. For completeness, we will prove the latter fact in Section \ref{sec:ty_sigma}, see Proposition \ref{prop:exact_values} below. Going forward, we assume that $\tau $ is such that $(2\ell N,2\tau N)$ is a turning point. 

In the discussion to follow we often need to make a choice of local coordinates on $\mathcal{R}$, that is a local chart. With respect to the description of the Riemann surface in Figure \ref{fig:rs_sheets}, there is a canonical choice of local chart away from branch points given by the projection $(z, w(z))\mapsto z$. 
In the instances that we require a local chart in the analysis to follow, we assume this convention unless explicitly stated otherwise.
That is, we identify the point $(z, w(z))\in\mathcal{R}$ (which is not a branch point) with the point $z\in\mathbb{C}$ locally on the corresponding sheet/copy of the complex plane in this representation of $\mathcal{R}$.

To be able to appropriately deform the contours in \eqref{eqn:Jd} in order to study the asymptotics of $\mathcal{J}_d$ it is important to understand the zeros and poles of $\d G$.
\begin{lemma}
\label{lem:zerospoles}
The set of zeros of $\d G$ is given by precisely $2$ zeros on each compact oval $A_k$, $k=1,\dots,\ell-1$, as well as a simple zero at $q_{\infty}$. The set of poles is given by the simple poles at $p_0$, $p_{\infty}$, $q_0$.
\end{lemma}
\begin{proof}
Explicitly we have
$$\d G=\left(\frac{\ell}{z-1}+\frac{w'(z)}{w(z)}-\frac{\tau }{z}\right)\d z.$$
The differential $\d G$ has simple poles exactly at $p_0=(0,1)$, $p_{\infty}=(\infty,1)$, and $q_0=(1,0)$. The fact that $\d G$ has a pole at $(1,0)$ but not at $(1,\infty)$ follows from the observation that $\frac{w'(z)}{w(z)}\d z=\frac{\d w}{w}$ has a simple pole with residue $+\ell$ at $q_0$ while it has a simple pole with residue $-\ell$ at $q_\infty$. 

Since $\d G$ is a meromorphic differential form on a compact Riemann surface, the number of poles tells us the number of zeros of $\d G$. Indeed, by Abel's theorem, 
\begin{equation}
\#\text{\{zeros of $\d G$\}}-\#\text{\{poles of $\d G$\}}=2g-2,
\end{equation}
where $g$ is the genus of $\mathcal R$. Since $g=\ell-1$ (recall our assumption from Section \ref{sec:kernel}) we conclude that $\d G$ cannot have more zeros than the zeros given in the statement. By Definition \ref{def:turning_point}, the constant $\tau $ is defined so that $\d G$ has a zero at $q_\infty$. So, what remains to show is that $\d G$ has two zeros on $A_k$ for each $k=1,\dots,\ell-1$. 

Note that the real part of $G$ is a continuous well-defined function on $\mathcal R$. Since the differential $\d G$ has no poles on the compact oval $A_k$, for each $k=1,\dots,\ell-1$, we get that
\begin{equation}
\re \left(\int_{A_k} \d G\right)=0.
\end{equation}
Moreover, $\d G$ is real on $A_k$. These two properties imply that $\d G$ has, at least, two zeros on $A_k$. See \cite{DK21, Ber21, BB23} for details. 
\end{proof}
\begin{remark}
The structure of the zeros and poles of $\d G$ is a consequence of the fact that $\d G$ is a, so-called, \emph{imaginary normalized differential} on a, so-called, \emph{M-curve} with simple poles only on $A_0$. See \cite[Lemma 4.2]{Kri13}.
\end{remark}

We mention several properties of the action function we need in our saddle point analysis.
\begin{corollary}
 \label{lem:action_taylor}       
We have that
\begin{equation}\label{eqn:action_fn_at_other_poles}
\re G(0,1)=+\infty, \quad \re G(\infty,1)=+\infty, \quad\text{and} \quad \re G(1,0)= -\infty.
\end{equation}
Furthermore in a neighbourhood of $(z,w)=q_\infty$, with respect to the canonical choice of coordinates,
we denote the function $z\mapsto G(z,w)$ by $G(z)$ (with some abuse of notation). Then we have that
\begin{equation}
 G'(1)=0, \quad\text{and} \quad G''(1)=\sigma^2\quad \text{for some } \sigma^2>0.
\end{equation}
\end{corollary}

\begin{proof}
Explicitly we have that
$$\re G= \ell \log|z-1|+\log|w|-\tau \log |z|$$
and all three identities in \eqref{eqn:action_fn_at_other_poles} follow immediately from this expression, recall $\tau\in (0, \ell)$.

Using the local chart described by $(z,w)\mapsto z$ near $q_{\infty}$, due to Lemma \ref{lem:zerospoles}, $G'(z)$ has a zero of order exactly $1$ at $q_{\infty}$ and in particular $G''(1)\not=0$. 
To prove that $G''(1)>0$ we observe that $G$ is real-valued when $z\in(0,\infty)$. As $dG$ has no zeros or poles on $(1,\infty)$ due to Lemma \ref{lem:zerospoles}, it follows that $G$ is monotone on this interval and since $G(\infty)=+\infty$, $G$ is increasing and so $G''(1)>0$.
\end{proof}

\begin{remark}
It is not important for our argument, but we can explicitly compute 
$$G(1)=\log\left( \prod_{m=1}^\ell(1+\alpha_m^{-1}\beta_m)(1+\alpha_{m+1}\beta_m^{-1})\right),$$
see \eqref{eq:at_1_2} in Lemma \ref{lem:at_1}.
 and $G''(1)$, which is computed in Proposition \ref{prop:exact_values}.
\end{remark}
In the proof of Corollary \ref{lem:action_taylor} we showed that $G=\re G$ is monotone increasing along the interval $(1,\infty)$ with respect to the local chart described before and we denote the curve corresponding to this interval $(q_{\infty}, p_{\infty})$ and we could in the same way see that $\re G$ is monotone increasing on the curve going along the interval from $q_{\infty}$ to $p_0$; these curves are usually referred to as \textit{curves of steepest ascent} for $ G$. Likewise the \textit{curves of steepest descent} of $G$ have the properties that $\re G$ is decreasing along them and $\exp(\im G)$ is constant along them, and these properties characterise these curves. They are important for the upcoming double-contour integral analysis and we describe them in the next lemma. 

\begin{lemma}
 \label{lemma:curvesofsteepest}
There are two curves of steepest descent of $G$ starting at $q_{\infty}$ and ending at $q_0$ which together form a closed contour 
such that $p_0$ lies within its interior and $p_{\infty}$-- within its exterior. The argument of $z$ in a parameterisation of the curve changes by exactly $4\pi$ as we tread around it once. In particular this contour is bounded away from $p_0$ and $p_\infty$, has finite length, and treads around $p_0$ exactly once on $\mathcal{R}$. We denote it $\Gamma_{des}$. 
\end{lemma}

\begin{proof}
Now due to Corollary \ref{lem:action_taylor}, near $q_{\infty}$, in the local chart, described in the beginning of this section,
$$G(z)=G(q_{\infty})+\frac{\sigma^2}{2}(z-1)^2+\mathcal{O}((z-1)^3)$$
and in particular locally the level lines of 
$\exp(\im G)$ are close to the corresponding level lines of 
\\ $\exp(\im \frac{\sigma^2}{2}(z-1)^2)$, given by $\re z=1,  \im z=0$. Thus there are $4$ level lines of $\exp(\im G)$ starting at $q_{\infty}$ in 4 different directions.
 \item 

Since we are constructing a curve starting at $q_{\infty}$ along the level line of $\exp(\im G)$ on which $\re G$ is strictly decreasing, due to \eqref{eqn:action_fn_at_other_poles}, this curve is bounded away from $p_0$ and $p_{\infty}$. 
The two curves of steepest descent are symmetric with respect to the copy of the real line (coming from $(z,w)\mapsto z$) on each sheet of the surface, since $\re {G}(z,w)= \re G(\overline{z}, \overline{w})$ and so we restrict our attention to one of them. 
We argue that there is a curve of steepest descent starting at $q_{\infty}$ and treading first locally up along the imaginary axis and ends at $q_0$ so that it is contained in the white region of the underlying Riemann surface in Figure
\ref{fig:rs_sheets}, see Figure \ref{fig:schematic_curve}. 
Namely whenever the curve intersects a copy of the real line either:
\begin{enumerate}
\item it passes through a cut in which case we choose to continue treading along the the other sheet of the Riemann surface in the white region;
\item it passes through a compact oval in which case, since $\im G= constant$ on the compact oval we can let the curve thread towards a zero on that compact oval along the real line. Since for at least one of the zeros, $\zeta$, on the compact oval it would hold that $G''(\zeta)>0$ (with respect to the canonical local chart) we can eventually choose the contour of steepest descent at that zero to continue to tread inside the white region in Figure \ref{fig:rs_sheets}, see Figure \ref{fig:schematic_curve}; 
\item it intersects the non-compact oval on the sheet containing $q_0$ and then by the same argument as in the proof of Lemma \ref{lem:action_taylor}, the contour of steepest descent threads along the corresponding direction of the real line towards $q_0$ and terminates there. 
\end{enumerate}
As the so-constructed curve does not self-intersect (since the real part is monotone), any of the above cases can only emerge finitely many times and the curve eventually terminates at $q_0$, see Figure  \ref{fig:schematic_curve}.
Thus on this piece of the contour of steepest descent the angle with respect to $p_0$ changes by exactly $2\pi$ and the curve constructed from the union of its reflection and itself, threads around $p_0$ exactly once on $\mathcal{R}$.
\color{black}
\end{proof}

\begin{figure}
\begin{tikzpicture}[>=stealth, scale=0.84]

  \begin{scope}[shift={(-5cm,0)}]
  
    \begin{scope}
      \clip (-6,-2.8) rectangle (3,0);
      \fill[gray!10] (-6,-2.9) rectangle (3,0);
    \end{scope}
    
    \draw[thick] (-6,0) -- (-5,0);
    \draw[thick] (-4,0) -- (-3,0);
    \draw[thick] (-2.5,0) -- (-1,0);
    \draw[thick] (0,0) -- (3,0);
    
    \draw[dashed, thick] (-5,0) -- (-4,0);
    \draw[dashed, thick] (-3,0) -- (-2.5,0);
    \draw[dashed, thick] (-1,0) -- (0,0);
    
    \fill (0,0) circle (2pt) node[below] {$p_0$};
    \fill (1,0) circle (2pt) node[below] {$q_{\infty}$};
    \fill (3,0) circle (2pt) node[below] {$p_{\infty}$};
    
    \draw[red, thick, decoration={markings, mark=at position 0.5 with {\arrow{>}}}, postaction={decorate}] 
          (1,0) .. controls (1,0.7) and (1.2,1.4) .. (0.9,1.9) 
          .. controls (0.5,2.2) and (0.1,1.8) .. (-0.2,2.1) 
          .. controls (-0.6,2.4) and (-0.9,2.0) .. (-1.1,1.6) 
          .. controls (-1.3,1.1) and (-1.2,0.5) .. (-1.5,0);
    
    \draw[red, thick, decoration={markings, mark=at position 0.7 with {\arrow{>}}}, postaction={decorate}] 
          (-1.5,0) -- (-2,0);
    
    \draw[red, thick, decoration={markings, mark=at position 0.5 with {\arrow{>}}}, postaction={decorate}] 
          (-2,0) .. controls (-2,0.6) and (-1.8,1.2) .. (-2.1,1.7) 
          .. controls (-2.4,2.1) and (-2.9,1.9) .. (-3.2,2.2) 
          .. controls (-3.6,2.5) and (-3.9,2.1) .. (-4.1,1.5) 
          .. controls (-4.3,1.0) and (-4.6,0.6) .. (-4.5,0);
    
    \draw[orange, thick, dashed, decoration={markings, mark=at position 0.5 with {\arrow{<}}}, postaction={decorate}] 
          (1,0) .. controls (1,-0.7) and (1.2,-1.4) .. (0.9,-1.9) 
          .. controls (0.5,-2.2) and (0.1,-1.8) .. (-0.2,-2.1) 
          .. controls (-0.6,-2.4) and (-0.9,-2.0) .. (-1.1,-1.6) 
          .. controls (-1.3,-1.1) and (-1.2,-0.5) .. (-1.5,0);
    
    \draw[orange, thick, dashed, decoration={markings, mark=at position 0.3 with {\arrow{<}}}, postaction={decorate}] 
          (-1.5,0) -- (-2,0);
    
    \draw[orange, thick, dashed, decoration={markings, mark=at position 0.5 with {\arrow{<}}}, postaction={decorate}] 
          (-2,0) .. controls (-2,-0.6) and (-1.8,-1.2) .. (-2.1,-1.7) 
          .. controls (-2.4,-2.1) and (-2.9,-1.9) .. (-3.2,-2.2) 
          .. controls (-3.6,-2.5) and (-3.9,-2.1) .. (-4.1,-1.5) 
          .. controls (-4.3,-1.0) and (-4.6,-0.6) .. (-4.5,0);
    
    \fill (-1.5,0) circle (1.5pt);
    \fill (-2,0) circle (1.5pt);
    \fill (-4.5,0) circle (1.5pt);
    
  \end{scope}
  
  \begin{scope}[shift={(5cm,0)}]
  
    \begin{scope}
      \clip (-6,0) rectangle (3,2.8);
      \fill[gray!10] (-6,0) rectangle (3,2.9);
    \end{scope}
    
    \draw[thick] (-6,0) -- (-5,0);
    \draw[thick] (-4,0) -- (-3,0);
    \draw[thick] (-2.5,0) -- (-1,0);
    \draw[thick] (0,0) -- (3,0);
    
    \draw[dashed, thick] (-5,0) -- (-4,0);
    \draw[dashed, thick] (-3,0) -- (-2.5,0);
    \draw[dashed, thick] (-1,0) -- (0,0);
    
    \fill (0,0) circle (2pt) node[below] {$p_0$};
    \fill (1,0) circle (2pt) node[below] {$q_0$};
    \fill (3,0) circle (2pt) node[below] {$p_{\infty}$};
    
    \draw[red, thick, decoration={markings, mark=at position 0.5 with {\arrow{>}}}, postaction={decorate}] 
          (-4.5,0) .. controls (-4.5,-0.5) and (-4.3,-1.0) .. (-4.1,-1.2) 
          .. controls (-3.9,-1.3) and (-3.8,-1.1) .. (-3.7,-0.8) 
          .. controls (-3.7,-0.5) and (-3.8,-0.2) .. (-3.8,0);
    
    \draw[red, thick, decoration={markings, mark=at position 0.7 with {\arrow{>}}}, postaction={decorate}] 
          (-3.8,0) -- (-3.2,0);
    
    \draw[red, thick, decoration={markings, mark=at position 0.5 with {\arrow{>}}}, postaction={decorate}] 
          (-3.2,0) .. controls (-3.2,-0.5) and (-3.0,-0.9) .. (-2.6,-1.1) 
          .. controls (-2.2,-1.3) and (-1.9,-1.2) .. (-1.6,-1.0) 
          .. controls (-1.3,-0.8) and (-1.0,-1.1) .. (-0.7,-1.2) 
          .. controls (-0.3,-1.3) and (0.2,-1.0) .. (0.5,-0.7) 
          .. controls (0.8,-0.4) and (1,-0.3) .. (1,0);
    
    \draw[orange, thick, dashed, decoration={markings, mark=at position 0.5 with {\arrow{<}}}, postaction={decorate}] 
          (-4.5,0) .. controls (-4.5,0.5) and (-4.3,1.0) .. (-4.1,1.2) 
          .. controls (-3.9,1.3) and (-3.8,1.1) .. (-3.7,0.8) 
          .. controls (-3.7,0.5) and (-3.8,0.2) .. (-3.8,0);
    
    \draw[orange, thick, dashed, decoration={markings, mark=at position 0.3 with {\arrow{<}}}, postaction={decorate}] 
          (-3.8,0) -- (-3.2,0);
    
    \draw[orange, thick, dashed, decoration={markings, mark=at position 0.5 with {\arrow{<}}}, postaction={decorate}] 
          (-3.2,0) .. controls (-3.2,0.5) and (-3.0,0.9) .. (-2.6,1.1) 
          .. controls (-2.2,1.3) and (-1.9,1.2) .. (-1.6,1.0) 
          .. controls (-1.3,0.8) and (-1.0,1.1) .. (-0.7,1.2) 
          .. controls (-0.3,1.3) and (0.2,1.0) .. (0.5,0.7) 
          .. controls (0.8,0.4) and (1,0.3) .. (1,0);
             
    \fill (-4.5,0) circle (1.5pt);
    \fill (-3.8,0) circle (1.5pt);
    \fill (-3.2,0) circle (1.5pt);
    \fill (1,0) circle (1.5pt);
    
  \end{scope}

\end{tikzpicture}
\caption{Schematic representation of the curve of steepest descent on ~$\mathcal R$ constructed in the proof of Lemma \ref{lemma:curvesofsteepest} in red together with its reflection in orange. The image on the left represents the curve on the sheet of ~$\mathcal R$ containing $q_{\infty}$ and the image on the right represents the part of the curve on the sheet containing ~$q_0$. The red curve was constructed to exclusively tread in the non-shaded region in the picture.} \label{fig:schematic_curve}
\end{figure}
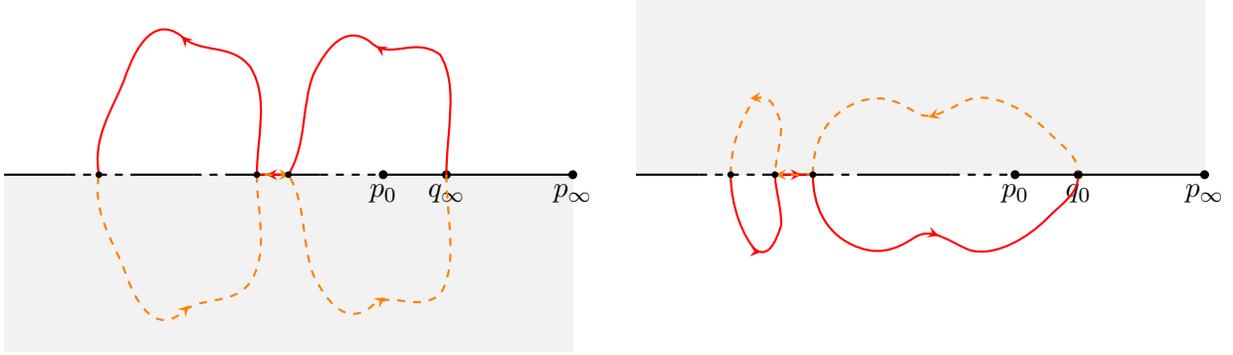

We now have sufficient information about the action function to proceed with studying the double-contour integrals. 

\subsection{Local analysis at $q_{\infty}$ and some global estimates}\label{sec:local_analysis}
In this section, we compute the asymptotic behaviour of the integrand in the double contour integral in a neighbourhood of~$q_\infty$. We also prove global bounds necessary to prove that the main contribution from the double contour integral comes from a neighbourhood of~$q_\infty$. We first study separately the matrix part of the  integrand, $\boldsymbol{\mathcal{M}}_{i_1,j_1,i_2,j_2}$, depending only on the $i_k, j_k,$ for $k=1,2$.
 After that we perform the local analysis for the rest of the integrand (which depends on $N$).

We will use the following identities that can be found either in~\cite{Ber21} or~\cite{BNR25}. For completeness, we state and prove these identities here. 
\begin{lemma}\label{lem:at_1}
The following identities hold:
\begin{equation}\label{eq:at_1_1}
Q(q_\infty)=\frac{1}{1+\beta_\ell^{-1} \alpha_1}
\begin{pmatrix}
1 \\
\alpha_1
\end{pmatrix}
\begin{pmatrix}
1 & \beta_\ell^{-1}
\end{pmatrix},
\end{equation}
\begin{equation}\label{eq:at_1_2}
\left.(z-1)^\ell w\right|_{(z,w)=q_{\infty}}=\prod_{m=1}^\ell(1+\alpha_m^{-1}\beta_m)(1+\alpha_{m+1}\beta_m^{-1}),
\end{equation}
\begin{equation}\label{eq:at_1_2_q0}
\left.(z-1)^{-\ell} w\right|_{(z,w)=q_{0}}=\left(\prod_{m=1}^\ell(1+\alpha_m^{-1}\beta_m)(1+\alpha_{m+1}\beta_m^{-1})\right)^{-1},
\end{equation}
and
\begin{equation}\label{eq:at_1_3}
\left.(z-1)^{i-i'}\prod_{m=2i'+1}^{2i}\phi_m(z)\right|_{z=1}
=\prod_{m=i'+1}^i(1+\alpha_m^{-1}\beta_m)\prod_{m=i'+1}^{i-1}(1+\alpha_{m+1}\beta_m^{-1})
\begin{pmatrix}
1 \\
\alpha_{i'+1}
\end{pmatrix}
\begin{pmatrix}
1 & \beta_i^{-1}
\end{pmatrix},
\end{equation}
for~$0\leq i'<i\leq \ell$.
\end{lemma}
\begin{proof}
We begin by proving~\eqref{eq:at_1_3}. Recall the definition~\eqref{eq:transition_matrices} of~$\phi_m$. For~$m=1,\dots,\ell$,
\begin{equation}
\phi_{2m-1}(1)=
\begin{pmatrix}
1 \\
\alpha_m
\end{pmatrix}
\begin{pmatrix}
1 & \alpha_m^{-1}
\end{pmatrix}
\quad \text{and} \quad 
(z-1)\phi_{2m}(z)|_{z=1}=
\begin{pmatrix}
1 \\
\beta_m
\end{pmatrix}
\begin{pmatrix}
1 & \beta_m^{-1}
\end{pmatrix}.
\end{equation}
Multiplying these equalities for~$m=i'+1,\dots,i$, we obtain the fourth equality in the statement.

Taking the trace of~\eqref{eq:at_1_3} with~$i'=0$ and~$i=\ell$, we get
\begin{equation}\label{eq:at_1_trace}
\Tr\left((z-1)^\ell \Phi(z)|_{z=1}\right)=\prod_{m=1}^\ell(1+\alpha_m^{-1}\beta_m)(1+\alpha_{m+1}\beta_m^{-1}).
\end{equation}
On the other hand,
\begin{equation}\label{eq:trace_eigenvalues}
\Tr\left((z-1)^\ell \Phi(z)\right)=(z-1)^\ell(w_1(z)+w_2(z)),
\end{equation}
where~$(z,w_1(z))$ and~$(z,w_2(z))$ are local coordinates in a neighborhood of~$q_\infty=(1,\infty)$ and~$q_0=(1,0)$, respectively. We are using here that~$w_1(z)$ and~$w_2(z)$ are the two eigenvalues of~$\Phi(z)$ close to~$z=1$. If we take~$z\to 1$ in~\eqref{eq:trace_eigenvalues}, the right-hand side tends to~$(z-1)^\ell w_1(z)$, and combining this limit with~\eqref{eq:at_1_trace}, leads to~\eqref{eq:at_1_2}.

Since $\det \Phi(z)=1$, we have $w_1(z)w_2(z)=1$ where $w_1$ and $w_2$ are as above. So, \eqref{eq:at_1_2_q0} follows from \eqref{eq:at_1_2} by taking $z\to 1$ in the equality
\begin{equation}
1=(z-1)^\ell w_1(z) (z-1)^{-\ell} w_2(z).
\end{equation}

To prove~\eqref{eq:at_1_1}, we first note that~$\partial_w\det(wI-\Phi(z))=2w-\Tr\Phi(z)$. Recall also~\eqref{eq:adjuagate_matrix}. It follows from~\eqref{eq:at_1_2},~\eqref{eq:at_1_trace}, and~\eqref{eq:at_1_3} with~$i'=0$ and~$i=\ell$, that
\begin{equation}
\left.\frac{\adj\left(wI-\Phi(z)\right)}{\partial_w\det(wI-\Phi(z))}\right|_{(z,w)=q_\infty}=\adj\left(I-\frac{1}{1+\alpha_1\beta_\ell^{-1}}
\begin{pmatrix}
1 & \beta_\ell^{-1} \\
\alpha_1 & \alpha_1\beta_\ell^{-1}
\end{pmatrix}\right)
=
\frac{1}{1+\alpha_1\beta_\ell^{-1}}
\begin{pmatrix}
1 & \beta_\ell^{-1} \\
\alpha_1 & \alpha_1\beta_\ell^{-1}
\end{pmatrix},
\end{equation}
which proves~\eqref{eq:at_1_1}. 
\end{proof}

We continue to provide the leading behavior of the matrix part of the integrand.
\begin{lemma}\label{lem:matrix_local}
For $(z_1,w_1)$ and $(z_2,w_2)$ in a neighbourhood of $q_{\infty}$, let
~$z_i=1+\frac{Z_i}{N^\frac{1}{2}}$,~$i=1,2$, and set
\begin{equation}
\tilde g(\ell x+i,2y+j)=\prod_{m=1}^i(1+\alpha_m^{-1}\beta_m)(1+\alpha_{m+1}\beta_m^{-1})\alpha_{i+1}^{-j},
\end{equation}
and let~$\nu$ be given in~\eqref{eq:period_function}. Then, as~$N\to \infty$, 
\begin{multline} \label{eqn:matrix_precise}
\boldsymbol{\mathcal{M}}_{i_1,j_1,i_2,j_2}\left(z_1,w_1;z_2,w_2\right)
= \nu(\ell-i_2,j_2)\frac{Z_1^{i_1}}{Z_2^{i_2}}\frac{\tilde g(\ell x_2+i_2,2y_2+j_2)}{\tilde g(\ell x_1+i_1,2y_1+j_1)}\frac{N^{\frac{i_2}{2}}}{N^{\frac{i_1}{2}}}\left(1+\Ordo\left(\frac{|Z_1|+|Z_2|}{N^{\frac{1}{2}}}\right)\right),
\end{multline}
where, assuming $|Z_k|\leq N^{\frac{1}{2}}$, the implicit constant in the remainder term is independent of $Z_k, x_k, i_k, y_k, j_k$, $k=1,2$.
\end{lemma}
The gauge function~$\tilde g$ in the previous lemma is part of the gauge function~$g$ in~\eqref{eq:gauge_function}.
\begin{proof}
To prove the lemma, we prove that
\begin{multline}\label{eq:at_1_full}
\left.\frac{(z_2-1)^{i_2}}{(z_1-1)^{i_1}}\left(\prod_{m=1}^{2i_1}\phi_m(z_1)\right)^{-1}Q(z_1,w_1)Q(z_2,w_2)\prod_{m=1}^{2i_2}\phi_m(z_2)\right|_{(z_k,w_k)=q_\infty, k=1,2} \\
=\frac{\prod_{m=1}^{i_2}(1+\alpha_m^{-1}\beta_m)(1+\alpha_{m+1}\beta_m^{-1})}
{\prod_{m=1}^{i_1}(1+\alpha_m^{-1}\beta_m)(1+\alpha_{m+1}\beta_m^{-1})}
\begin{psmallmatrix}
1 & 0 \\
0 & \alpha_{i_1+1}
\end{psmallmatrix}
\frac{
\begin{psmallmatrix}
1 & \alpha_{i_2+1}\beta_{i_2}^{-1} \\
1 & \alpha_{i_2+1}\beta_{i_2}^{-1}
\end{psmallmatrix}}
{1+\alpha_{i_2+1}\beta_{i_2}^{-1}}
\begin{psmallmatrix}
1 & 0 \\
0 & \alpha_{i_2+1}^{-1}
\end{psmallmatrix}.
\end{multline}
The~$(j_1+1)(j_2+1)$-entry of the right-hand side is equal to the leading term on the right-hand side of the equality in the statement. Since the expression on the left-hand side of the equality in the statement is rational in~$z_1$ and~$z_2$, the result of the lemma follows from a Taylor expansion.

We begin with the factor depending on~$(z_2,w_2)$. It follows from Lemma~\ref{lem:at_1} that
\begin{equation}\label{eq:at_1_factor_1}
\left.(z_2-1)^{i_2}Q(z_2,w_2)\prod_{m=1}^{2i_2}\phi_m(z_2)\right|_{(z_2,w_2)=q_\infty}=
\prod_{m=1}^{i_2}(1+\alpha_m^{-1}\beta_m)\prod_{m=1}^{i_2-1}(1+\alpha_{m+1}\beta_m^{-1})
\begin{pmatrix}
1 \\
\alpha_{1}
\end{pmatrix}
\begin{pmatrix}
1 & \beta_{i_2}^{-1}
\end{pmatrix}.
\end{equation}

We continue with the factor depending on~$(z_1,w_2)$. By definition~\eqref{eq:adjuagate_matrix} of~$Q$, it follows that for~$(z,w)\in \mathcal R$,~$\Phi(z)Q(z,w)=wQ(z,w)$. This implies that
\begin{equation}
(z_1-1)^{-i_1}\left(\prod_{m=1}^{2i_1}\phi_m(z_1)\right)^{-1}Q(z_1,w_1)
=\frac{(z_1-1)^{\ell-i_1}}{(z_1-1)^\ell w_1}\prod_{m=2i_1+1}^{2\ell}\phi_m(z_1)Q(z_1,w_1).
\end{equation}
We evaluate the previous equality at~$(z_1,w_1)=q_\infty$. Using Lemma~\ref{lem:at_1}, we get that the right-hand side is equal to
\begin{multline}\label{eq:at_1_factor_2}
\frac{\prod_{m=i_1+1}^\ell(1+\alpha_m^{-1}\beta_m)\prod_{m=i_1+1}^{\ell-1}(1+\alpha_{m+1}\beta_m^{-1})}
{\prod_{m=1}^\ell(1+\alpha_m^{-1}\beta_m)(1+\alpha_{m+1}\beta_m^{-1})}
\begin{pmatrix}
1 \\
\alpha_{i_1+1}
\end{pmatrix}
\begin{pmatrix}
1 & \beta_\ell^{-1}
\end{pmatrix} \\
=
\frac{1}{\prod_{m=1}^{i_1}(1+\alpha_m^{-1}\beta_m)(1+\alpha_{m+1}\beta_m^{-1})}\frac{1}{1+\alpha_1\beta_\ell^{-1}}
\begin{pmatrix}
1 \\
\alpha_{i_1+1}
\end{pmatrix}
\begin{pmatrix}
1 & \beta_\ell^{-1}
\end{pmatrix}.
\end{multline}

Multiplying~\eqref{eq:at_1_factor_1} and~\eqref{eq:at_1_factor_2} together with the equality 
\begin{equation}
\begin{pmatrix}
1 \\
\alpha_{i_1+1}
\end{pmatrix}
\begin{pmatrix}
1 & \beta_{i_2}^{-1}
\end{pmatrix}
=
\begin{pmatrix}
1 & 0 \\
0 & \alpha_{i_1+1}
\end{pmatrix}
\begin{pmatrix}
1 & \alpha_{i_2+1}\beta_{i_2}^{-1} \\
1 & \alpha_{i_2+1}\beta_{i_2}^{-1}
\end{pmatrix}
\begin{pmatrix}
1 & 0 \\
0 & \alpha_{i_2+1}^{-1}
\end{pmatrix}
\end{equation}
yield~\eqref{eq:at_1_full}.
\end{proof}

We also need to provide uniform bounds on the same quantity that we investigated locally around~$q_\infty$ in the previous lemma. For that we have the following lemma.
\begin{lemma} \label{lem:global_matrix_meromorphicity}
The 1-forms
\begin{equation}
\mathcal R\ni (z_k,w_k)\mapsto \boldsymbol{\mathcal{M}}_{i_1,j_1,i_2,j_2}\left(z_1,w_1;z_2,w_2\right) \frac{\d z_1 \d z_2}{z_2(z_2-z_1)}, 
\end{equation}
for $k=1,2$ are meromorphic with possible poles only at $p_\infty$, $p_0$, $q_\infty$, $q_0$, and $(z_1,w_1)=(z_2,w_2)$. 
\end{lemma}
\begin{proof}
As a function on $\mathcal R$, the matrix $\adj(w_kI-\Phi(z_k))$ can only have poles at $p_\infty$, $p_0$, $q_\infty$ and $q_0$. Since $\mathcal R$ is non-singular, the 1-form $\d z_k/\partial_{w_k}\det(w_kI-\Phi(z_k))$ is holomorphic. Indeed, both $\partial_{w_k}\det(w_kI-\Phi(z_k))$ and $\d z_k$ have simple zeros at the branch points, so they cancel each other, and $\partial_{w_k}\det(w_kI-\Phi(z_k))$ has no other zeros. Or differently put, the identity
\begin{equation}
\frac{\d z_k}{\partial_{w_k}\det (w_kI-\Phi(z_k))}=-\frac{\d w_k}{\partial_{z_k}\det (w_kI-\Phi(z_k))},
\end{equation}
which is derived by differentiating $\det(w_kI-\Phi(z_k))=0$, implies that the 1-form is analytic also at the branch points. 

Since the 1-forms from the statement are simply the product of $\adj(w_kI-\Phi(z_k))$, $\d z_k/\partial_{w_k}(w_kI-\Phi(z_k))$, $\prod_{m=1}^{2i}\phi_m(z_k)$ or $\left(\prod_{m=1}^{2i'}\phi(z_k)\right)^{-1}$, and $\frac{1}{z_2(z_2-z_1)}$, what remains to show is that the forms have no pole at $z_1=z_2$ if $w_1\neq w_2$. 

If $(z,w)\in \mathcal R$, then, by definition of $Q$ \eqref{eq:adjuagate_matrix}, 
\begin{equation}
wQ(z,w)=\Phi(z)Q(z,w)=Q(z,w)\Phi(z).
\end{equation}
So, if $(z,w_1), (z,w_2)\in \mathcal R$, 
\begin{equation}
w_1Q(z,w_1)Q(z,w_2)=Q(z,w_1)\Phi(z)Q(z,w_2)=w_2Q(z,w_1)Q(z,w_2).
\end{equation}
In particular, if $w_1\neq w_2$, then
\begin{equation}
Q(z,w_1)Q(z,w_2)=0,
\end{equation}
and, hence,
\begin{equation}
\mathcal M_{i_1,j_1,i_2,j_2}(z,w_1;z,w_2)=0.
\end{equation} 
This shows that the 1-forms have removable singularities at $z_1=z_2$ if $w_1\neq w_2$.
\end{proof}
\begin{remark}
The previous lemma tells us that we do not need to consider a residue at $z_1=z_2$ if the corresponding points on $\mathcal R$ lies on different sheets as we deform the contours in the double contour integral \eqref{eqn:Jd}. 
\end{remark}

We have now dealt with the matrix part of the integrand in \eqref{eqn:Jd}.
We proceed to investigate the scalar part of the integrand that also depends on $N$. We prove the following local estimate.

\begin{lemma}\label{lemma:local_action_function}
Let~$z_k=1+\frac{Z_k}{N^\frac{1}{2}}$,~$k=1,2$. Then, as~$N\to \infty$, for $(z_1,w_1)$,  $(z_2,w_2)$ in the neighbourhood of $q_{\infty}$, we have that
\begin{multline} \label{eqn:integrand_near_q_infinity}
\exp\left({N(G(z_2,w_2)-G(z_1,w_1))}\right)
\frac{w_1^{r_1}}{w_2^{r_2}}
 \frac{z_1^{\mu_1\sqrt{N}}}{z_2^{\mu_2\sqrt{N}}}
 \frac{1}{z_2(z_2-z_1)}
\\= N^{\frac{1}{2}} \frac{\left(1+\mathcal{O}\left( \frac{|Z_1|^3+|Z_2|^3}{N^{\frac{1}{2}}} \right)\right)}{Z_2-Z_1} 
 N^{\frac{\ell r_1}{2}-\frac{\ell r_2}{2}} 
B^{ r_1- r_2} 
\frac{Z_2^{\ell r_2}}{Z_1^{\ell r_1}}
\exp\left(\frac{\sigma^2}{2}(Z_2^2-Z_1^2)+Z_1\mu_1-Z_2\mu_2\right),
\end{multline}
where $B=\left(\prod_{m=1}^\ell(1+\alpha_m^{-1}\beta_m)(1+\alpha_{m+1}\beta_m^{-1})\right)$ and the implicit constant in the remainder term is uniform for compact subsets of $\Lambda_{\{0,1\}}^2$. 
 For $|Z_1|, |Z_2| \leq N^{\frac{1}{12}}$, the remainder term
 is uniformly $\mathcal{O}(N^{-\frac{1}{4}})$, as $N\rightarrow\infty$.
\end{lemma}

Since the local estimates necessary to prove this statement are useful in the asymptotic analysis we include them in a separate lemma. 
\begin{lemma}\label{lemma:expansions}
Let~$z_k=1+\frac{Z_k}{N^\frac{1}{2}}$,~$k=1,2$, with $|Z_k|< N^{\frac{1}{2}}$ (on the canonical local chart in this text, at a neighbourhood of $q_{\infty}$). Then, as~$N\to \infty$, we have that
\begin{equation} \label{eqn:local_action_contribution}
\exp(NG(z_k,w_k))= \exp\left(NG(1)+\frac{\sigma^2}{2}Z_k^2+\mathcal{O}\left(\frac{|Z_k|^3}{N^{\frac{1}{2}}}\right)\right),
\end{equation}
\begin{equation} \label{eqn:local_w_contribution}
w_k^{r_k}=
Z_k^{-\ell r_k} 
N^{\frac{\ell r_k}{2}} 
\left(B+\mathcal{O}\left(\frac{|Z_k|}{N^{\frac{1}{2}}}\right)\right)^{r_k}
=
Z_k^{-\ell r_k} 
N^{\frac{\ell r_k}{2}} 
B^{r_k}\left(1+\mathcal{O}\left(\frac{r_k |Z_k|}{N^{\frac{1}{2}}}\right)\right),
\end{equation}
and
\begin{equation} \label{eqn:local_z_contribution}
z_i^{\sqrt{N}\mu_k}=\left(1+\frac{Z_k}{ \sqrt{N}}\right)^{\mu_k\sqrt{N}}=\exp\left(\mu_k Z_k+\mathcal{O}\left(\frac{\mu_k |Z_k|^2}{N^{\frac{1}{2}}}\right)\right),
\end{equation}
where $B$ is as in Lemma \ref{lemma:local_action_function}.
\end{lemma}

\begin{proof}[Proof of Lemma \ref{lemma:expansions}]
The first identity \eqref{eqn:local_action_contribution} follows from Lemma \ref{lem:action_taylor}.
Since $(z-1)^{\ell}w$ is analytic at $q_{\infty}$ as a function of $z$ and due to Lemma \ref{lem:at_1}, $(z-1)^{\ell} w=B+\mathcal{O}(z-1)$ it follows that 
\begin{equation}
w_k^{r_k}=
Z_k^{-\ell r_k} 
N^{\frac{\ell r_k}{2}} 
\left(B+\mathcal{O}\left(\frac{|Z_k|}{N^{\frac{1}{2}}}\right)\right)^{r_k}
=
Z_i^{-\ell r_k} 
N^{\frac{\ell r_k}{2}} 
B^{r_k}\left(1+\mathcal{O}\left(\frac{r_k|Z_k|}{N^{\frac{1}{2}}}\right)\right).
\end{equation}
and \eqref{eqn:local_w_contribution} follows. Finally \eqref{eqn:local_z_contribution} is a consequence of Taylor expanding the logarithm.
\end{proof}

Lemma \ref{lemma:local_action_function} now follows immediately.
\begin{proof}[Proof of Lemma \ref{lemma:local_action_function}]

Plugging in \eqref{eqn:local_action_contribution}, \eqref{eqn:local_w_contribution}, \eqref{eqn:local_z_contribution} in the left-hand side of \eqref{eqn:integrand_near_q_infinity} given the prescribed change of variables yields that
\begin{multline} 
\exp\left({N(G(z_2,w_2)-G(z_1,w_1))}\right)
\frac{w_1^{r_1}}{w_2^{r_2}}
 \frac{z_1^{\mu_1\sqrt{N}}}{z_2^{\mu_2\sqrt{N}}}
 \frac{1}{z_2(z_2-z_1)}
\\
= N^{\frac{1}{2}} \frac{(1+\mathcal{O}\left( N^{-\frac{1}{2}}
\sum_{k=1}^{2} (\mu_k|Z_k|^2 +r_k |Z_k|+|Z_k|^3) \right)}{Z_2-Z_1} 
N^{\frac{\ell r_1}{2}-\frac{\ell r_2}{2}} 
B^{ r_1- r_2} 
\\ \times 
\frac{Z_2^{\ell r_2}}{Z_1^{\ell r_1}}
\exp\left(\frac{\sigma^2}{2}(Z_2^2-Z_1^2)+Z_1\mu_1-Z_2\mu_2\right),\quad \text{as}\quad N\rightarrow\infty.
\end{multline} 
We deduce the result. 
\end{proof}
We also need the following lemma in a neighbourhood of $q_{0}$. 
\begin{lemma}\label{lemma:G_at_q0}
The function 
$\exp(G(z,w))= (z-1)^{\ell} w z^{-\tau }$ has a zero of order $2\ell$ at $q_0$. Hence in a neighbourhood of $q_0$ setting $z=1+\frac{Z}{\sqrt{N}}$, with respect to the (canonical) local chart coming from $(z,w)\mapsto z$ we have that
\begin{equation} \label{eqn:afnearq0}
w_k^{-r_k}\exp(NG(z_k,w_k))=\left(B^{-1}+\mathcal{O}\left(\frac{|Z_k|}{N^{\frac{1}{2}}}\right)\right)^{N-r_2}\left(\frac{Z_k}{N^{\frac{1}{2}}}\right)^{2\ell N-r_k\ell},\quad\text{for}\quad k=1,2,
\end{equation}
where $B>0$ is as in Lemma \ref{lemma:local_action_function}.
\end{lemma}
\begin{proof}
The result follows from \eqref{eq:at_1_2_q0} in Lemma \ref{lem:at_1}.
\end{proof}
We are now ready to proceed to study the asymptotic behaviour of the contour integrals in \eqref{eqn:kerneljdjs}.
\subsection{Asymptotic analysis of $\mathcal{J}_d$}\label{sec:Jd}
In this section we prove the asymptotic behaviour of the double contour integral, \eqref{eqn:Jd}, in the expression for the kernel of the studied process. 

To keep the notation compact, we introduce another version of the gauge factor from Theorem~\ref{thm:limit_correlation_function},
\begin{equation}\label{eqn:new_gauge}
h(\ell x+i,2y+j)=\sigma^{-(\ell x+i)}g(\ell x+i,2y+j).
\end{equation} 
Then, in the coordinates given in \eqref{eqn:proof_coordinates},
\begin{equation}\label{eqn:gauge_proof_coordinates}
\frac{h(\ell x_1+i_1,2y_1+j_1)}{h(\ell x_2+i_2,2y_2+j_2)}
=N^{\frac{\ell r_2-i_2}{2}-\frac{\ell r_1-i_1}{2}} 
B^{r_2-r_1}
\frac{\tilde g(\ell x_1+i_1,2y_1+j_1)}{\tilde g(\ell x_2+i_2,2y_2+j_2)}.
\end{equation}

We first prove that we can restrict our attention to a double-contour integral for which the contours of integration are contained in a small neighbourhood of $q_{\infty}$.
\begin{lemma}
\label{lem:gcan}
Let $R_2>R_1>0$, then
 \begin{multline}\label{eqn:gcan}
\mathcal{J}_d=
\frac{1}{(2\pi\i)^2}\int_{\tilde \Gamma_s'}\int_{\tilde \Gamma_l'}
\boldsymbol{\mathcal{M}}_{i_1,j_1,i_2,j_2}(z_1,w_1;z_2,w_2) 
\exp\left({N(G(z_2,w_2)-G(z_1,w_1))}\right)
\frac{w_1^{r_1}}{w_2^{r_2}}
 \frac{z_1^{\mu_1\sqrt{N}}}{z_2^{\mu_2\sqrt{N}}}
 \frac{\d z_2\d z_1}{z_2(z_2-z_1)}
\\+N^{-\frac{1}{2}}\frac{h(\ell x_2+i_2,2y_2+j_2)}{h(\ell x_1+i_1,2y_1+j_1)}\mathcal{O}\left(e^{-\frac{\sigma^2}{4}\log^2 N}\right),\quad \text{as}\quad N\rightarrow\infty
\end{multline}
where the curves of integration have been replaced by $\widetilde{\Gamma}_s'$, a circle of radius $\frac{R_1}{\sqrt{N}}$ around $q_{\infty}$ and $\widetilde{\Gamma}_l'$ is a curve to the right of $\widetilde{\Gamma}_s'$ first treading over a half-circle of radius  $\frac{R_2}{\sqrt{N}}$ and then from  $z_2=1+i\frac{R_2}{\sqrt{N}}$ to $z_2=1+i\frac{\log N}{\sqrt{N}}$ along the $y$ axis and from $z_2=1-i\frac{R_2}{\sqrt{N}}$ to $z_2=1-i \frac{\log N}{\sqrt{N}}$ along the $y$-axis, see Figure \ref{fig:local_sight}. The implicit constant in the remainder term is uniform on compact subsets of $\Lambda_{\{0,1\}}^2$.
\end{lemma}

\begin{proof}

We denote the integrand
\begin{equation}\label{eqn:integrand_to_bound}
I_N=\boldsymbol{\mathcal{M}}_{i_1,j_1,i_2,j_2}(z_1,w_1;z_2,w_2) 
\exp\left({N(G(z_2,w_2)-G(z_1,w_1))}\right)
\frac{w_1^{r_1}}{w_2^{r_2}}
 \frac{z_1^{\mu_1\sqrt{N}}}{z_2^{\mu_2\sqrt{N}}}
 \frac{1}{z_2(z_2-z_1)}.
\end{equation}
As a function of $(z_1,w_1)$, it has a zero at $p_0$ and not a pole and hence is analytic in the interior of $\Gamma_s$ except at $q_{\infty}$, see Lemma \ref{lem:global_matrix_meromorphicity}. Thus we can deform $\widetilde{\Gamma}_s$ to $\widetilde{\Gamma}_s'$ (a small circle around $q_{\infty}$), without changing the value of the integral.

We define the contour $\Gamma_{des}^{(N)}$ to be the union of several pieces, see Figure \ref{fig:local_sight}:
\begin{enumerate}
\item The contour of steepest descent $\Gamma_{des}$, of $G(z,w)$ except in a ball of radius $N^{-\frac{1}{4}+\varepsilon}$ centered at $q_{\infty}$. 

\item Let $\zeta_N$ be the point in $\Gamma_{des}\cap \partial B(q_{\infty}, N^{-\frac{1}{4}+\varepsilon})$, with positive imaginary part. We choose this piece of ($\Gamma_{des}^{(N)}$) to thread from $\zeta_N$ to $1+iN^{-1/4+\varepsilon}$ along the shortest arc of $\partial B(q_{\infty}, N^{-\frac{1}{4}+\varepsilon})$ connecting them. Symmetrically we let the curve also tread along the shortest arc connecting $\overline{\zeta_N}$ to $1-iN^{-\frac{1}{4}+\varepsilon}$.

\item $\Gamma_{des}^{(N)}$ treads from $1+iN^{-\frac{1}{4}+\varepsilon}$ to $1+i\frac{\log N}{N^{\frac{1}{2}}}$ and $1-i N^{-\frac{1}{4}}$ to $1-i\frac{\log N}{N^{\frac{1}{2}}}$, parallel to the imaginary axis. %
\item The last piece is the curve $\widetilde{\Gamma}_s'$ described in the statement.

 \end{enumerate}
 \begin{figure}
 \begin{center}
\begin{tikzpicture}[scale=1.1, >=Stealth]

\draw[lightgray, thin, ->] (-3.1,0) -- (3.1,0) node[right, black] {$\Re$};
\draw[lightgray, thin, ->] (0,-3.3) -- (0,3.4) node[above, black] {$\Im$};

\filldraw[black] (0,0) circle (1pt) node[below=0pt, xshift=-2pt] {$q_{\infty}$};
\filldraw[black] (0,1.5) circle (1pt) node[left] {$i\log N$};
\filldraw[black] (0,-1.5) circle (1pt) node[left] {$-i\log N$};
\filldraw[black] (0,2.6) circle (1pt) node[left] {$iN^{\frac{1}{4}-\varepsilon}$};
\filldraw[black] (0,-2.6) circle (1pt) node[left] {$-iN^{\frac{1}{4}-\varepsilon}$};

\filldraw[black] (0,0.4) circle (1pt);
\node[black, left] at (-0.1,0.45) {$iR_1$}; 
\filldraw[black] (0,-0.5) circle (1pt);
\node[black, right] at (0.1,-0.55) {$-iR_2$}; 

\draw[blue, thick] (0,0) circle (0.4);
\node[blue, left] at (-0.4,0.1) {$\widetilde{\Gamma}_s'$};

\draw[red, thick] (0,0.5) arc (90:-90:0.5);
\draw[red, thick] (0,0.5) -- (0,1.5);
\draw[red, thick] (0,-0.5) -- (0,-1.5);
\node[red, right] at (0.3,0.8) {$\widetilde{\Gamma_{\ell}}'$};
\node[red, right, font=\large] at (0.02,-1.0) {${}^{(4)}$}; 

\draw[orange, thick] (0,1.5) -- (0,2.6);
\draw[orange, thick] (0,-1.5) -- (0,-2.6);
\node[orange, right, font=\large] at (0.02,2.1) {${}^{(3)}$};
\node[orange, right, font=\large] at (0.02,-2.1) {${}^{(3)}$};

\draw[lightgray, dashed] (0,0) circle (2.6);

\pgfmathsetmacro{\ycoord}{sqrt(2.6^2 - 0.8^2)} 

\coordinate (zeta) at (0.8, \ycoord);
\coordinate (tzeta) at (0.8, -\ycoord);

\filldraw[black] (zeta) circle (1pt) node[above right] {$\zeta_N$};
\filldraw[black] (tzeta) circle (1pt) node[below right] {$\widetilde{\zeta_N}$};

\draw[green!50!black, thick] (zeta) arc[start angle=72.1, end angle=90, radius=2.6];
\draw[green!50!black, thick] (tzeta) arc[start angle=-72.1, end angle=-90, radius=2.6];

\node[green!50!black, font=\large] at (0.4, 2.75) {${}^{(2)}$};
\node[green!50!black, font=\large] at (0.4, -2.75) {${}^{(2)}$};

\pgfmathsetmacro{\xone}{1*0.8}
\pgfmathsetmacro{\yone}{1*\ycoord}
\pgfmathsetmacro{\xonefive}{1.3*0.8}
\pgfmathsetmacro{\yonefive}{1.3*\ycoord}
\pgfmathsetmacro{\xonesix}{1.4*0.8}
\pgfmathsetmacro{\yonesix}{1.4*\ycoord}

\draw[purple, thick] (\xone, \yone) -- (\xonefive, \yonefive);
\draw[purple, thick, dashed] (\xonefive, \yonefive) -- (\xonesix, \yonesix);

\node[purple, right, font=\large] at ({1.25*0.8}, {1.25*\ycoord}) {${}^{(1)}$};

\pgfmathsetmacro{\xtildeone}{1*0.8}
\pgfmathsetmacro{\ytildeone}{-1*\ycoord}
\pgfmathsetmacro{\xtildeonefive}{1.3*0.8}
\pgfmathsetmacro{\ytildeonefive}{-1.3*\ycoord}
\pgfmathsetmacro{\xtildeonesix}{1.4*0.8}
\pgfmathsetmacro{\ytildeonesix}{-1.4*\ycoord}

\draw[purple, thick] (\xtildeone, \ytildeone) -- (\xtildeonefive, \ytildeonefive);
\draw[purple, thick, dashed] (\xtildeonefive, \ytildeonefive) -- (\xtildeonesix, \ytildeonesix);

\node[purple, left, font=\large] at ({1.25*0.8}, {-1.25*\ycoord}) {${}^{(1)}$};

\end{tikzpicture}

 \end{center}
\caption{Schematic representation on the sheet of $\mathcal{R}$ containing $q_{\infty}$ of the local sight of the curves $\widetilde{\Gamma}_{s}'$ (in blue), and $\Gamma_{des}^{(N)}$   with coordinates centred at $q_{\infty}$ and rescaled by $N^{\frac{1}{2}}$. The incomplete purple curve corresponds to piece $4$ of $\Gamma_{des}^{(N)}$. Correspondingly in green, yellow, and red are drawn pieces 2., 3., and 4. 
} \label{fig:local_sight} 
\end{figure}
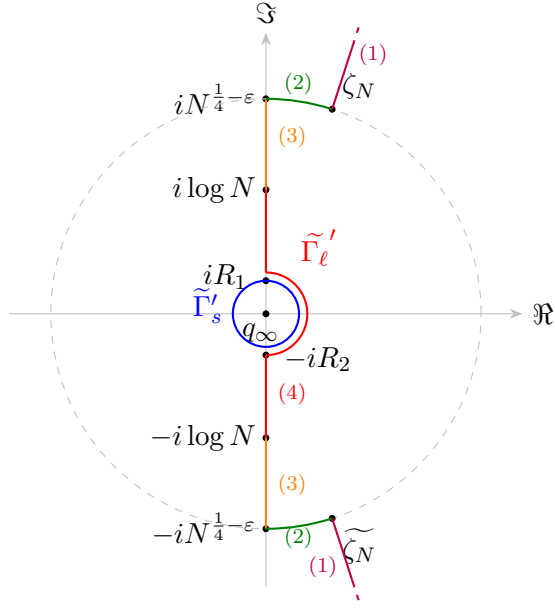

We deform $\widetilde{\Gamma}_{\ell}$ to $\Gamma_{des}^{(N)}$ which can be done since $p_0$ and $q_{\infty}$ are in the interior of both $\widetilde{\Gamma}_{\ell}$ and $\Gamma_{des}^{(N)}$.

In the rest of the proof we bound the integral in \eqref{eqn:integrand_to_bound} on the different pieces of $\widetilde{\Gamma}_{des}^{(N)}$ and $\widetilde{\Gamma}_s'$. 
We begin with several bounds valid for $(z_1,w_1)\in\widetilde{\Gamma}_s'$ and on $(z_2,w_2)$ on any piece of $\Gamma_{des}^{(N)}$. 
\begin{equation}\label{eqn:z1bound}
\left|z_1^{\mu_1 N^{\frac{1}{2}}}\right|\leq\left|1+{R_1}{N^{-\frac{1}{2}}}\right|^{\mu_1 N^{\frac{1}{2}}}\leq e^{\mu_1 R_1} \text{ and  } 
|w_1^{r_1}|
\leq
B^{r_1} R_1^{-\ell r_1} 
N^{\frac{\ell r_1}{2}} 
\left(1+\mathcal{O}\left( r_1N^{-\frac{1}{2}}\right)\right)
\text{ for }(z_1,w_1)\in\widetilde{\Gamma}_s'.
\end{equation}
Since $\Gamma_{des}^{(N)}$ is bounded away from $p_0$ and $p_{\infty}$ for some $C>1$ independent of $N$, we have that
 \begin{equation}\label{eqn:z2bound}
 \left| {z_2^{-\mu_2N^{\frac{1}{2}}}}\right|\leq C^{\mu_2N^{\frac{1}{2}}}=\exp\left(\mu_2N^{\frac{1}{2}}\log C\right),\quad\text{for}\quad (z_2,w_2)\in \Gamma_{des}^{(N)}.
 \end{equation}
By the same reasoning due to the Lemma \ref{lem:global_matrix_meromorphicity} 
$\boldsymbol{\mathcal{M}}_{i_1,j_1,i_2,j_2}(z_1,w_1;z_2,w_2)$ (which can only have poles of order at most $\ell$ at $q_0$ and $q_{\infty}$) for some $C>0$ (independent of $N$) satisfies
\begin{multline}\label{eqn:global_matrix_proof}
\left|\boldsymbol{\mathcal{M}}_{i_1,j_1,i_2,j_2}(z_1,w_1;z_2,w_2)\right | \leq C\left(1+\frac{1}{|z_1-1|^{\ell}}\right)\left(1+\frac{1}{ |z_2-1 |^{\ell}}\right),
 \\\text{for}\quad (z_1,w_1)\in\widetilde{\Gamma}_s',\quad (z_2,w_2)\in\Gamma_{des}^{(N)}.
\end{multline}
In the setup for the analysis it is convenient to treat the case when $(z_2,w_2)$ is in a small neighbourhood of $q_0$ separately, though the integral is easy to control there. With some abuse of notation we denote the neighbourhood of $q_{0}$ corresponding to the ball $B(1, R_2 N^{-\frac{1}{2}})$ in the local chart to be $B(q_0, R_2 N^{-\frac{1}{2}})$. We use Lemmas \ref{lem:global_matrix_meromorphicity}, \ref{lemma:G_at_q0} to bound terms in \eqref{eqn:integrand_to_bound} depending on $(z_2,w_2)\in B(q_0, R_2 N^{-\frac{1}{2}})$, and Lemma  \ref{lemma:expansions} to bound terms depending only on $(z_1,w_1)\in \widetilde{\Gamma}_s'$, together with the bounds
  $|z_2-1|\leq R_2 N^{-\frac{1}{2}}$ and $|z_1-1|={R_1}{N^{-\frac{1}{2}}}$, to obtain that
 \begin{multline} \label{eqn:INq0}
 \left| \exp(N(G(z_2,w_2)-G(z_1,w_1)) \frac{w_1^{r_1}}{w_2^{r_2}} \frac{\boldsymbol{\mathcal{M}}_{i_1,j_1,i_2,j_2}(z_1,w_1;z_2,w_2)}{(z_2-z_1) z_2} \frac{z_1^{\mu_1 N^{\frac{1}{2}}}}{z_2^{\mu_2N^{\frac{1}{2}}}} \right|
 \\ \leq \exp\left(-2\ell N(1+o(1)) \frac{1}{2}\log N  \right)
 \text{ for $(z_2,w_2)\in B(q_0, R_2 N^{-\frac{1}{2}})$,  $(z_1,w_1)\in \widetilde{\Gamma}_s'$}.
 \end{multline}
We also have that
\begin{equation} \label{eqn:trivial}
z_2^{-1}=\mathcal{O}(1)\quad\text{and}\quad \frac{1}{(z_2-z_1)}=\mathcal{O}\left(\frac{N^{\frac{1}{2}}}{R_2-R_1}\right)
=\mathcal{O}(N^{\frac{1}{2}}), \text{ for } z_1\in \widetilde{\Gamma}_s',\quad  z_2\in\Gamma_{des}^{(N)}\setminus B(q_0, R_2N^{-\frac{1}{2}}).
\end{equation}
We proceed with bounds specific to the different pieces of $\Gamma_{des}^{(N)}$.
 
\begin{enumerate}
 \item
We begin with a bound on piece 2. 
On the local chart given by $(z,w)\mapsto z$ due to Lemma \ref{lemma:curvesofsteepest} we have $G'(1)>0$, and so 
 $\zeta_N=1+iN^{-\frac{1}{4}+\varepsilon}(1+o(1))$ (and $\overline{\zeta_N}= 1-iN^{-\frac{1}{4}+\varepsilon}(1+o(1))$)
and the same holds for the corresponding $z_2$ on piece 2.
  Evaluating the right-hand side of \eqref{eqn:local_action_contribution} in Lemma \ref{lemma:expansions} at
 $Z_2=\sqrt{N}(\zeta-1)=\pm iN^{\frac{1}{4}+\varepsilon}(1+o(1))$, 
we see that 
 \begin{multline}\label{eqn:on_arc}
\exp(N\re(G(\zeta,w(\zeta))-G(z_1,w_1))
 \leq \exp\left(-\frac{\sigma^2}{2} N^{\frac{1}{2}+2\varepsilon}(1+o(1))\right)
 \quad
 \\
 \text{for}
 \quad
 (z_1,w_1)\in \widetilde{\Gamma}_s',
\quad (\zeta, w(\zeta))\quad \text{as chosen above.}
\end{multline}
\item
We proceed to bound simultaneously in regard to piece 1 and piece 2 from this point onwards.  To bound the integrand on piece 1, we make use of the curve of steepest descent in the following way. We have that for $(z_2, w_2)$ on piece 1 of $\Gamma_{des}^{(N)}$,
 \begin{equation}
\exp(N\re(G(z_2,w_2)))
 \leq
  \exp\left(N\re G(\zeta_N, w(\zeta_N))\right)
 \quad \text{for} 
 \quad (z_2,w_2)
 \in \Gamma_{des}\setminus B(q_{\infty}, N^{-\frac{1}{4}+\varepsilon}),
 \end{equation}
and so it holds that
 \begin{multline}\label{eqn:steepestdescentafb}
 \exp(N\re(G(z_2,w_2)-G(z_1,w_1))
 \leq \exp\left(-\frac{\sigma^2}{2} N^{\frac{1}{2}+2\varepsilon}(1+o(1))\right),
 \\ \text{ for }
 (z_2,w_2)\text{ on piece 1 and 2 of $\Gamma_{des}^{(N)}$ and }
 (z_1,w_1)\in \widetilde{\Gamma}_s'.
 \end{multline}

Now since $w_2$ has a pole of order $\ell$ at $q_{\infty}$ and a zero of that order at $q_0$ (and at no other locations), it follows that for some $C>0$,
\begin{equation} \label{eqn:w_2_global_bound}
|w_2^{-r_2}|\leq C^{{r_2}} 
N^{\frac{\ell r_2}{2}} \quad \text{for}\quad w_2\in\Gamma_{des}^{(N)}\setminus B(q_0, {R_2}N^{-\frac{1}{2}}).
\end{equation}

Due to \eqref{eqn:INq0}, whenever  $(z_2,w_2)\in B(q_0, R_2 N^{-\frac{1}{2}})$ and else 
combining the bounds in \eqref{eqn:z1bound}, \eqref{eqn:z2bound}, 
\eqref{eqn:global_matrix_proof},
\eqref{eqn:w_2_global_bound}, \eqref{eqn:trivial},
 and \eqref{eqn:steepestdescentafb} for the different terms on the right-hand side of \eqref{eqn:integrand_to_bound},
 we see that for
$(z_1,w_1)\in \widetilde{\Gamma}_s'$ and $(z_2,w_2)$ 
on piece 1 or piece 2 of $\Gamma_{des}^{(N)}$,
\begin{equation}
I_N
=\mathcal{O}
\left(
\exp \left(
-\frac{\sigma^2}{2}N^{\frac{1}{2}+2\varepsilon}(1+o(1))
\right)
\right),\quad \text{as}\quad N\rightarrow\infty.
\end{equation}
Furthermore on compact subsets of $\Lambda_{\{0,1\}}^2$ the bound is uniform. As $\widetilde{\Gamma}_s'$ has length $\frac{2\pi R_1}{\sqrt{N}}$ and $\Gamma_{des}^{(N)}$ has a finite length, it follows that
\begin{multline} 
\left| 
\frac{1}{2\pi i}
\int_{\widetilde{\Gamma}_s'} 
\int_{\Gamma_{des}^{(N)}\setminus B(q_{\infty}, N^{-\frac{1}{4}+\varepsilon})}
I_N
dz_1 dz_2
\right|
=\mathcal{O}
\left(
\exp
\left(-\frac{\sigma^2}{2}N^{\frac{1}{2}+2\varepsilon}(1+o(1))
\right)
\right),\quad \text{as}\quad N\rightarrow\infty,
\end{multline}
uniformly on compact subsets of $\Lambda_{\{0,1\}}^2$. Since the gauge factor, \eqref{eqn:gauge_proof_coordinates}, is 
$\mathcal{O}\left(N^{\frac{\ell r_2-i_2}{2}-\frac{\ell r_1-i_1}{2}}\right)$ it follows straight away that 
\begin{multline} \label{eqn:cut_piece_1}
\frac{h(\ell x_1+i_1,2y_1+j_1)}{h(\ell x_2+i_2,2y_2+j_2)}
\left| 
\frac{1}{2\pi i}
\int_{\widetilde{\Gamma}_s'} 
\int_{\Gamma_{des}^{(N)}\setminus B(q_{\infty}, N^{-\frac{1}{4}+\varepsilon})}
I_N
dz_1 dz_2
\right|
\\=\mathcal{O}
\left(
\exp
\left(-\frac{\sigma^2}{2}N^{\frac{1}{2}+2\varepsilon}(1+o(1))
\right)
\right),\quad \text{as}\quad N\rightarrow\infty,
\end{multline}
uniformly on compact subsets of $\Lambda_{\{0,1\}}^2$.
We have now bounded the contribution to $\mathcal{J}_d$ from the first two pieces of $\Gamma_{des}^{(N)}$. 

\item
It remains to bound the contribution from piece 3 of $\Gamma_{des}^{(N)}$. We make the local change of variables around $q_{\infty}$, $z_k=1+\frac{Z_k}{\sqrt{N}},\quad \d z_k=N^{-\frac{1}{2}}\d Z_k,$ for $ k=1,2$. In these coordinates we integrate over the straight line where $Z_2\in i[\log N, N^{\frac{1}{4}+\varepsilon}] $ (and analogously $Z_2\in-i[\log N, N^{\frac{1}{4}+\varepsilon}])$ and $|Z_1|=R_1$.
We once again estimate the individual terms in $I_N$.
Due to Lemmas \ref{lemma:expansions}, \ref{lem:matrix_local} we have that for $|Z_1|=R_1$, $Z_2\in i[\log N, N^{\frac{1}{4}+\varepsilon}]$, correspondingly adapting \eqref{eqn:local_action_contribution}, \eqref{eqn:local_z_contribution}, \eqref{eqn:matrix_precise}, as $N\rightarrow\infty$,
\begin{equation}\label{eqn:piece4action}
\exp(N(G(z_2,w_2)-G(z_1,w_1)))=\exp\left(-\frac{\sigma^2}{2}Z_2^2(1+\mathcal{O}(N^{-\frac{1}{4}}))\right),
\end{equation}
\begin{equation}\label{eqn:piece4w}
w_1^{r_1}w_2^{-r_2}=
B^{r_1-r_2}
N^{\frac{\ell r_1}{2}-\frac{\ell r_2}{2}}
 {Z_2^{\ell r_2}}{R_1^{-\ell r_1}}\left(1+\mathcal{O}(N^{-\frac{1}{4}+\varepsilon})\right),
\end{equation}
\begin{equation}\label{eqn:piece4matrix}
\boldsymbol{\mathcal{M}}_{i_1,j_1,i_2,j_2}\left(z_1,w_1;z_2,w_2\right)
= 
{N^{\frac{i_2-i_1}{2}}}
\nu(\ell-i_2,j_2){Z_2^{-i_2}}
\frac{\tilde g(\ell x_2+i_2,2y_2+j_2)}{\tilde g(\ell x_1+i_1,2y_1+j_1)}
{R_1^{i_1}}
\left(1+\Ordo\left(N^{-1/2}\right)\right).
\end{equation}
Together with the observation that for $Z_2\in i\mathbb{R}$, $\left|1+\frac{Z_2}{N^{\frac{1}{2}}}\right|\geq 1$ and  \eqref{eqn:z1bound}, \eqref{eqn:piece4action}, \eqref{eqn:piece4w}, \eqref{eqn:piece4matrix} yield that in these new coordinates, as $N\rightarrow\infty$,
\begin{multline} \label{eqn:in_bound1}
|I_N|
\leq \exp(\mu_1 R_1) R_1^{-\ell r_1+i_1}(1+o(1))(\log N-R_1)^{-1}
\nu(\ell-i_2,j_2)
\\\times
B^{r_1-r_2} 
\frac{\tilde g(\ell x_2+i_2,2y_2+j_2)}{\tilde g(\ell x_1+i_1,2y_1+j_1)}
N^{\frac{\ell r_1 -i_1-\ell r_2+i_2}{2}} 
\\\times|Z_2|^{\ell r_2-i_2}
\exp\left(-\frac{\sigma^2}{2}|Z_2|^2(1+\mathcal{O}(N^{-\frac{1}{4}}))\right) \times N^{\frac{1}{2}},
\end{multline}
where the estimates are uniform on compact subsets of $\Lambda_{\{0,1\}}^2$. Thus, as the length of the closed curve is $2\pi R_1$, 
recognising the second line in \eqref{eqn:in_bound1} to be the inverse of the gauge factor, \eqref{eqn:gauge_proof_coordinates}, we have that
\begin{multline}\label{eqn:cut_piece_2}
N^{\frac{1}{2}}
\frac{h(\ell x_1+i_1,2y_1+j_1)}{h(\ell x_2+i_2,2y_2+j_2)}
\int_{Z_2\in i[\log N,N^{\frac{1}{4}+\varepsilon}]}\int_{|Z_1|=R_1} |I_N| \frac{|\d Z_1| |\d Z_2|}{N}
\\= 
\nu(\ell-i_2,j_2)
%
\times\mathcal{O}
\left(
\int_{\log N}^{N^{\frac{1}{4}+\varepsilon}}\xi^{\ell r_2-i_2} \exp\left(-\frac{\sigma^2}{2}\xi^2(1+o(1)) \right) \d \xi
\right)
\\=%
\mathcal{O}(N^{-\frac{\sigma^2}{2} \log N(1+o(1))}),\quad \text{as}\quad N\rightarrow\infty,
\end{multline}
where in the last line we used an estimate for Gaussian tails and the bound is uniform on compact subsets of $\Lambda_{\{0,1\}}^2$. 
The bounds in \eqref{eqn:cut_piece_1}, \eqref{eqn:cut_piece_2} prove \eqref{eqn:gcan}. 
\end{enumerate}
\end{proof}

We continue by computing asymptotic behaviour of the double-contour integral appearing in the previous lemma (on the right-hand side of \eqref{eqn:gcan}), which carries the main contribution to $\mathcal{J}_d$. Since this integral is completely contained in a small neighbourhood of $q_{\infty}$, we can fix our choice of local chart described in the beginning of the section and analyse it as an integral on $\mathbb{C}$. 
\begin{lemma}\label{lem:Jdvalue}
Let $\widetilde{\Gamma}_s'$ and $\widetilde{\Gamma}_l'$ be the curves in Lemma \ref{lem:gcan} and denote
\begin{equation}
\mathcal{\widetilde{J}}_d=\frac{1}{(2\pi\i)^2}\int_{\tilde \Gamma_s'}\int_{\tilde \Gamma_l'}
\boldsymbol{\mathcal{M}}_{i_1,j_1,i_2,j_2}(z_1,w_1;z_2,w_2) 
\exp\left({N(G(z_2,w_2)-G(z_1,w_1))}\right)
\frac{w_1^{r_1}}{w_2^{r_2}}
 \frac{z_1^{\mu_1\sqrt{N}}}{z_2^{\mu_2\sqrt{N}}}
 \frac{\d z_2\d z_1}{z_2(z_2-z_1)}.
 \end{equation}
 Then we have that, uniformly on compact subsets of $\Lambda_{\{0,1\}}^2$,
 \begin{multline}
 N^{\frac{1}{2}}
\frac{h(\ell x_1+i_1,2y_1+j_1)}{h(\ell x_2+i_2,2y_2+j_2)}
 \mathcal{\widetilde{J}}_d
 \\=
\frac{\nu(\ell-i_2,j_2)}{(2\pi\i)^2}\left(
\int_{|z_1|=R_1}
\int_{-i\infty}^{i\infty}
\e^{\frac{\sigma^2}{2}
(z_2^2-z_1^2)}\e^{\mu_1 z_1-\mu_2 z_2}\frac{z_2^{\ell r_2-i_2}}{z_1^{\ell r_1-i_1}}\frac{\d z_1\d z_2}{z_2-z_1}+\mathcal{O}(N^{-\frac{1}{4}})\right),
\quad\text{as}\quad N\rightarrow\infty,
\end{multline}
where $R_1$ is as in Lemma \ref{lem:gcan}, the closed contour around the origin is oriented counter-clockwise, and the contour connection of $-i\infty$ and $+i\infty$ is to the right of this contour. 
\end{lemma}
\begin{proof}
We make the local change of variables $z_i=1+\frac{Z_i}{N^\frac{1}{2}}$, $dz_i=\frac{dZ_i}{N^{\frac{1}{2}}}$ ~$i=1,2$, near $q_{\infty}$.
From the local expansions in Lemmas \ref{lem:matrix_local} and \ref{lemma:local_action_function} it follows that 
\begin{multline*}
N^{\frac{1}{2}}N^{\frac{\ell r_2-i_2}{2}-\frac{\ell r_1-i_1}{2}}
 B^{r_2-r_1}
\frac{\tilde g(\ell x_1+i_1,2y_1+j_1)}{\tilde g(\ell x_2+i_2,2y_2+j_2)}
\mathcal{\widetilde{J}}_d
=
\frac{\nu(\ell-i_2,j_2)}{(2\pi\i)^2}
\times
\\
\int_{|Z_1|=R_1}
\int_{\sqrt{N}(-1+\widetilde{\Gamma}_l')}
\frac{Z_2^{\ell r_2-i_2}}{Z_1^{\ell r_1-i_1}}\exp\left(\frac{\sigma^2}{2}(Z_2^2-Z_1^2)+Z_1\mu_1-Z_2\mu_2\right)
\frac{\left(1+\mathcal{O}(N^{-\frac{1}{4}})\right)}{(Z_2-Z_1)}
\d Z_1\d Z_2.
\end{multline*}
We recognise the gauge factor, \eqref{eqn:gauge_proof_coordinates}, and see that the pre-factor before the integral is the desired one. It remains to analyse the integral.
What follows is a comparison between the double-contour integral over the red and blue contours in Figures \ref{fig:local_sight} and \ref{fig:contours_gue_corners}.
 We separate the error term in the integrand and bound the integral that arises from it, see $J_1$ and $J_2$ below. We add and subtract the integral corresponding to connecting the $Z_2$ contour
 along the straight line from $i\log{N}$ to $i\infty$ and $-i\log N$ to $-i\infty$, contained in $J_3$ below. We obtain that as $N\rightarrow\infty$,
\begin{multline}\label{eqn:decomposition_local_computation}
\frac{1}{(2\pi\i)^2}\int_{|Z_1|=R_1}
\int_{\sqrt{N}(-1+\Gamma_l')}
\frac{Z_2^{\ell r_2-i_2}}{Z_1^{\ell r_1-i_1}}\exp\left(\frac{\sigma^2}{2}(Z_2^2-Z_1^2)+Z_1\mu_1-Z_2\mu_2\right)
\frac{\left(1+\mathcal{O}(N^{-1/4})\right)}{(Z_2-Z_1)}
\d Z_1\d Z_2
\\=
\frac{1}{(2\pi\i)^2}\int_{|Z_1|=R_1}
\int_{-i\infty}^{i\infty}
\frac{Z_2^{\ell r_2-i_2}}{Z_1^{\ell r_1-i_1}}\exp{\left(\frac{\sigma^2}{2}(Z_2^2-Z_1^2)+Z_1\mu_1-Z_2\mu_2\right)}
\frac{\d Z_1\d Z_2}{(Z_2-Z_1)}
\\+ \mathcal{O}\left(N^{-\frac{1}{4}}J_1(\ell_2 r_2-i_2, \mu_2, \ell_1r_1-i_1, \mu_1)+N^{-\frac{1}{4}}J_2(\ell_2 r_2-i_2, \mu_2, \ell_1r_1-i_1, \mu_1)\right)
+J_3,
\end{multline}
where the contour connection of $-i\infty$ and $i\infty$ is by construction to the right of the closed contour around the origin and
\begin{equation}
J_1= \int_{|Z_1|=R_1}\int_{Z_2=R_2e^{it2\pi}, t\in[-\pi/2,\pi,2]}
\frac{|Z_2|^{\ell r_2-i_2}}{|Z_1|^{\ell r_1-i_1}}\left|\exp\left(\frac{\sigma^2}{2}(Z_2^2-Z_1^2)+Z_1\mu_1-Z_2\mu_2\right)\right|
\frac{|\d Z_1||\d Z_2|}{|Z_2-Z_1|},
\end{equation}
\begin{equation}
J_2=  \int_{|Z_1|=R_1}\int_{Z_2\in[iR_2, i\log N]\bigcup [-i\log N, -iR_2]} 
\frac{|Z_2|^{\ell r_2-i_2}}{|Z_1|^{\ell r_1-i_1}}\left|\exp\left(\frac{\sigma^2}{2}(Z_2^2-Z_1^2)+Z_1\mu_1-Z_2\mu_2\right)\right|
\frac{|\d Z_1||\d Z_2|}{|Z_2-Z_1|},
\end{equation}
\begin{equation}
J_3=\int_{|Z_1|=R_1}\int_{Z_2\in[i\log N, i\infty]\bigcup (-i\infty, -i\log N]} 
\frac{Z_2^{\ell r_2-i_2}}{Z_1^{\ell r_1-i_1}}\exp{\left(\frac{\sigma^2}{2}(Z_2^2-Z_1^2)+Z_1\mu_1-Z_2\mu_2\right)}
\frac{\d Z_1\d Z_2}{(Z_2-Z_1)}.
\end{equation}
Finally we see that
\begin{equation}
J_1\leq \frac{2\pi^2 R_1R_2}{R_2-R_1} \frac{R_2^{\ell r_2-i_2}}{R_1^{\ell r_1-i_1}}
\exp\left(\frac{\sigma^2}{2}(R_1^2+R_2^2)+R_1\mu_1+R_2\mu_2\right), 
\end{equation}
which on compact subsets of $\Lambda_{\{0,1\}}^2$ (in the described embedding) is uniformly bounded. Similarly
\begin{equation}
J_2\leq 4\pi \frac{R_1}{R_2-R_1} R_1^{i_1-\ell r_1}e^{\frac{\sigma^2}{2} R_1^2+R_1\mu_1}\int_{R_2}^{\infty} t^{\ell r_2-i_2} e^{-\frac{\sigma^2}{2} t^2}\d t,
\end{equation}
is uniformly bounded on compact subsets of $\Lambda_{\{0,1\}}^2$. Finally identically to $J_2$, we bound $J_3$,
\begin{equation}
\left|J_3\right|
\leq 4\pi \frac{R_1}{\log N-R_1} R_1^{i_1-\ell r_1}
e^{\frac{\sigma^2}{2} R_1^2+R_1\mu_1}
\int_{\log N}^{\infty} t^{\ell r_2-i_2} e^{-\frac{\sigma^2}{2} t^2}\d t= \mathcal{O}\left(e^{-\frac{\sigma^2}{2} (\log N)^2(1+o(1))}\right),
\end{equation}
 as $N\rightarrow\infty$. We use the bounds on $J_1$, $J_2$, $J_3$ in \eqref{eqn:decomposition_local_computation} to deduce the result. 
\end{proof}

\subsection{Asymptotic analysis of $\mathcal{J}_s$}\label{sec:Js}
We analyse $\mathcal{J}_s$ (recall \eqref{eqn:Js}) in this section. 
\begin{lemma}
\label{lemma:Js}
Assuming $\ell r_1-i_1> \ell r_2-i_2$, we have that 
\begin{multline}
N^{\frac{1}{2}}
\frac{h(\ell x_1+i_1,2y_1+j_1)}{h(\ell x_2+i_2,2y_2+j_2)}
 \mathcal{J}_{s}=
\\=\begin{cases}
\mathcal{O}
\left(N^{\frac{\ell r_2-i_2}{2}-\frac{\ell r_1-i_1}{2}+\frac{1}{2}} \exp\left(-\sqrt{N}\log 2(\mu_1-\mu_2)\right)\right) 
 &\text{for } \mu_1 \leq \mu_2,
\\
\frac{\nu(\ell-i_2,j_2)}{2\pi i}
\int_{|Z|=1} 
  Z^{\ell r_2-i_2-\ell r_1+i_1}
\exp((\mu_1-\mu_2) Z)
dZ
+
\mathcal{O}\left(N^{-\frac{1}{2}}\right)
&\text{for} \quad \mu_1>\mu_2,
\end{cases}
\end{multline}
as $N\rightarrow\infty$ and in particular the left-hand side is uniformly bounded in $N$ on compact subsets of $\Lambda_{\{0,1\}}^2$, under the constructed embedding, and converges point-wise.
\end{lemma}

\begin{proof}
From Lemmas \ref{lem:at_1} and \ref{lem:matrix_local} we can deduce that
$\boldsymbol{\mathcal{M}}_{i_1,j_1,i_2,j_2}(z,w;z,w) 
w^{r_1-r_2}$
has a pole of order $(\ell r_1-i_1)-(\ell r_2-i_2)$ at $q_{\infty}$ and a zero at $q_0$.
\begin{enumerate}
\item
Now if $\mu_1 \leq \mu_2$, since $q_{\infty}$ and $p_0$ are in the interior of the curve $\tilde{\Gamma}_s$ and the integrand is analytic at $q_0$, we deform the contour so that $|z|=2$. Hence we have that 
\begin{equation}
(2\pi)^{-1}\left|\boldsymbol{\mathcal{M}}_{i_1,j_1,i_2,j_2}(z,w;z,w) 
w^{r_1-r_2}
z^{(\mu_1-\mu_2)\sqrt{N}-1}\right|
\leq H(\ell r_1, i_1, \ell r_2, i_2) 2^{-\sqrt{N}(\mu_2-\mu_1)-1}
,
\end{equation}
where $H$ is uniformly bounded on compact sets of $\Lambda_{\{0,1\}}$. 
Hence since the length of the curve of integration is finite and recalling the form of the gauge factor, \eqref{eqn:gauge_proof_coordinates},  
there is a constant $C>0$ (independent of $N$) such that
$$N^{\frac{1}{2}}
\frac{h(\ell x_1+i_1,2y_1+j_1)}{h(\ell x_2+i_2,2y_2+j_2)}  |\mathcal{J}_s|\leq 
 C B^{r_2-r_1} 
N^{\frac{\ell r_2-i_2}{2}-\frac{\ell r_1-i_1}{2}+\frac{1}{2}}
 H(\ell r_1, i_1, \ell r_2, i_2)  2^{-\sqrt{N}(\mu_2-\mu_1)-1}.
 $$
In particular since $\ell r_1-i_1-\ell r_2 +i_2 \geq 1$ we have a uniform bound on $\mathcal{J}_s$ for $\mu_1\leq \mu_2$ (on compact subsets of $\Lambda_{\{0,1\}}^2$). %
Pointwise, for $\mu_1<\mu_2$, we have that
\begin{equation}
N^{\frac{1}{2}}
\frac{h(\ell x_1+i_1,2y_1+j_1)}{h(\ell x_2+i_2,2y_2+j_2)}
 |\mathcal{J}_{s}|
=\mathcal{O}
\left(\exp\left(-\sqrt{N}\log 2(\mu_2-\mu_1)\right)\right), \quad \text{as}\quad N\rightarrow\infty.
\end{equation}
\item
Let $\mu_1\geq \mu_2$. In this case the pole at $p_0$ is finite and we can deform $\Gamma_s$ to a simple closed curve treading around $p_0$, $\Gamma_s^{(1)}$ with $p_{\infty}, q_{0}, q_{\infty}$ in its exterior, and a simple curve around $q_{\infty}$, $\Gamma_s^{(2)}$ with $p_0$, $q_0$, $p_{\infty}$ in its exterior. 

 We choose  $\Gamma_s^{(2)}$ to be a circle of radius $\frac{1}{\sqrt{N}}$ around $q_{\infty}$, and we make the change of variables $z=1+\frac{Z}{\sqrt{N}}, dz=\frac{dZ}{\sqrt{N}}$ (on the circle $|Z|=1$). By the identities proved in Lemmas \ref{lem:matrix_local} and \ref{lemma:expansions}, we have that
\begin{multline}
N^{\frac{1}{2}}
\frac{h(\ell x_1+i_1,2y_1+j_1)}{h(\ell x_2+i_2,2y_2+j_2)}
\frac{\one_{\ell r_1-i_1>\ell r_2-i_2}}{(2\pi\i)}\int_{ \Gamma_s^{(2)}}
\boldsymbol{\mathcal{M}}_{i_1,j_1,i_2,j_2}(z,w;z,w) 
w^{r_1-r_2}
z^{(\mu_1-\mu_2)\sqrt{N}}
\frac{\d z}{z}
\\=
%
\frac{\nu(\ell-i_2,j_2)}{2\pi i}
\int_{|Z|=1} Z^{\ell r_2-i_2-\ell r_1+i_1}
\exp((\mu_1-\mu_2) Z)
 \left(1+\mathcal{O}(N^{-\frac{1}{2}})\right)dZ
\\= \nu(\ell-i_2,j_2)
\left(
 \frac{1}{2\pi i}\int_{|Z|=1} 
  Z^{\ell r_2-i_2-\ell r_1+i_1}
\exp((\mu_1-\mu_2) Z)
\d Z
+\mathcal{O}\left(N^{-\frac{1}{2}}\right)\right),
\text{  as  }
 N\rightarrow\infty,
  \end{multline}
where the implicit constant is uniformly bounded on compact subsets of $\Lambda_{\{0,1\}}^2$.

The integral over $\Gamma_s^{(1)}$ is uniformly bounded, since the pole at $p_0$ of the integrand is of finite order (in fact $\leq 1$). Thus we have that uniformly on compact subsets of $\Lambda_{\{0,1\}}^2$,
\begin{multline}\label{eqn:justify_last_line}
N^{\frac{1}{2}}
\frac{h(\ell x_1+i_1,2y_1+j_1)}{h(\ell x_2+i_2,2y_2+j_2)}
\frac{\one_{\ell r_1-i_1>\ell r_2-i_2}}{(2\pi\i)}\int_{ \Gamma_s^{(1)}}
\boldsymbol{\mathcal{M}}_{i_1,j_1,i_2,j_2}(z,w;z,w) 
w^{r_1-r_2}
z^{(\mu_1-\mu_2)\sqrt{N}}
\frac{\d z}{z}
\\=\mathcal{O}\left(N^{\frac{\ell r_2-i_2-\ell r_1+i_1+1}{2}}\right),\quad\text{as } N\rightarrow\infty.
\end{multline}
If $\mu_1>\mu_2$, for $N$ sufficiently large, the integrand of $\mathcal{J}_s$ is analytic at $p_0$ and so the integral over $\Gamma_s^{(1)}$ is 0.
\end{enumerate}
The proof is now complete.
\end{proof}

\begin{remark}
A residue computation shows that
\begin{equation}
 \frac{1}{2\pi i}\int_{|Z|=1} 
  Z^{t_2-t_1}
\exp((\mu_1-\mu_2) Z)
\d Z
=
\begin{cases}
 \frac{(\mu_1-\mu_2)^{t_1-t_2-1}}{(t_1 - t_2-1)!}\quad\text{if}\quad t_1> t_2,
 \\ 0,\quad\text{otherwise.}
 \end{cases}
\end{equation}
This is the term given in \cite{OR06}.
\end{remark}
We now have all the ingredients to prove a convergence of the correlation kernel in Theorem \ref{thm:limit_correlation_function}.

\subsection{Proof of Theorem \ref{thm:limit_correlation_function}}\label{sec:proof}
We collect the asymptotic results proved in this section to prove Theorem \ref{thm:limit_correlation_function}. 
\begin{proof}[Proof of Theorem \ref{thm:limit_correlation_function}]
Recall once again the coordinates we introduced in this section, \eqref{eqn:kerneljdjs}, and the gauge transform in these coordinates, \eqref{eqn:gauge_proof_coordinates}, and observe that due to the periodicity of $\nu$,
$$\nu(\ell-i_2,j_2)= \nu(\ell r_2-i_2,j_2)=\nu(t_2,j_2).$$
Thus it follows from Lemmas \ref{lem:gcan}, \ref{lem:Jdvalue}, and \ref{lemma:Js} that for $\mu_1\not=\mu_2$, as $N\rightarrow\infty$,
\begin{multline}\label{eqn:limit_in_proof}
 N^{\frac{1}{2}}
\frac{h(\ell x_1+i_1,2y_1+j_1)}{h(\ell x_2+i_2,2y_2+j_2)}
K_{\operatorname{Int}}(\ell x_1+i_1, 2y_1+j_1; \ell x_2+i_2, 2y_2+j_2 )
\\\rightarrow
\frac{\nu(\ell-i_2,j_2)}{(2\pi\i)^2}
\int_{|z_1|=R_1}
\int_{-i\infty}^{i\infty}
\e^{\frac{\sigma^2}{2}
(z_2^2-z_1^2)}\e^{\mu_1 z_1-\mu_2 z_2}\frac{z_2^{\ell r_2-i_2}}{z_1^{\ell r_1-i_1}}\frac{\d z_1\d z_2}{z_2-z_1}
\\ - \one_{\ell r_1-i_1>\ell r_2-i_2, \mu_1>\mu_2}
 \frac{\nu(\ell-i_2,j_2)}{2\pi i}\int_{|z|=1} 
 z^{\ell r_2-i_2-\ell r_1+i_1}
e^{(\mu_1-\mu_2) z}
dz
\\=
\frac{\nu(t_2,j_2)}{(2\pi\i)^2}
\int_{|z_1|=R_1}
\int_{-i\infty}^{i\infty}
\e^{\frac{\sigma^2}{2}
(z_2^2-z_1^2)}\e^{\mu_1 z_1-\mu_2 z_2}\frac{z_2^{t_2}}{z_1^{t_1}}\frac{\d z_1\d z_2}{z_2-z_1}
\\ - \one_{t_1>t_2, \mu_1>\mu_2}
 \frac{\nu(t_2, j_2)}{2\pi i}\int_{|z|=1} 
 z^{t_2-t_1}
e^{(\mu_1-\mu_2) z}
dz,
\end{multline}
 and the contour connection of $-i\infty$ and $i\infty$ is to the right of the counter-clockwise oriented circle $|z|=R_1$.
The 
sequence on the left-hand side is uniformly bounded on compact subsets of $\Lambda_{\{0,1\}}^2$ (even for $\mu_1=\mu_2$).

To bring the limiting correlation kernel to the form stated in the theorem we make the following observations:
\begin{enumerate}
\item
 For $t_1\leq t_2$, $ z^{t_2-t_1}
e^{(\mu_1-\mu_2) z}$ is holomorphic and so
$$- \one_{t_1>t_2, \mu_1>\mu_2}
 \frac{1}{2\pi i}\int_{|z|=1} 
 z^{t_2-t_1}
e^{(\mu_1-\mu_2) z}
dz=- \one_{\mu_1>\mu_2}
 \frac{1}{2\pi i}\int_{|z|=1} 
 z^{t_2-t_1}
e^{(\mu_1-\mu_2) z}
dz.$$
\item
In the limiting double-contour integral, in the case $\mu_1>\mu_2$, moving the contour connection of $-i\infty$ to $i\infty$ to the left of $\{|z_1|=R_1\}$ corresponds exactly to picking up a residue at $z_1=z_2$ for $|z_1|=R_1$, that is an extra term of
$$+  \frac{\nu(t_2, j_2)}{2\pi i}\int_{|z|=R_1} 
 z^{t_2-t_1}
\exp((\mu_1-\mu_2) z)
\d z,$$
which then cancels the single integral.
\end{enumerate}
These observations, after multiplying the left-hand side in \eqref{eqn:limit_in_proof} with $\sigma^{t_2-t_1}$ (recall \eqref{eqn:new_gauge}) show that
\begin{multline*}
 N^{\frac{1}{2}}
\frac{g(\ell x_1+i_1,2y_1+j_1)}{g(\ell x_2+i_2,2y_2+j_2)}
K_{\operatorname{Int}}(\ell x_1+i_1, 2y_1+j_1; \ell x_2+i_2, 2y_2+j_2 )
\\\rightarrow
\sigma^{t_2-t_1}
\frac{\nu(t_2,j_2)}{(2\pi\i)^2}
\int_{|z_1|=R_1}
\int_{\gamma_{\ell}'}
\e^{\frac{\sigma^2}{2}
(z_2^2-z_1^2)}\e^{\mu_1 z_1-\mu_2 z_2}\frac{z_2^{t_2}}{z_1^{t_1}}\frac{\d z_1\d z_2}{z_2-z_1},
\end{multline*}
with contours of integration as in Figure \ref{fig:contours_gue_corners}.
The proof is complete, after making the change of variables $z_k=\sigma^{-1}\zeta_k$, for $k=1,2$ inside the integral.
\end{proof}

\section{The parameters~$\tau $ and~$\sigma^2$}\label{sec:ty_sigma} 
In this section we express~$\tau $ and~$\sigma^2$ in terms of the edge weights. Recall that
\begin{equation}\label{eq:action_function_derivative}
\d G=\left(\frac{\ell}{z-1}+\frac{w'(z)}{w(z)}-\frac{\tau }{z}\right)\d z,
\end{equation}
where $w(z)$ is the map $(z,w)\mapsto w$, and the constant~$\tau $ is defined so that~$\d G(q_\infty)=0$. Let~$G(z)$ be the action function defined in the local coordinates around~$q_\infty$ given by the map $(z,w)\mapsto z$. Then~$\sigma^2$ is defined as~$\sigma^2=G''(1)$. See Lemma~\ref{lem:zerospoles} and Corollary~\ref{lem:action_taylor}.

The goal of this section is to prove the following proposition.
\begin{proposition}\label{prop:exact_values}
Let~$\tau $ and~$\sigma^2$ be as above and set~$a_k=\alpha_k\beta_{k-1}^{-1}$ and~$b_k=\alpha_k^{-1}\beta_k$. Then
\begin{equation}\label{eq:ty_sum}
\tau =
\sum_{k=1}^\ell \frac{1+a_k+a_kb_k+a_kb_ka_{k+1}}
{(1+a_k)(1+b_k)(1+a_{k+1})}
\end{equation}
and
\begin{equation}\label{eq:sigma_sum}
\sigma^2=
\sum_{k=1}^\ell \frac{(1+a_k+a_kb_k+a_kb_ka_{k+1})(b_k+a_{k+1}+b_ka_{k+1}+a_ka_{k+1})}
{(1+a_k)^2(1+b_k)^2(1+a_{k+1})^2}
\end{equation}
where~$a_{\ell+1}=a_1$.
\end{proposition}
Before we prove the proposition, we introduce, for convenience, the polynomial
\begin{equation}
q(z)=(z-1)^\ell \Tr \Phi(z).
\end{equation} 
Recall from Section \ref{sec:kernel} that~$p(z)=q(z)^2-(z-1)^{2\ell}$ and, locally around~$(z,w)=q_\infty$,~$(z-1)^\ell w(z)=q(z)+\frac{1}{2}\sqrt{p(z)}$, where~$p$ is defined in Lemma~\ref{lem:zeros_disc}, and the square root is the principle branch. The following lemma will be useful in the proof of Proposition~\ref{prop:exact_values}.
\begin{lemma}\label{lem:pq}
The following equalities hold:
\begin{enumerate}[(a)]
\item~$p(1)=q(1)^2$, \label{eq:pq_1}
\item~$\frac{p'(1)}{p(1)}=2\frac{q'(1)}{q(1)}$, \label{eq:pq_2}
\item~$\left.\frac{w'(z)}{w(z)}+\frac{\ell}{z-1}\right|_{z=1}=\frac{q'(1)}{q(1)}$, \label{eq:pq_3}
\item~$\left.\frac{\d}{\d z}\left(\frac{w'(z)}{w(z)}+\frac{\ell}{z-1}\right)\right|_{z=1}=\frac{q''(1)}{q(1)}-\left(\frac{q'(1)}{q(1)}\right)^2$. \label{eq:pq_4}
\end{enumerate}
\end{lemma}
\begin{proof}
We prove the equalities in the order given in the statement. 
\begin{enumerate}[(a)]
\item Equality~\eqref{eq:pq_1} follows directly from the equality~$p(z)=q(z)^2-(z-1)^{2\ell}$.
\item Differentiating~$p$ at~$z=1$ give us the equality~$p'(1)=2q(1)q'(1)$, and using~\eqref{eq:pq_1}, this leads to~\eqref{eq:pq_2}.
\item We apply the logarithmic derivative evaluated at~$z=1$ to the equality~$(z-1)^\ell w(z)=q(z)+\frac{1}{2}\sqrt{p(z)}$. It is clear that the left-hand side become the left-hand side of~\eqref{eq:pq_3}, and using~\eqref{eq:pq_1} and~\eqref{eq:pq_2} it is not hard to see that the right-hand side becomes the right-hand side of~\eqref{eq:pq_3}.
\item We now take the derivative of the logarithmic derivative of the same equality as in the proof of~\eqref{eq:pq_3}. The left-hand side becomes the left-hand side of~\eqref{eq:pq_4}. Using~\eqref{eq:pq_3}, as well as~\eqref{eq:pq_1} and~\eqref{eq:pq_2}, we get that the right-hand side is equal to
\begin{footnotesize}
\begin{equation}
\frac{q''(1)+\frac{1}{4}\sqrt{p(1)}\left(\frac{p''(1)}{p(1)}-\frac{1}{2}\left(\frac{p'(1)}{p(1)}\right)^2\right)}{q(1)+\frac{1}{2}\sqrt{p(1)}}-\left(\frac{q'(1)}{q(1)}\right)^2
=\frac{2}{3}\left(\frac{q''(1)}{q(1)}+\frac{1}{4}\frac{p''(1)}{p(1)}-\frac{1}{2}\left(\frac{q'(1)}{q(1)}\right)^2\right)-\left(\frac{q'(1)}{q(1)}\right)^2.
\end{equation}
\end{footnotesize}
Similarly to how we proved~\eqref{eq:pq_2}, we get that~$\frac{p''(1)}{p(1)}=2\left(\frac{q'(1)}{q(1)}\right)^2+2\frac{q''(1)}{q(1)}$. Combining this with the previous equality leads to the right-hand side of~\eqref{eq:pq_4}.
\end{enumerate}
\end{proof}

With the previous lemma together with Lemma~\ref{lem:at_1}, we prove Proposition \ref{prop:exact_values}. 
\begin{proof}[Proof of Proposition~\ref{prop:exact_values}]
For brevity, we set~$\varphi_m(z)=(z-1)\phi_{2m-1}(z)\phi_{2m}(z)$ for~$m=1,\dots,\ell$. Then~$q(z)=\Tr \prod_{m=1}^\ell \varphi(z)$. Note first that
\begin{equation}
\varphi_m'(z)=\varphi_m'=
\begin{pmatrix}
1 & 0 \\
\alpha_m+\beta_m & 1
\end{pmatrix}.
\end{equation}
By the product rule,
\begin{equation}\label{eq:prod_trace_1}
\frac{\d}{\d z}\prod_{m=1}^\ell \varphi_m(z)=\sum_{k=1}^\ell \left(\prod_{m=1}^{k-1} \varphi_m(z)\right) \varphi_k' \left(\prod_{m=k+1}^\ell \varphi_m(z)\right),
\end{equation} 
and
\begin{equation}\label{eq:prod_trace_2}
\frac{\d^2}{\d z^2}\prod_{m=1}^\ell \varphi_m(z)
=
2\sum_{1\leq n<k\leq \ell} \left(\prod_{m=1}^{n-1} \varphi_m(z)\right) \varphi_n'\left(\prod_{m=n+1}^{k-1} \varphi_m(z)\right) \varphi_k' \left(\prod_{m=k+1}^\ell \varphi_m(z)\right). 
\end{equation}

We begin by expressing~$\tau $ in term of the edge weights. By~\eqref{eq:action_function_derivative} and Lemma~\ref{lem:pq}~\eqref{eq:pq_3},~$\d G(q_\infty)=0$ if and only if~$\tau =\frac{q'(1)}{q(1)}$. Lemma~\ref{lem:at_1} implies that~$q(1)=\prod_{m=1}^\ell(1+\alpha_m^{-1}\beta_m)(1+\alpha_{m+1}\beta^{-1})$.  It follows from the same lemma and~\eqref{eq:prod_trace_1} that
\begin{multline}
\frac{q'(1)}{q(1)}=\sum_{k=2}^{\ell-1} \frac{\Tr\left(
\begin{pmatrix}
1 \\
\alpha_1
\end{pmatrix}
\begin{pmatrix}
1 & \beta_{k-1}^{-1}
\end{pmatrix}
\varphi_k'
\begin{pmatrix}
1 \\
\alpha_{k+1}
\end{pmatrix}
\begin{pmatrix}
1 & \beta_{\ell}^{-1}
\end{pmatrix}\right)}
{(1+\alpha_k\beta_{k-1}^{-1})(1+\alpha_k^{-1}\beta_k)(1+\alpha_{k+1}\beta_k^{-1})}\frac{1}{1+\alpha_1\beta_\ell^{-1}} \\
+\frac{\Tr\left(
\begin{pmatrix}
1 \\
\alpha_1
\end{pmatrix}
\begin{pmatrix}
1 & \beta_{\ell-1}^{-1}
\end{pmatrix}
\varphi_\ell'
\right)}
{(1+\alpha_\ell\beta_{\ell-1}^{-1})(1+\alpha_\ell^{-1}\beta_\ell)(1+\alpha_1\beta_\ell^{-1})} 
+\frac{\Tr\left(
\varphi_1'
\begin{pmatrix}
1 \\
\alpha_{2}
\end{pmatrix}
\begin{pmatrix}
1 & \beta_{\ell}^{-1}
\end{pmatrix}\right)}
{(1+\alpha_1\beta_\ell^{-1})(1+\alpha_1^{-1}\beta_1)(1+\alpha_{2}\beta_1^{-1})}.
\end{multline}
Computing the trace of the matrix proves that
\begin{equation}\label{eq:ty_sum_2}
\tau =
\sum_{k=1}^\ell \frac{1+\alpha_{k+1}\beta_{k-1}^{-1}+\beta_{k-1}^{-1}\alpha_k+\beta_{k-1}^{-1}\beta_k}
{(1+\alpha_k\beta_{k-1}^{-1})(1+\alpha_k^{-1}\beta_k)(1+\alpha_{k+1}\beta_k^{-1})},
\end{equation}
which proves~\eqref{eq:ty_sum}. 

To compute~$\sigma^2$, we use~\eqref{eq:action_function_derivative} and Lemma~\ref{lem:pq}~\eqref{eq:pq_4} to see that
\begin{equation}
\sigma^2=\frac{q''(1)}{q(1)}-\left(\frac{q'(1)}{q(1)}\right)^2+\tau =\frac{q''(1)}{q(1)}-\tau ^2+\tau 
\end{equation}
Using a similar computation as we did for~$\tau $, using~\eqref{eq:prod_trace_2} instead of~\eqref{eq:prod_trace_1}, we get
\begin{small}
\begin{equation}
\frac{q''(1)}{q(1)}
=2\sum_{1\leq n<k\leq \ell} \frac{1+\alpha_{n+1}\beta_{n-1}^{-1}+\beta_{n-1}^{-1}(\alpha_n+\beta_n)}
{(1+\alpha_n\beta_{n-1}^{-1})(1+\alpha_n^{-1}\beta_n)(1+\alpha_{n+1}\beta_n^{-1})}
\frac{1+\alpha_{k+1}\beta_{k-1}^{-1}+\beta_{k-1}^{-1}(\alpha_k+\beta_k)}
{(1+\alpha_k\beta_{k-1}^{-1})(1+\alpha_k^{-1}\beta_k)(1+\alpha_{k+1}\beta_k^{-1})}.
\end{equation}
\end{small}
We note that the right-hand side is equal to~$\tau ^2$, using~\eqref{eq:ty_sum_2}, up to the terms along the diagonal~$n=k$:
\begin{equation}
\frac{q''(1)}{q(1)}-\tau ^2=
-\sum_{k=1}^\ell \left(\frac{1+\alpha_{k+1}\beta_{k-1}^{-1}+\beta_{k-1}^{-1}\alpha_k+\beta_{k-1}^{-1}\beta_k}
{(1+\alpha_k\beta_{k-1}^{-1})(1+\alpha_k^{-1}\beta_k)(1+\alpha_{k+1}\beta_k^{-1})}\right)^2.
\end{equation}
Adding~\eqref{eq:ty_sum_2} to the above equality, lead us to~\eqref{eq:sigma_sum}.
\end{proof}

\bibliographystyle{plain}
\bibliography{bibliotek}

\end{document}